\documentclass[reqno, 11pt, a4paper]{amsart}
\usepackage{amsmath, amssymb, amsthm}  
\numberwithin{equation}{section}

\usepackage{appendix}
\usepackage{graphicx}
\usepackage[table]{xcolor}
\usepackage{enumitem}
\usepackage{fancyhdr}
\usepackage{mathrsfs}
\usepackage[truedimen,top=3truecm, bottom=3truecm, left=2.5truecm, right=2.5truecm, includefoot]{geometry}
\usepackage{dsfont} 
\usepackage{pifont}
\usepackage{fontawesome}

\usepackage{mathtools}
\usepackage{mathdots}

\usepackage{tikz}
\usetikzlibrary{shapes,arrows,positioning,fit,backgrounds,calc}
\usepackage{pgfplots}

\usepackage{diagbox}
\usepackage{makecell}

\definecolor{bluelinks}{RGB}{51,0,204}
\definecolor{greenlinks}{RGB}{51,204,102}

\usepackage{aliascnt}
\usepackage[colorlinks=true,linkcolor=bluelinks,citecolor=greenlinks]{hyperref}
\usepackage[capitalise,noabbrev,nameinlink]{cleveref}

\theoremstyle{plain}
\newtheorem{theorem}{Theorem}[section]

\newaliascnt{proposition}{theorem}
\newtheorem{proposition}[proposition]{Proposition}
\aliascntresetthe{proposition}

\newaliascnt{lemma}{theorem}
\newtheorem{lemma}[lemma]{Lemma}
\aliascntresetthe{lemma}

\newaliascnt{corollary}{theorem}
\newtheorem{corollary}[corollary]{Corollary}
\aliascntresetthe{corollary}

\theoremstyle{definition}

\newaliascnt{definition}{theorem}
\newtheorem{definition}[definition]{Definition}
\aliascntresetthe{definition}

\theoremstyle{remark}

\newaliascnt{remark}{theorem}
\newtheorem{remark}[remark]{Remark}
\aliascntresetthe{remark}

\renewcommand{\qedsymbol}{$\blacksquare$}

\crefname{theorem}{Theorem}{Theorems}
\crefname{proposition}{Proposition}{Propositions}
\crefname{lemma}{Lemma}{Lemmas}
\crefname{corollary}{Corollary}{Corollaries}
\crefname{definition}{Definition}{Definitions}
\crefname{remark}{Remark}{Remarks}
\crefname{appendix}{Appendix}{Appendices}

\newcommand{\N}{\mathbb{N}}
\newcommand{\Z}{\mathbb{Z}}
\newcommand{\R}{\mathbb{R}}

\newcommand{\p}{\mathbb{P}}
\newcommand{\E}{\mathbb{E}}
\newcommand{\ind}{\mathds{1}}
\newcommand{\diff}{\mathrm{d}}
\renewcommand{\bar}{\overline}

\newcommand{\tnorm}[1]{\left|\!\left|\!\left|#1\right|\!\right|\!\right|}
\renewcommand{\ge}{\geqslant}
\renewcommand{\le}{\leqslant}

\title[Dynamical and stationary fluctuations for a boundary-driven EPE]{Dynamical and stationary non-equilibrium fluctuations for a boundary-driven exclusion process with energy}

\author[H. Da Cunha]{Hugo Da Cunha}
\address{Hugo Da Cunha -- Université Lyon 1, Centrale Lyon, INSA Lyon, Université Jean Monnet, CNRS, ICJ UMR5208, 69622 Villeurbanne, France}
\email{\href{mailto:dacunha@math.univ-lyon1.fr}{dacunha@math.univ-lyon1.fr}}


\begin{document}
\begin{abstract}
    We consider the exclusion process with energy, a model with mass and energy conservation, and we put it in a boundary-driven setting. We are interested in its hydrodynamic and hydrostatic limits, together with the associated dynamical and stationary fluctuations. We provide a new proof for the derivation of dynamical fluctuations that does not require explicit estimates on the two-point correlation functions. While the hydrodynamic limit is given by two uncoupled heat equations, the dynamical fluctuations are described by a system of coupled generalized Ornstein-Uhlenbeck processes.
\end{abstract}

\maketitle
\section{Introduction}

A challenging problem in non-equilibrium statistical mechanics is to derive macroscopic transport equations and fluctuations directly from microscopic dynamics. Boundary-driven interacting particle systems are a natural framework to study this problem. When the boundary reservoirs impose different densities, a persistent macroscopic current is induced and the system is driven out of equilibrium. Unlike equilibrium systems, the stationary state of non-equilibrium systems generally exhibits correlations over macroscopic distances. These correlations were first identified in \cite{spohn_long_1983} and further investigated in e.g.\cite{derrida_non-equilibrium_2007,bertini_long_2007,derrida_free_2001}.

\medskip

The simplest boundary-driven interacting particle system is the symmetric simple exclusion process (SSEP) with boundaries, extensively studied in the literature \cite{baldasso_exclusion_2017,goncalves_IHP_2019,goncalves_non-equilibrium_2020,goncalves_hydrodynamics_2023,landim_stationary_2008}. In this model, particles perform symmetric random walks on a one-dimensional lattice with the constraint that each site can be occupied by at most one particle. At the two ends of the lattice, particles are added or removed by boundary reservoirs with different densities, and whose interactions are tuned by a factor $N^{-\theta}$ for a parameter $\theta\ge 0$. The hydrodynamic limit and non-equilibrium stationary state (NESS) of this model were investigated in \cite{baldasso_exclusion_2017}. Dynamical and stationary fluctuations for this model were established first in \cite{landim_stationary_2008} for $\theta =0$, then in \cite{franco_non-equilibrium_2019} for $\theta =1$ and later in \cite{goncalves_non-equilibrium_2020} for general $\theta\ge 0$. Each of these works relies heavily on obtaining a suitable bound on the two-point correlation function of the system, which is a challenging task. A classical way to do so is to derive a discrete equation for the correlation function and then estimate its solution with random walk techniques.

\medskip

Most of the work on boundary-driven interacting particle systems has focused on models with a single conserved quantity like the SSEP, typically the particle density. However, many physical systems have multiple conserved quantities that interact with each other. Multi-species exclusion processes are an instance of such systems. We mention the ABC model whose hydrodynamic behaviour in the boundary-driven setting has been studied in \cite{goncalves_hydrodynamics_2023}, for which no fluctuation results have been established yet. Only very recently have fluctuation results been obtained for multi-species boundary-driven exclusion processes. In \cite{chen_non-equilibrium_2026}, the authors investigate a two-species exclusion process with slow boundaries and species conversion. They establish dynamical fluctuations for this model, also by obtaining a bound on the two-point correlation functions which in that case are no longer scalar, but have four components. In this paper, we also consider a two-component boundary-driven exclusion process, but it is of different nature: its second conserved quantity is not a second species of particles, but an energy carried by the particles themselves.

\medskip

The model we consider is the symmetric Exclusion Process with Energy (EPE), recently introduced in \cite{da_cunha_stationary_2026} without boundaries, that we extend here to the boundary-driven setting. In this model, particles evolve according to the SSEP dynamics, but also carry an amount of energy given by an integer between $1$ and $\kappa\ge 2$. When two particles are adjacent, they can exchange energy unit by unit. Reservoirs of particles and energy are attached at the boundaries of the system, and they can inject or remove particles and energy with a speed tuned by a factor~$N^{-\theta}$. The dynamics is designed to ensure that the model is gradient, \textit{i.e.}~that the instantaneous currents of particles and energy can be written in gradient form (a kind of macroscopic Fick's law is satisfied), even at the boundaries. In the present case, the gradient conditions are very simple as the currents are respectively proportional to the discrete gradients of the particle and energy densities. Even though this property makes the model particularly tractable, estimates on the two-point correlation functions are still challenging to obtain as equations for particle-energy and energy-energy correlations do not close. Therefore, we need a new method to derive non-equilibrium and stationary fluctuations for this model.

\medskip

At the macroscopic level and in the diffusive timescale, this system is described by the particle and energy densities $\rho (t,u)$ and $\mathcal{E}(t,u)$. They respectively satisfy the heat equation with diffusion coefficient $\kappa -1$, with boundary conditions of Dirichlet type if $\theta <1$, Robin type if $\theta =1$ and Neumann type if $\theta >1$. Although the hydrodynamic equations are uncoupled, the fluctuation fields of particles and energy are heavily coupled. The hydrodynamic limit is stated in the article, but not proved as it follows the same strategy as the one of the SSEP \cite{baldasso_exclusion_2017,goncalves_IHP_2019}. Our first result (\cref{thm:dynamical_fluctuations}) establishes the convergence of the joint density-energy fluctuation field to a generalized Ornstein-Uhlenbeck process that can be formally written as the solution of the following SPDE
\begin{equation*}
    \partial_t \mathcal{Y}_t = (\kappa -1)\Delta \mathcal{Y}_t + \nabla\big( \sqrt{2(\kappa -1)\chi(\rho_t,\mathcal{E}_t)}\dot{\mathcal{W}}_t\big)
\end{equation*}
where $\chi (\rho ,\mathcal{E})$ is the compressibility matrix of the system, and $\dot{\mathcal(W)}_t$ is a two-dimensional space-time white noise. In the Robin case $\theta =1$, some additional boundary terms arise. This result is obtained without requiring any precise estimate on the two-point correlation functions, and this is one of the main contributions of this paper. Instead, we rely on a variance estimate that we assume at initial time and is propagated along the dynamics (see \ref{eq:assumption_variance_estimate}) thanks to the very simple gradient property of the model. In particular, this relaxes the assumptions on the initial measure imposed in \cite{franco_non-equilibrium_2019,goncalves_non-equilibrium_2020} and allows to consider initial states associated to rougher macroscopic profiles, and where less information about the correlations is available.

\medskip

We also investigate the stationary fluctuations of this model, \emph{i.e.}~the fluctuations of the system started from its NESS. In \cref{thm:stationary_fluctuations}, for $\theta\le 1$, we prove that the stationary fluctuation field converges to a Gaussian field whose covariance decomposes into a local-equilibrium and a non-equilibrium part. At the critical value $\theta =1$, additional boundary terms emerge for the energy-energy correlations. In the case $\theta >1$, the stationary fluctuation decompose into a Gaussian field and a random vector determined by the fluctuations of the total number of particles and energy, see \cref{thm:stationary_fluctuations_neumann}. We also establish the hydrostatic limit of the model, for which the mere variance estimate is sufficient rather than stationary correlation estimates.

\medskip

We emphasize that if we forget about the energy and only consider the particle density, then the model reduces to the SSEP with boundaries. Therefore, our results must be consistent with the ones obtained in \cite{franco_non-equilibrium_2019,goncalves_non-equilibrium_2020} for the SSEP. We say a word about this in \cref{appendix:robin_boundary_matrices}. Moreover, setting $\kappa =2$ and making some linear combinations of the conserved quantities in the EPE model, we recover the symmetric ABC model. In particular, dynamical and stationary fluctuations for the symmetric ABC model with slow boundary, which was not yet available in the literature, can be deduced from our results. 

\medskip

Let us say a word about the proof strategy. For dynamical fluctuations, we first prove tightness of the fluctuation fields relying on the variance estimate mentioned above, and Dynkin's martingale decomposition. Then, we identify the limit points as the solution of the martingale problem associated to the generalized Ornstein-Uhlenbeck process. To do so, we need a local-equilibrium replacement lemma that is uniform over space, including at the boundaries. The proof of the local-equilibrium property relies on spectral gap estimate and equivalence of ensembles already established for the EPE in \cite{da_cunha_stationary_2026}. For stationary fluctuations, we first prove that the NESS of this model satisfies all the required assumptions to apply the dynamical fluctuations result. This allows to obtain a characterization of the stationary fluctuation field, and then we compute its covariance by taking large time limits in the dynamical covariance.

\subsection*{Outline of the paper}

\cref{sec:definition_model} introduces the EPE model, and its generalization to the boundary-driven case. In \cref{sec:main_results}, we state the main results of the paper, namely the hydrodynamic and hydrostatic limits, and the dynamical and stationary fluctuations. We introduce in \cref{sec:dynkin_martingale} the martingales that will be useful throughout the paper. \cref{sec:tightness} is devoted to the propagation of the variance estimate along the dynamics, and to the proof of tightness of the fluctuation fields, while \cref{sec:identification} identifies the limit points of the fluctuation fields as a generalized Ornstein-Uhlenbeck process. Stationary fluctuations are investigated in \cref{sec:stationary_fluctuations}. In \cref{sec:local_equilibrium}, we prove the local-equilibrium property, which is a key ingredient in the proof of dynamical fluctuations. Finally, we conclude with an \cref{sec:appendix_OU} about definition and uniqueness of the generalized Ornstein-Uhlenbeck process, and an \cref{appendix:robin_boundary_matrices} about the Robin boundary matrices that appear in the critical case $\theta =1$, together with some discussion about the consistency with the existing literature on the SSEP.

\subsection*{General notations}

\begin{itemize}
    \item Elements of $\R^n$ are represented by column vectors. For any $\mathbf{x}\in\R^n$, we denote by $\mathbf{x}^\dagger$ its transpose, and by $|\mathbf{x}|$ its Euclidean norm.
    \item If $m$ is a measure and $f\in L^2(m)$, we sometimes denote by $\langle m,f\rangle$ and sometimes by~$m(f)$ the integral of $f$ with respect to $m$. We also denote by $\langle \cdot ,\cdot\rangle_m$ the inner product in $L^2(m)$.
    \item We denote by $\mathcal{D}([0,T],E)$ the Skorokhod space of càdlàg (\textit{i.e.}~ right-continuous with left limits) functions from $[0,T]$ to a topological space $E$, and by $C([0,T],E)$ the space of continuous functions from $[0,T]$ to $E$.
    \item If $\Lambda$ is a finite set, we denote by $|\Lambda|$ its cardinality.
    \item For non-negative sequences $(u_k)_{k\in\N}$ and $(v_k)_{k\in\N}$, we write $v_k=O(u_k)$ if $|v_k|\le Cu_k$ for some constant $C>0$ that is independent of $k$ but may depend on other parameters.
    \item Throughout, we denote by $C$ a generic positive constant that may change from line to line. When necessary, we emphasize the dependence of $C$ on some parameters by writing~$C=C(\cdot)$.
\end{itemize}

\section{Definition of the model}
\label{sec:definition_model}

\subsection{Bulk dynamics}

Consider the lattice $\Lambda_N = \{1,\hdots ,N-1\}$ that we refer to as the \emph{bulk}, and let $\kappa \ge 2$ be an integer. We start by introducing the symmetric Exclusion Process with Energy (EPE). In this model, we put at most one particle per site on $\Lambda_N$, and each particle carries an amount of energy given by an integer between $1$ and $\kappa$. In other words, we consider \emph{configurations}~$\eta =(\eta_x)_{x\in\Lambda_N}$ in the state space $\Omega_N\coloneq \{0,1,\hdots ,\kappa\}^{\Lambda_N}$, where $\eta_x = 0$ means that there is no particle at site~$x$, and $\eta_x = k$ for $k\in\{1,\hdots ,\kappa\}$ means that there is a particle at site~$x$ with energy $k$. For any~$x\in\Lambda_N$, define the variable $\xi_x=\xi_x(\eta )=\ind_{\{\eta_x\ge 1\}}$ that indicates whether there is a particle at site $x$ or not. Then, we let the particles evolve on the bulk according to the Markovian dynamics generated by the operator~$\mathscr{L}_0$ that acts on functions~$f:\Omega_N\longrightarrow\R$ through the formula 
\begin{equation}\label{def:generator_bulk}
    \mathscr{L}_0f(\eta ) = \sum_{x=1}^{N-2} c_{x,x+1}^p(\eta )\big[ f(\eta^{x,x+1})-f(\eta )\big] + \sum_{x=1}^{N-2} \sum_{|y-x|=1} c_{x\to y}^e(\eta ) \big[ f(\eta^{x\to y})-f(\eta )\big]
\end{equation}
for any $\eta\in\Omega_N$. In the above formula, the configuration $\eta^{x,y}$ stands for the configuration obtained from $\eta$ by exchanging the values of $\eta_x$ and $\eta_y$, while $\eta^{x\to y}$ is the configuration obtained from $\eta$ by transferring one unit of energy from site $x$ to site $y$ whenever possible, namely
\begin{equation}\label{eq:transformedconfig}
    \eta_z^{x,y}=\begin{cases}
        \eta_y & \mbox{ if }z=x,\\
        \eta_x & \mbox{ if }z=y,\\
        \eta_z & \mbox{ otherwise},
    \end{cases}
    \qquad\mbox{ and }\qquad \eta_z^{x\to y}=\begin{cases}
        \eta_x-1 & \mbox{ if }z=x,\\
        \eta_y+1 & \mbox{ if }z=y,\\
        \eta_z & \mbox{ otherwise}.
    \end{cases}
\end{equation}
In this dynamics, there are two types of exchanges: the first one is the jump of particles between neighbouring sites with rate $c_{x,x+1}^p(\eta )$, and the second one is the exchange of energy (unit by unit) between neighbouring particles with rate $c_{x,y}^e(\eta )$. We choose these rates to be given by
\begin{subequations}\label{eq:rates}
\begin{equation}\label{eq:rate_particle}
    c_{x,x+1}^p(\eta ) = (\kappa -1)\big[ \xi_x(1-\xi_{x+1}) + (1-\xi_x)\xi_{x+1}\big],
\end{equation}
\begin{equation}
    c_{x\to y}^e(\eta ) = \xi_x\xi_y (\eta_x-1)(\kappa -\eta_y)= (\eta_x-\xi_x)(\kappa\xi_y-\eta_y),
\end{equation}
\end{subequations}
for any $x,y\in\Lambda_N$ and $\eta\in \Omega_N$, where the last identity comes from the fact that $\xi_x\eta_x=\eta_x$. The choice of these rates is made to ensure that the model belongs to the class of \emph{gradient systems}, meaning the instantaneous currents of particles $j_{x,x+1}^p(\eta )$ and of energy $j_{x,x+1}^e(\eta )$ through a bond~$\{x,x+1\}\subset\Lambda_N$, respectively defined by
\begin{equation*}
    j_{x,x+1}^p(\eta ) = c_{x,x+1}^p(\eta )(\xi_x-\xi_{x+1}),
\end{equation*}
\begin{equation*}
    j_{x,x+1}^e(\eta ) = c_{x,x+1}^p(\eta )(\eta_x-\eta_{x+1}) + c_{x\to x+1}^e(\eta )-c_{x+1\to x}^e(\eta ),
\end{equation*}
can be rewritten as the discrete gradients
\begin{equation}\label{eq:gradient_currents_bulk}
    j_{x,x+1}^p(\eta ) = (\kappa -1) (\xi_x-\xi_{x+1})\qquad\mbox{ and }\qquad j_{x,x+1}^e(\eta )=(\kappa -1)(\eta_x-\eta_{x+1}).
\end{equation}
We emphasize that the rate $\kappa -1$ for the jumps of SSEP prescribed in \eqref{eq:rate_particle} is fundamental to ensure the gradient property.

\medskip

Notice that this dynamics preserves both the total number of particles $\sum\xi_x$ and the total energy $\sum\eta_x$. Moreover, they necessarily satisfy the following bounds
\begin{equation*}
    \sum_x\xi_x\le\sum_x\eta_x\le\kappa\sum_x\xi_x.
\end{equation*}
It leads us to define the set of admissible values for the particle and energy densities, namely the triangle
\begin{equation}\label{def:triangle}
    \mathcal{T}_\kappa = \big\{ (\rho ,\mathcal{E}) \; : \; 0\le \rho\le 1\mbox{ and }\rho\le \mathcal{E}\le\kappa\rho\big\},
\end{equation}
and let $\mathring{\mathcal{T}}_\kappa$ be its interior.

\medskip

The generator $\mathscr{L}_0$ admits a family of invariant (reversible) measures, that we define hereafter. For the proof of the reversibility, we refer to \cite[Proposition 3]{da_cunha_stationary_2026}.

\begin{definition}[Invariant measures]\label{defin:invariant_measures}
    Let $(\rho,\mathcal{E})\in\mathcal{T}_\kappa$ corresponding to particle and energy densities, respectively. When $\rho >0$, define the marginal $\nu_{\rho ,\mathcal{E}}$ to be the probability measure on~$\{0,1,\hdots ,\kappa\}$ under which:
    \begin{itemize}
        \item the variable $\xi_x$ follows a Bernoulli distribution of parameter $\rho$,
        \item conditionally on $\xi_x=1$, the variable $\eta_x-1$ follows a Binomial distribution of parameters $\kappa -1$ and
        \begin{equation}\label{def:binomial_parameter}
            \mathfrak{p} = \mathfrak{p}(\rho ,\mathcal{E}) = \frac{\frac{\mathcal{E}}{\rho}-1}{\kappa -1}\in [0,1].
        \end{equation}
    \end{itemize}
    Namely,
    \begin{equation*}
        \nu_{\rho ,\mathcal{E}}(\eta_x=0)=1-\rho \qquad\mbox{ and }\qquad \nu_{\rho ,\mathcal{E}}(\eta_x=\ell ) = \rho {{\kappa -1}\choose{\ell -1}}\mathfrak{p}^{\ell -1}(1-\mathfrak{p})^{\kappa -\ell} \mbox{ for }1\le \ell\le\kappa.
    \end{equation*}
    When $\rho =0$, we define $\nu_{0,0}$ to be the Dirac measure concentrated on $0$. Then, for any $N\ge 1$, define the product measure on $\Omega_N$ given by $\nu_{\rho,\mathcal{E}}^N \coloneq \nu_{\rho ,\mathcal{E}}^{\otimes\Lambda_N}$. One can check that for any $(\rho,\mathcal{E})\in\mathcal{T}_\kappa$, the measure $\nu_{\rho ,\mathcal{E}}^N$ is reversible with respect to the generator $\mathscr{L}_0$.
\end{definition}

The definition of the invariant measures will be useful to define the boundary dynamics, which we introduce in the next section.

\subsection{Boundary dynamics}

At the boundaries of the bulk, we attach reservoirs that can inject or remove particles and energy. For this, we fix parameters $(\rho_\ell ,\mathcal{E}_\ell)\in \mathcal{T}_\kappa$ that stand respectively for the particle and energy densities of the left reservoir (at site $x=0$). Similarly, let $(\rho_r,\mathcal{E}_r)\in\mathcal{T}_\kappa$ stand for the right reservoir (at site $x=N$) densities. We describe the dynamics that we impose on the left boundary, but the one on the right boundary is defined in a similar way with the corresponding parameters. The dynamics is as follows:
\begin{itemize}
    \item If site $x=1$ is occupied by a particle (\textit{i.e.}~$\xi_1=1$), then the reservoir can absorb it at rate $(\kappa -1)(1-\rho_\ell )$. It can also inject one energy unit to it with rate $(\mathcal{E}_\ell -\rho_\ell )(\kappa -\eta_1)$, and absorb one energy unit from it with rate $(\kappa\rho_\ell-\mathcal{E}_\ell )(\eta_1-1)$.
    \item If site $x=1$ is empty (\textit{i.e.}~$\xi_1=0$), then the reservoir can inject a particle on it at rate~$(\kappa -1)\rho_\ell$, and the energy of the injected particle is sampled according to the distribution~$1+\mathrm{Bin}(\kappa -1, \mathfrak{p}_\ell)$ with $\mathfrak{p}_\ell \coloneq\mathfrak{p} (\rho_\ell, \mathcal{E}_\ell)$ defined in \eqref{def:binomial_parameter}.
\end{itemize}

The infinitesimal generator $\mathscr{L}_\ell$ of this dynamics acts on functions $f:\Omega_N\longrightarrow\R$ via
\begin{multline}
    \mathscr{L}_\ell f(\eta )= (\kappa -1)(1-\rho_\ell )\xi_1 \big[ f(\eta^{1;0})-f(\eta )\big] \\
    + (\kappa -1)\rho_\ell (1-\xi_1) \sum_{j=1}^\kappa {{\kappa -1}\choose{j-1}}\mathfrak{p}_\ell^{j-1}(1-\mathfrak{p}_\ell)^{\kappa -j}\big[ f(\eta^{1;j})-f(\eta )\big] \\
    + (\mathcal{E}_\ell-\rho_\ell)(\kappa\xi_1-\eta_1) \big[ f(\eta^{1,+})-f(\eta )\big] + (\kappa\rho_\ell-\mathcal{E}_\ell)(\eta_1-\xi_1)\big[ f(\eta^{1,-})-f(\eta )\big].
\end{multline}
In the above formula, the configuration $\eta^{x;j}$ is the configuration obtained from $\eta$ by replacing the value of $\eta_x$ by $j$, and $\eta^{x,\pm}$ is the configuration obtained from $\eta$ after replacing $\eta_x$ by $\eta_x\pm 1$ whenever possible, namely
\begin{equation*}
    \eta_z^{x;j}=\begin{cases}
        j & \mbox{ if }z=x,\\
        \eta_z & \mbox{ otherwise},
    \end{cases}
    \qquad\mbox{ and }\qquad \eta_z^{x,\pm}=\begin{cases}
        \eta_x\pm 1 & \mbox{ if }z=x,\\
        \eta_z & \mbox{ otherwise}.
    \end{cases}
\end{equation*}
Similarly, the infinitesimal generator $\mathscr{L}_r$ for the right boundary acts on func\-tions~$f:\Omega_N\longrightarrow\R$ via
\begin{multline}
    \mathscr{L}_r f(\eta )= (\kappa -1)(1-\rho_r )\xi_{N-1} \big[ f(\eta^{N-1;0})-f(\eta )\big] \\
    + (\kappa -1)\rho_r (1-\xi_{N-1}) \sum_{j=1}^\kappa {{\kappa -1}\choose{j-1}}\mathfrak{p}_r^{j-1}(1-\mathfrak{p}_r)^{\kappa -j}\big[ f(\eta^{N-1;j})-f(\eta )\big] \\
    + (\mathcal{E}_r-\rho_r)(\kappa\xi_{N-1}-\eta_{N-1}) \big[ f(\eta^{N-1,+})-f(\eta )\big] + (\kappa\rho_r-\mathcal{E}_r)(\eta_{N-1}-\xi_{N-1})\big[ f(\eta^{N-1,-})-f(\eta )\big].
\end{multline}

Finally, we take a parameter $\theta\ge 0$ ruling the speed of interactions between the bulk and the reservoirs, and we define the full generator of the dynamics as
\begin{equation}\label{def:generator_full}
    \mathscr{L}_N \coloneq \mathscr{L}_0 +\frac{1}{N^\theta}\big( \mathscr{L}_\ell + \mathscr{L}_r\big).
\end{equation}
Thanks to the gradient condition \eqref{eq:gradient_currents_bulk}, one can check that for any $x\in\Lambda_N$, we have the local conservation laws
\begin{equation}\label{eq:local_conservation_laws}
    \mathscr{L}_N\xi_x = j_{x-1,x}^p(\eta )-j_{x,x+1}^p(\eta ) \qquad\mbox{ and }\qquad \mathscr{L}_N\eta_x = j_{x-1,x}^e(\eta )-j_{x,x+1}^e(\eta ),
\end{equation}
where the instantaneous currents at the boundaries are given by
\begin{equation}\label{eq:boundary_particle_current}
    j_{0,1}^p (\eta ) = \frac{\kappa -1}{N^\theta} (\rho_\ell -\xi_1) ,\qquad j_{N-1,N}^p (\eta )= \frac{\kappa -1}{N^\theta} (\xi_{N-1}-\rho_r)
\end{equation}
and
\begin{equation}\label{eq:boundary_energy_current}
    j_{0,1}^e (\eta ) = \frac{\kappa -1}{N^\theta} (\mathcal{E}_\ell -\eta_1) ,\qquad j_{N-1,N}^e (\eta )= \frac{\kappa -1}{N^\theta} (\eta_{N-1}-\mathcal{E}_r).
\end{equation}
Set $\bar{\Lambda}_N = \Lambda_N\cup\{0,N\}$. By convention, we extend $\xi$ and $\eta$ to $\bar{\Lambda}_N$ by setting
\begin{equation}\label{eq:conventions_boundary_values}
    \xi_0\equiv\rho_\ell ,\quad \xi_N\equiv\rho_r,\quad \eta_0\equiv\mathcal{E}_\ell\quad\mbox{ and }\quad \eta_N\equiv\mathcal{E}_r.
\end{equation} 

For simplicity, in what follows we will assume that $(\rho_\ell ,\mathcal{E}_\ell),(\rho_r,\mathcal{E}_r)\in \mathring{\mathcal{T}}_\kappa$, so that the boundary dynamics is not degenerate.

\medskip

Fix a probability measure $\mu^N$ on $\Omega_N$, and let $(\eta (t))_{t\ge 0}$ be the Markov process on $\Omega_N$ driven by the diffusively accelerated generator $N^2\mathscr{L}_N$ and starting from $\mu^N$. We denote by $\p_{\mu^N}$ the law that this process generates on the Skorokhod space $\mathcal{D}([0,T],\Omega_N)$ for an arbitrarily fixed time horizon $T>0$. Also denote by $\E_{\mu^N}$ the corresponding expectation. Because of the boundary dynamics, and since $\rho_\ell,\rho_r\neq 0$, the process $(\eta (t))_{t\ge 0}$ is irreducible and therefore admits a unique stationary measure~$\mu_\mathrm{ss}^N$.

\begin{proposition}
    In the equilibrium case where $(\rho_\ell ,\mathcal{E}_\ell)=(\rho_r ,\mathcal{E}_r)=(\rho ,\mathcal{E})$, the stationary measure~$\mu_\mathrm{ss}^N$ coincides with the product measure $\nu_{\rho ,\mathcal{E}}^N$ given in \cref{defin:invariant_measures}. In the non-equilibrium case $(\rho_\ell ,\mathcal{E}_\ell )\neq (\rho_r,\mathcal{E}_r)$, we have no explicit expression for $\mu_\mathrm{ss}^N$, but the density and energy fields under $\mu_\mathrm{ss}^N$, respectively defined by $\rho_\mathrm{ss}^N(x)\coloneq \mu_\mathrm{ss}^N(\xi_x)$ and $\mathcal{E}_\mathrm{ss}^N(x)\coloneq \mu_\mathrm{ss}^N(\eta_x)$ for $x\in\Lambda_N$, write as
    \begin{equation*}
        \rho_\mathrm{ss}^N(x) = \frac{(\rho_r-\rho_\ell)x+(N^\theta -1)(\rho_\ell+\rho_r)+N\rho_\ell}{2N^\theta +N-2},
    \end{equation*}
    \begin{equation}\label{eq:stationary_empirical profiles}
        \mathcal{E}_\mathrm{ss}^N(x) = \frac{(\mathcal{E}_r-\mathcal{E}_\ell)x+(N^\theta -1)(\mathcal{E}_\ell+\mathcal{E}_r)+N\mathcal{E}_\ell}{2N^\theta +N-2}.
    \end{equation}
\end{proposition}

\begin{proof}
    We already know that for any parameters $\rho,\mathcal{E}$, the measure $\nu_{\rho,\mathcal{E}}^N$ is reversible for the generator $\mathscr{L}_0$. Moreover, the generator $\mathscr{L}_\ell$ (resp.~$\mathscr{L}_r$) is built in such a way that the measure~$\nu_{\rho_\ell,\mathcal{E}_\ell}^N$ (resp.~$\nu_{\rho_r,\mathcal{E}_r}^N$) is reversible for it. Therefore, in the equilibrium case $(\rho_\ell ,\mathcal{E}_\ell)=(\rho_r ,\mathcal{E}_r)=(\rho ,\mathcal{E})$, the measure $\nu_{\rho ,\mathcal{E}}^N$ is reversible for the full generator $\mathscr{L}_N$, and hence it coincides with the stationary measure $\mu_\mathrm{ss}^N$. 

    \medskip

    In order to obtain the expressions of the stationary density and energy fields, it suffices to use the stationarity relations $\mu_\mathrm{ss}^N(\mathscr{L}_N\xi_x)=0$ and $\mu_\mathrm{ss}^N(\mathscr{L}_N\eta_x)=0$ for any $x\in\Lambda_N$ and to solve the resulting linear systems. We omit the details of the computations, which are straightforward but tedious.
\end{proof}

\section{Main results}
\label{sec:main_results}

\subsection{Hydrodynamic limit}

We start by introducing some notations. Define the random vector
\begin{equation}\label{eq:Q_x}
    \mathbf{Q}_x(t) = \begin{pmatrix}
        \xi_x(t) \\
        \eta_x(t)
    \end{pmatrix}
\end{equation}
for any $x\in\Lambda_N$ and $t\ge 0$. Let $\rho^\mathrm{ini} : [0,1]\longrightarrow [0,1]$ and $\mathcal{E}^\mathrm{ini} : [0,1]\longrightarrow [0,\kappa]$ be two measurable, Riemann-integrable, profiles satisfying the compatibility bounds $\rho^\mathrm{ini}\le \mathcal{E}^\mathrm{ini}\le\kappa\rho^\mathrm{ini}$. Set
\begin{equation*}
    \mathbf{q}^\mathrm{ini}(u) = \begin{pmatrix}
        \rho^\mathrm{ini}(u) \\
        \mathcal{E}^\mathrm{ini}(u)
    \end{pmatrix} \in \mathcal{T}_\kappa \qquad\mbox{ for any }u\in [0,1],
\end{equation*}
and also
\begin{equation*}
    \mathbf{q}_\ell = \begin{pmatrix}
        \rho_\ell \\
        \mathcal{E}_\ell
    \end{pmatrix}, \qquad \mathbf{q}_r = \begin{pmatrix}
        \rho_r \\
        \mathcal{E}_r
    \end{pmatrix}.
\end{equation*}

The first assumption we make is that the initial distribution $\mu^N$ is \emph{associated} with the profiles~$\rho^\mathrm{ini}$ and $\mathcal{E}^\mathrm{ini}$ in the sense that for any $\delta >0$ and any continuous test function $G:[0,1]\longrightarrow\R^2$, we have
\begin{equation}\label{eq:association_initial_distribution}
   \lim_{N\to +\infty} \mu^N \left( \eta\in\Omega_N \; :\; \bigg| \frac1N\sum_{x\in\Lambda_N} G\left(\frac xN\right)^\dagger \!\mathbf{Q}_x - \int_0^1 G(u)^\dagger \mathbf{q}^\mathrm{ini}(u)\diff u\bigg| >\delta \right) =0. \tag{H1}
\end{equation}
For any $t\ge 0$, define the \emph{empirical measure} $\pi^N_t$ to be the random measure on $[0,1]$ acting on~$\R^2$-valued functions which is given by
\begin{equation}\label{eq:empirical_measure}
    \pi^N_t(\diff u) = \frac1N\sum_{x\in\Lambda_N} \mathbf{Q}_x(t)^\dagger\delta_{x/N}(\diff u),
\end{equation}
where $\delta_v(\diff u)$ is the Dirac measure concentrated at $v\in [0,1]$ which acts on $\R^2$-valued functions. Then, hypothesis \ref{eq:association_initial_distribution} translates into the convergence in probability of the empirical measure~$\pi^N_0$ towards the deterministic measure $\mathbf{q}^\mathrm{ini}(u)\diff u$ as $N\to +\infty$, with respect to the weak topology on the space of measures on~$[0,1]$.

\medskip

We have the following result stating that the hydrodynamic limit of this model is given by a system of two uncoupled heat equations with different boundary conditions, depending on the value of the parameter $\theta$.

\begin{theorem}[Hydrodynamic limit]\label{thm:hydrodynamic_limit}
    Assume \eqref{eq:association_initial_distribution}. Then, for any $t\ge 0$, we have the convergence in probability
    \begin{equation*}
        \pi_t^N(\diff u)\xrightarrow[N\to +\infty]{\p} \mathbf{q}_t(u)^\dagger\diff u,
    \end{equation*}
    where $\mathbf{q}_t(u) = \begin{pmatrix}
        \rho_t(u) \\
        \mathcal{E}_t(u)
    \end{pmatrix}$ is the unique weak solution to the two-component heat equation
    \begin{equation}\label{eq:hydrodynamic_equation}
        \begin{cases}
            \partial_t\mathbf{q}_t(u) = D\partial_u^2\mathbf{q}_t(u) & \mbox{ for } (t,u)\in [0,T]\times [0,1],\\
            \mathbf{q}_0 = \mathbf{q}^\mathrm{ini},
        \end{cases}
    \end{equation}
    with diffusion coefficient $D=\kappa -1$, and with
    \begin{itemize}
        \item if $0\le\theta <1$, Dirichlet boundary conditions
        \begin{equation*}
            \mathbf{q}_t(0) = \mathbf{q}_\ell ,\qquad \mathbf{q}_t(1) = \mathbf{q}_r \quad \mbox{ for any }t\in [0,T] \; ;
        \end{equation*}
        \item if $\theta =1$, Robin boundary conditions
        \begin{equation*}
            \partial_u\mathbf{q}_t(0) = \mathbf{q}_t(0)-\mathbf{q}_\ell ,\qquad \partial_u\mathbf{q}_t(1) = \mathbf{q}_r-\mathbf{q}_t(1) \quad \mbox{ for any }t\in [0,T] \; ;
        \end{equation*}
        \item if $\theta >1$, Neumann boundary conditions
        \begin{equation*}
            \partial_u\mathbf{q}_t(0) = \partial_u\mathbf{q}_t(1) =0 \quad \mbox{ for any }t\in [0,T] .
        \end{equation*}
    \end{itemize}
\end{theorem}

\begin{remark}[Maximum principle and compatibility bounds]\label{remark:maximum_principle}
    As a consequence of the maximum principle, or simply with the explicit formula of the solution to the heat equation as a convolution with the heat kernel, we have that $\mathbf{q}_t(u)\in\mathcal{T}_\kappa$ for any $(t,u)\in [0,T]\times [0,1]$, \textit{i.e.}~the compatibility bounds $\rho_t(u)\le\mathcal{E}_t(u)\le\kappa\rho_t(u)$ are preserved by the hydrodynamic equation \eqref{eq:hydrodynamic_equation}. Moreover, for any $t >0$ and any $u\in (0,1)$, as a consequence of the infinite speed of propagation of the heat equation, we have the inequalities
    \begin{equation*}
       0< \rho_t(u)<1\quad\mbox{ and }\quad \rho_t(u)<\mathcal{E}_t(u)<\kappa\rho_t(u)\qquad i.e.\qquad \mathbf{q}_t(u)\in\mathring{\mathcal{T}}_\kappa 
    \end{equation*}
    In the Neumann case $\theta >1$, this holds provided that  $\int_0^1\mathbf{q}^\mathrm{ini}(u)\diff u$ belongs to $\mathring{\mathcal{T}}_\kappa$.
\end{remark}

We do not provide a proof of the hydrodynamic limit here, as thanks to the gradient relations \eqref{eq:gradient_currents_bulk}, \eqref{eq:boundary_particle_current} and \eqref{eq:boundary_energy_current}, it can be obtained by the exact same arguments as in \cite{baldasso_exclusion_2017,goncalves_IHP_2019}.

\medskip

This results is a law of large numbers for the empirical measure $\pi_t^N$. In the next section, we are interested in the fluctuations of this empirical measure around its hydrodynamic limit, \textit{i.e.}~in a central limit theorem for $\pi_t^N$ both in the dynamical and stationary regimes. We start with the former.

\subsection{Non-equilibrium dynamical fluctuations}

Let us introduce the space of test functions that we use to define the fluctuation fields. It is a space of smooth test functions satisfying Dirichlet, Robin or Neumann type boundary conditions depending on the value of $\theta$, and stabilized by taking second derivative. 

\begin{definition}[Space of test functions]\label{defin:test_functions}
    We define the space $\mathcal{S}_\theta$ to be the set of smooth functions~$G:[0,1]\longrightarrow\R^2$, that satisfy the boundary conditions
    \begin{equation*}
        \begin{cases}
            \partial_u^{2k}G(0) = \partial_u^{2k}G(1) =0 & \mbox{ if } 0\le \theta <1 ,\\
            \partial_u^{2k+1}G(0)=\partial_u^{2k}G(0), \quad \partial_u^{2k+1}G(1)=-\partial_u^{2k}G(1) & \mbox{ if }\theta =1 ,\\
            \partial_u^{2k+1}G(0) = \partial_u^{2k+1}G(1) =0 & \mbox{ if }\theta >1,
        \end{cases}
    \end{equation*}
    for any integer $k\ge 0$. We endow this space with the topology induced by the family of seminorms given by $\tnorm{G}_j = \|\partial_u^jG\|_\infty$ for any $G\in\mathcal{S}_\theta$ and any integer $j\ge 0$. This makes $\mathcal{S}_\theta$ a nuclear Fréchet space\footnote{\textit{i.e.}~a complete Hausdorff space whose topology is induced by a countable family of seminorms, and in which all summable sequences are absolutely summable.}, and we denote by $\mathcal{S}_\theta'$ its topological dual.
\end{definition}

\begin{definition}
    Define the linear operators $\Delta_\theta : \mathcal{S}_\theta\longrightarrow\mathcal{S}_\theta$ and $\nabla_\theta : \mathcal{S}_\theta\longrightarrow C^\infty([0,1],\R^2)$ by
    \begin{equation*}
        \Delta_\theta G (u) = \begin{cases}
            \partial_u^2 G(0^+) & \mbox{ if } u=0,\\
            \partial_u^2 G(u) & \mbox{ if } u\in (0,1),\\
            \partial_u^2 G(1^-) & \mbox{ if } u=1,  
        \end{cases}\qquad \nabla_\theta G (u) = \begin{cases}
            \partial_u G(0^+) & \mbox{ if } u=0,\\
            \partial_u G(u) & \mbox{ if } u\in (0,1),\\
            \partial_u G(1^-) & \mbox{ if } u=1.  
        \end{cases}
    \end{equation*}
\end{definition}

\begin{definition}[Heat semigroup]
    Let $(T_t^\theta)_{t\ge 0}$ be the heat semigroup associated to the equation \eqref{eq:hydrodynamic_equation} with the corresponding boundary conditions depending on the value of $\theta$, in the homogeneous case $\mathbf{q}_\ell = \mathbf{q}_r =0$.
\end{definition}

\noindent We list hereafter some properties of the heat semigroup that will be useful in the sequel. 

\begin{proposition}\label{prop:heat_semigroup_properties}
    The semigroup $(T_t^\theta)_{t\ge 0}$ satisfies the following properties:
    \begin{enumerate}
        \item For any $G\in\mathcal{S}_\theta$ and any $t\ge 0$, we have $T_t^\theta G\in\mathcal{S}_\theta$.
        \item For any $G\in\mathcal{S}_\theta$, we have the convergence
        \begin{equation*}
            T_t^\theta G \xrightarrow[t\to\infty]{} \begin{cases}
                0 & \mbox{ if } 0\le \theta\le 1,\\
                \displaystyle\int_0^1 G(u)\diff u & \mbox{ if }\theta >1,
            \end{cases}
        \end{equation*}
        uniformly on $[0,1]$. Moreover, this convergence is exponentially fast.
        \item When $0\le \theta\le 1$, the operator $\Delta_\theta$ is a bijection from $\mathcal{S}_\theta$ to itself. When $\theta >1$, it is solely a bijection from the subspace of $\mathcal{S}_\theta$ consisting of mean-zero functions, to itself. 
    \end{enumerate}
\end{proposition}

\begin{proof}
    The proof of these results relies on the explicit formula of the semigroup, expanded as a series of sines when $\theta <1$, a series of cosines when $\theta >1$ and a series of sines and cosines when~$\theta =1$, see for instance \cite[Section 2.4.1]{franceschini_non-equilibrium_2024} for these explicit formulas. The proof of these results is properly exposed in \cite[Section 3]{franco_non-equilibrium_2019} in the case $\theta =1$. The general case is very similar so we do not reproduce it here. 
\end{proof}

\medskip

As we said earlier, we are interested in the fluctuations of the empirical measure $\pi_t^N$ around its hydrodynamic limit, that is, for any test function $G$ we want to study the limiting behaviour of the quantity~$\sqrt{N}\big( \langle \pi_t^N,G\rangle - \E_{\mu^N}[\langle \pi_t^N,G\rangle]\big)$. This leads us to define the \emph{fluctuation field} as follows.

\begin{definition}[Fluctuation field]
    For any $t\ge 0$, define the \emph{fluctuation field} $\mathcal{Y}_t^N$ to be the random element of $\mathcal{S}_\theta'$ given by
    \begin{equation}\label{def:fluctuation_field}
        \mathcal{Y}_t^N(G) = \frac{1}{\sqrt{N}}\sum_{x\in\Lambda_N} G\left(\frac xN\right)^\dagger\!\bar{\mathbf{Q}}_x(t)
    \end{equation}
    for any $G\in\mathcal{S}_\theta$. Above, we set
    \begin{equation*}
        \bar{\mathbf{Q}}_x(t) = \mathbf{Q}_x(t)- \mathbf{m}_t^N(x) \quad \mbox{ where }\quad \mathbf{m}_t^N(x) = \E_{\mu^N}[\mathbf{Q}_x(t)].
    \end{equation*}
\end{definition}

Throughout, since the fluctuation field only sees the values of the test functions on the discrete set $\frac1N\Lambda_N$, we often identify a test function $G\in\mathcal{S}_\theta$ with the vector $(G(x/N))_{x\in\Lambda_N}$.

\medskip

We need to impose another assumption on the initial distribution $\mu^N$, which is a variance estimate on the initial distribution. Assume that
\begin{equation}\label{eq:assumption_variance_estimate}
    \exists C_0>0,\quad \forall G\in\mathcal{S}_\theta ,\qquad\E_{\mu^N}\left[ \big| \mathcal{Y}_0^N(G)\big|^2\right] \le C_0 \| G\|_{N}^2\tag{H2}
\end{equation}
where $\| G\|_N^2 = \langle G,G\rangle_N$ is the discrete $\ell^2 (\Lambda_N)$-norm associated to the inner product 
\begin{equation}\label{eq:discrete_inner_product}
    \langle G,H\rangle_N = \frac1N\sum_{x\in\Lambda_N} G\left(\frac xN\right)^\dagger\! H\left(\frac xN\right).
\end{equation}

\begin{remark}
    Note that both assumptions \ref{eq:association_initial_distribution} and \ref{eq:assumption_variance_estimate} are satisfied under the initial product measure 
    \begin{equation}\label{eq:local_gibbs}
        \bigotimes_{x\in\Lambda_N} \nu_{\rho^\mathrm{ini}(x/N),\mathcal{E}^\mathrm{ini}(x/N)}.
    \end{equation}
    In the literature, an assumption that the two-point correlation function of the initial distribution is of order $1/N$ is usually made \cite{franceschini_non-equilibrium_2024,goncalves_non-equilibrium_2020,goncalves_hydrodynamics_2023,landim_stationary_2008}, and this in particular implies \ref{eq:assumption_variance_estimate}. However, we do not need this assumption here, and we start directly from the weaker assumption \ref{eq:assumption_variance_estimate} instead. This is one of the main novelties of this work, which also allows to relax the assumptions made on the regularity of the initial profiles.
\end{remark}

\begin{definition}[Compressibility matrix]\label{defin:compressibility_matrix}
    Define the \emph{compressibility matrix} to be the $2\times 2$ covariance matrix of the random variables $\xi_x$ and $\eta_x$ under the invariant measure $\nu_{\rho ,\mathcal{E}}$ associated parameters~$(\rho ,\mathcal{E})\in\mathcal{T}_\kappa$, namely
    \begin{equation*}
        \chi (\rho ,\mathcal{E}) =\begin{pmatrix}
            \mathrm{Var}_{\nu_{\rho ,\mathcal{E}}}(\xi_x) & \mathrm{Cov}_{\nu_{\rho ,\mathcal{E}}}(\xi_x ; \eta_x)\\
            \mathrm{Cov}_{\nu_{\rho ,\mathcal{E}}}(\xi_x ;\eta_x) & \mathrm{Var}_{\nu_{\rho ,\mathcal{E}}}(\eta_x)
        \end{pmatrix}
    \end{equation*}
    An explicit computation gives that for $(\rho ,\mathcal{E})\in\mathcal{T}_\kappa$, with $\rho\neq 0$, we have
    \begin{equation*}
        \chi (\rho ,\mathcal{E}) = \begin{pmatrix}
            \rho (1-\rho ) & \mathcal{E}(1-\rho ) \\
            \mathcal{E}(1-\rho ) & \frac{(\mathcal{E}-\rho )(\kappa\rho - \mathcal{E})}{(\kappa -1)\rho} + \frac{(1-\rho )\mathcal{E}^2}{\rho}
        \end{pmatrix}
    \end{equation*}
    and when $\rho =0$ this matrix is the zero matrix. Moreover, the matrix $\chi (\rho ,\mathcal{E})$ is positive semidefinite, and it is positive definite if $(\rho ,\mathcal{E})$ belongs to the interior triangle $\mathring{\mathcal{T}}_\kappa$.
\end{definition}

\begin{definition}[Boundary matrices]\label{defin:boundary_matrices}
    Let $m_2(\rho ,\mathcal{E})$ be the second moment of the energy under the invariant measure $\nu_{\rho ,\mathcal{E}}$, namely
    \begin{equation*}
        m_2(\rho ,\mathcal{E}) = \nu_{\rho ,\mathcal{E}}(\eta_x^2) = \frac{\mathcal{E}^2}{\rho} +\mathcal{E}-\rho - \frac{(\mathcal{E}-\rho )^2}{(\kappa -1)\rho}.
    \end{equation*}
    Define the left boundary conductivity matrix $\Sigma_\ell$ by
    \begin{multline}\label{eq:Sigma_left}
        \Sigma_\ell (\rho ,\mathcal{E}) = D\begin{pmatrix}
           (1-\rho_\ell )\rho + (1-\rho )\rho_\ell & (1-\rho_\ell )\mathcal{E} + (1-\rho )\mathcal{E}_\ell\\
           (1-\rho_\ell )\mathcal{E} + (1-\rho )\mathcal{E}_\ell & (1-\rho_\ell )m_2(\rho ,\mathcal{E})+(1-\rho )m_2(\rho_\ell ,\mathcal{E}_\ell)
        \end{pmatrix} \\
        + \big( (\mathcal{E}_\ell -\rho_\ell )(\kappa\rho -\mathcal{E}) + (\mathcal{E}-\rho )(\kappa\rho_\ell -\mathcal{E}_\ell)\big) \begin{pmatrix}
            0 & 0 \\
            0 & 1
           \end{pmatrix}
    \end{multline}
    Define similarly the right boundary matrix $\Sigma_r (\rho ,\mathcal{E})$ by replacing $(\rho_\ell ,\mathcal{E}_\ell)$ by $(\rho_r ,\mathcal{E}_r)$ in \eqref{eq:Sigma_left}. These matrices are positive semi-definite. We note that, under the proper choice of boundary parameters, \emph{Einstein's relation} holds, namely
    \begin{equation*}
        \Sigma_\ell (\rho_\ell ,\mathcal{E}_\ell) = 2D\chi (\rho_\ell ,\mathcal{E}_\ell)\quad\mbox{ and }\quad \Sigma_r (\rho_r ,\mathcal{E}_r) = 2D\chi (\rho_r ,\mathcal{E}_r).
    \end{equation*}
\end{definition}

Our main result about the non-equilibrium dynamical fluctuations is the following.

\begin{theorem}[Non-equilibrium dynamical fluctuations]\label{thm:dynamical_fluctuations}
    Assume \eqref{eq:association_initial_distribution} and \eqref{eq:assumption_variance_estimate}, and further assume that the initial fluctuation field $\mathcal{Y}_0^N$ converges in distribution to a random variable $\mathcal{Y}_0$ in~$\mathcal{S}_\theta '$. Then, the sequence of fluctuation fields $(\mathcal{Y}_t^N)_{t\in [0,T]}$ converges in distribution in~$\mathcal{D}([0,T],\mathcal{S}_\theta ')$ to a limit $(\mathcal{Y}_t)_{t\in [0,T]}$ with continuous trajectories that satisfies
    \begin{equation}\label{eq:decomposition_limit_point}
        \mathcal{Y}_t(G) = \mathcal{Y}_0(T_t^\theta G) + \int_0^t dM_s(T_{t-s}^\theta G)
    \end{equation}
    for any $G\in\mathcal{S}_\theta$ and any $t\in [0,T]$. Above, $(M_t)_{t\in [0,T]}$ is a mean-zero Gaussian martingale with quadratic variation given by
    \begin{equation*}
        \langle M(G)\rangle_t = \int_0^t \|\nabla_\theta G\|_{\theta ,\mathbf{q}_s}^2\diff s
    \end{equation*}
    where for any $t\in [0,T]$ and any functions $F,G$, we set
    \begin{multline}\label{eq:inner_product_theta}
        \langle F,G\rangle_{\theta ,\mathbf{q}_t} = 2D\int_0^1 F(u)^\dagger \chi (\mathbf{q}_t(u))G(u)\diff u \\
        + \ind_{\{\theta =1\}} \Big\{ F(0)^\dagger \Sigma_\ell (\mathbf{q}_t(0))G(0) +  F(1)^\dagger \Sigma_r (\mathbf{q}_t(1))G(1)\Big\}.
    \end{multline}
    Moreover, $(M_t)_{t\in [0,T]}$ is independent of the initial field $\mathcal{Y}_0$. In other words, the limit $(\mathcal{Y}_t)_{t\in [0,T]}$ is the unique (in law) solution to the generalized Ornstein-Uhlenbeck $\mathrm{OU}(\mathcal{S}_\theta ,D\Delta_\theta , \|\nabla\cdot\|_{\theta ,\mathbf{q}_t})$ process (\textit{cf.}~\cref{sec:appendix_OU}), which is the formal solution to the following stochastic partial differential equation
    \begin{equation}
        \partial_t\mathcal{Y}_t = D\Delta_\theta \mathcal{Y}_t + \nabla_\theta\big( \sqrt{2D\chi (\mathbf{q}_t)}\dot{W}_t\big) + \ind_{\{\theta =1\}}\Big(  \delta_0\sqrt{\Sigma_\ell (\mathbf{q}_t(0))}\dot{B}_t^\ell + \delta_1\sqrt{\Sigma_r (\mathbf{q}_t(1))}\dot{B}_t^r\Big)
    \end{equation}
    where $(\dot{W}_t)_{t\in [0,T]}$ is a two-dimensional space-time white noise and $B^\ell, B^r$ are two-dimensional Brownian motions, mutually independent and independent of the initial field $\mathcal{Y}_0$. Above, the square roots of the matrices $\chi$, $\Sigma_\ell$ and $\Sigma_r$  are well-defined as they are positive semidefinite.
\end{theorem}

The proof of this result relies on the so-called \emph{entropy method} developed in \cite{guo1988nonlinear}. It consists in proving the tightness of the sequence of fluctuation fields $(\mathcal{Y}_t^N)_{N\ge 1}$, and then to characterize the limit points as solutions of a martingale problem. We defer the proof of tightness to \cref{sec:tightness}, and the identification of limit points to \cref{sec:identification}. We immediately deduce the following result when the initial distribution is the local Gibbs state defined in \eqref{eq:local_gibbs}.

\begin{corollary}
    Assume that the process starts from the initial Gibbs state defined in \eqref{eq:local_gibbs}. Then, the initial fluctuation field $\mathcal{Y}_0^N$ converges in distribution to a mean-zero Gaussian field $\mathcal{Y}_0$ with covariance given by
    \begin{equation*}
        \E[\mathcal{Y}_0(G)\mathcal{Y}_0(H)] = \int_0^1 G(u)^\dagger \chi (\mathbf{q}^\mathrm{ini}(u))H(u)\diff u
    \end{equation*}
    for any $G,H\in\mathcal{S}_\theta$. In particular, the limit $(\mathcal{Y}_t)_{t\in [0,T]}$ is a mean-zero Gaussian process with covariance given by
    \begin{equation*}
        \E[\mathcal{Y}_t(G)\mathcal{Y}_s(H)] = \int_0^1 T_t^\theta G(u)^\dagger \chi (\mathbf{q}^\mathrm{ini}(u))T_s^\theta H(u)\diff u + \int_0^s \langle \nabla_\theta T_{t-r}^\theta G,\nabla_\theta T_{s-r}^\theta H\rangle_{\theta ,\mathbf{q}_r}\diff r
    \end{equation*}
    for any $G,H\in\mathcal{S}_\theta$ and any $0\le s\le t\le T$.
\end{corollary}

We carry on with the results in the stationary regime.

\subsection{Hydrostatics and non-equilibrium stationary fluctuations}

We conclude this section by stating the results about the hydrostatics and the non-equilibrium stationary fluctuations, \textit{i.e.}~when the initial distribution $\mu^N$ is taken to be the stationary measure $\mu_\mathrm{ss}^N$. We start with the hydrostatic limit, which states that the stationary measure $\mu_\mathrm{ss}^N$ is associated with the stationary solution of the hydrodynamic equation \eqref{eq:hydrodynamic_equation} with corresponding boundary conditions.

\begin{theorem}[Hydrostatic limit]\label{thm:hydrostatic_limit}
    For any $\delta >0$, and any continuous $G:[0,1]\longrightarrow\R^2$, we have
    \begin{equation*}
        \lim_{N\to +\infty} \mu_\mathrm{ss}^N \left( \eta\in\Omega \; :\; \bigg| \frac1N\sum_{x\in\Lambda_N}G\left(\frac xN\right)^\dagger \!\mathbf{Q}_x - \int_0^1 G(u)^\dagger\mathbf{q}_\mathrm{ss}(u)\diff u\bigg| >\delta\right) =0
    \end{equation*}
    where $\mathbf{q}_\mathrm{ss}(u)$ is the stationary solution of \eqref{eq:hydrodynamic_equation} with corresponding boundary conditions. It is given for all $u\in [0,1]$ by
    \begin{equation}\label{eq:stationary_solution}
        \mathbf{q}_\mathrm{ss}(u) = \begin{cases}
            (\mathbf{q}_r-\mathbf{q}_\ell )u + \mathbf{q}_\ell & \mbox{ if } 0\le \theta <1,\\
            (\mathbf{q}_r-\mathbf{q}_\ell )\frac{u+1}{3} + \mathbf{q}_\ell & \mbox{ if } \theta =1,\\
            \frac12(\mathbf{q}_\ell+\mathbf{q}_r ) & \mbox{ if } \theta >1.
        \end{cases}
    \end{equation}
\end{theorem}

Define the stationary fluctuation field to be the random element of $\mathcal{S}_\theta'$ given by
\begin{equation*}
    \mathcal{Y}_\mathrm{ss}^N(G) = \frac{1}{\sqrt{N}}\sum_{x\in\Lambda_N} G\left(\frac xN\right)^\dagger \big( \mathbf{Q}_x(t) - \mathbf{m}_\mathrm{ss}^N(x)\big)\qquad \mbox{ with }\quad \mathbf{m}_\mathrm{ss}^N(x) = \begin{pmatrix}
        \rho_\mathrm{ss}^N(x) \\
        \mathcal{E}_\mathrm{ss}^N(x)
    \end{pmatrix}
\end{equation*}
for any $G\in\mathcal{S}_\theta$, where $\rho_\mathrm{ss}^N$ and $\mathcal{E}_\mathrm{ss}^N$ are the stationary profiles defined in \eqref{eq:stationary_empirical profiles}. Then the stationary fluctuations in the case $0\le \theta\le 1$ can be stated as follows.

\begin{theorem}[Non-equilibrium stationary fluctuations, $\theta\le 1$]\label{thm:stationary_fluctuations}
    Under the stationary measure~$\mu_\mathrm{ss}^N$, the stationary fluctuation field~$\mathcal{Y}_\mathrm{ss}^N$ converges in distribution to a mean-zero Gaussian field $\mathcal{Y}_\mathrm{ss}$ with covariance given by
    \begin{multline}\label{eq:covariance_dirichlet_stationary}
        \E\big[\mathcal{Y}_\mathrm{ss}(F)\mathcal{Y}_\mathrm{ss}(G)\big] = \int_0^1 F(u)^\dagger \chi (\mathbf{q}_\mathrm{ss}(u))G(u)\diff u \\
         + D\int_0^\infty \int_0^1 T_s^\theta F(u)^\dagger \partial_u^2\big(\chi (\mathbf{q}_\mathrm{ss}(u))\big)T_s^\theta G(u)\diff u\diff s \\
        + \ind_{\{\theta =1\}} \int_0^\infty \Big( T_s^\theta F(1)^\dagger \Xi_r T_s^\theta G(1) + T_s^\theta F(0)^\dagger \Xi_\ell T_s^\theta G(0)\Big)\diff s
    \end{multline}
    for any $F,G\in\mathcal{S}_\theta$, where $\Xi_\ell$ and $\Xi_r$ are the boundary matrices defined by
    \begin{align}
        & \Xi_\ell\coloneq \frac{(\kappa -2)(3-2\rho_\ell -\rho_r)(\rho_\ell\mathcal{E}_r-\rho_r\mathcal{E}_\ell)^2}{3\rho_\ell (2\rho_\ell +\rho_r)^2}\begin{pmatrix}
        0 & 0 \\
        0 & 1
        \end{pmatrix},\label{eq:Xil}\\
        & \Xi_r\coloneq \frac{(\kappa -2)(3-\rho_\ell-2\rho_r)(\rho_\ell\mathcal{E}_r-\rho_r\mathcal{E}_\ell)^2}{3\rho_r(\rho_\ell +2\rho_r)^2}\begin{pmatrix}
        0 & 0 \\
        0 & 1
        \end{pmatrix}. \label{eq:Xir}
    \end{align}
    Above, $\mathbf{q}_\mathrm{ss}$ is the stationary solution of \eqref{eq:hydrodynamic_equation} given in \eqref{eq:stationary_solution}.
\end{theorem}

\begin{remark}
    The first term in \eqref{eq:covariance_dirichlet_stationary} corresponds to the local equilibrium fluctuations, while the second term corresponds to the non-equilibrium contributions (this term disappears in the equilibrium case $\rho_\ell=\rho_r$ and~$\mathcal{E}_\ell = \mathcal{E}_r$). In the Robin case $\theta =1$, there is an additional contribution coming from the boundary when $\kappa >2$ for the energy-energy fluctuations 
\end{remark}

\begin{remark}
    Notice that if we forget about the energy, the model reduces to the one-dimensional symmetric simple exclusion process with slow boundary, and diffusion coefficient $D=\kappa -1$. Therefore, the particle fluctuation field -- \textit{i.e.} the restriction of the field to test functions of the form~$G=(g,0)^\dagger$ -- must behave exactly as the fluctuation field of the SSEP with slow boundary. For the Dirichlet case $0\le \theta <1$, the formula we obtained is fully consistent with the results obtained in \cite{goncalves_non-equilibrium_2020}. In the Robin case $\theta =1$, it differs from the one obtained in \cite{franco_non-equilibrium_2019} and subsequently recalled in \cite{goncalves_non-equilibrium_2020} where non-zero boundary contributions are present. As explained in \cref{appendix:robin_boundary_matrices}, this inconsistency originates from a sign error in the proof of \cite[Theorem 2.8]{franco_non-equilibrium_2019}. By flipping the sign therein, we recover the cancellation of the boundary contributions in the Robin case $\theta =1$ for the SSEP, and this is consistent with our results.
\end{remark}

In the case $\theta >1$, the stationary fluctuations result is still a little incomplete, we can only fully characterize its fast modes, \textit{i.e.} its part acting on mean-zero test functions.

\begin{theorem}[Stationary fluctuations, $\theta >1$]\label{thm:stationary_fluctuations_neumann}
    Any function $G\in\mathcal{S}_\theta$ decomposes under the form~$G=\bar{G} + G^\circ$ where $\bar{G} = \int_0^1 G(u)\diff u$ and $G^\circ = G-\bar{G}$ is mean-zero. Then, under the stationary measure~$\mu_\mathrm{ss}^N$, the sequence $(\mathcal{Y}_\mathrm{ss}^N)_{N\ge 1}$ is tight and every limit point $\mathcal{Y}_\mathrm{ss}$ of this sequence satisfies the decomposition
    \begin{equation*}
        \mathcal{Y}_\mathrm{ss}(G) = \bar{G}^\dagger Z + \mathcal{Y}^\circ (G^\circ) 
    \end{equation*}
    where $Z$ is random vector in $\R^2$, and $\mathcal{Y}^\circ$ is a mean-zero Gaussian field on the subspace of mean-zero functions of $\mathcal{S}_\theta$ with covariance given by
    \begin{equation*}
        \E\big[\mathcal{Y}^\circ (F^\circ )\mathcal{Y}^\circ (G^\circ )\big] = \int_0^1 F^\circ (u)^\dagger \chi (\bar{\mathbf{q}})G^\circ (u)\diff u 
    \end{equation*}
    where $\bar{\mathbf{q}}\coloneq \frac12\mathbf{q}_\ell +\frac12\mathbf{q}_r$. Moreover, $Z$ and $\mathcal{Y}^\circ$ are independent.
\end{theorem}

\begin{remark}
    In the Neumann case, there is an additional contribution coming from the mean of the test function, under the form of the vector $Z$, whose law is not fully characterized. Upon proving that the stationary measure satisfies a central limit theorem for the total number of particles and total energy, one could fully characterize this contribution as a Gaussian random variable and deduce a full convergence result rather than simply a characterization of limit points. One way to prove this would be to estimate the two-point correlation function of the stationary measure, and for this, we cannot apply the strategy of \cite{baldasso_exclusion_2017} since the equations for the two-point correlation function are not closed. This is a technical difficulty that we leave for future work.
\end{remark}

\begin{remark}
    Stationary fluctuations for the SSEP with slow boundaries in the regime $\theta >1$ are described in \cite[Theorem 2.9]{goncalves_non-equilibrium_2020}. This result is not consistent with \cref{thm:stationary_fluctuations_neumann}, the reason being that fluctuation for the total number of particles is not accounted for. In particular, the result \cite[Theorem 2.9]{goncalves_non-equilibrium_2020} cannot hold for arbitrary test functions, but only for mean-zero test functions, and the present paper fills this gap.
    
    A complete result, that is a full convergence result to a Gaussian field, would require to identify the total-mass fluctuations. While this should be possible for the SSEP with slow boundaries thanks to the estimates on the two-point correlation function of the stationary measure, this is not possible for the present model as we don't have such estimates.
\end{remark}

We defer the proof of these results to \cref{sec:stationary_fluctuations}. It relies mostly on proving that the stationary measure satisfies the assumptions \ref{eq:association_initial_distribution} and \ref{eq:assumption_variance_estimate}, and then applying the results of \cref{thm:dynamical_fluctuations} to the stationary case together with the heat semigroup properties of \cref{prop:heat_semigroup_properties}.

\section{Dynkin's martingale}
\label{sec:dynkin_martingale}

We start by introducing a class of martingales associated to the fluctuation fields $\mathcal{Y}_t^N$, known as \emph{Dynkin's martingales}. We rely on the following result, whose proof can be found in \cite[Appendix 1.5]{kipnis_scaling_1999}.

\begin{proposition}[Dynkin's formula, \cite{kipnis_scaling_1999}]\label{prop:dynkin_formula}
    For any function $F:\R_+\times\Omega_N\longrightarrow\R$ that is differentiable in the first variable, the process defined by
    \begin{equation*}
        M_t^F = F(t,\eta (t)) - F(0,\eta (0)) - \int_0^t (\partial_s + N^2\mathscr{L}_N)F(s,\eta (s))\diff s
    \end{equation*}
    is a mean-zero martingale with respect to the natural filtration of the process $(\eta (t))_{t\ge 0}$, and its quadratic variation is given by
    \begin{equation*}
        \langle M^F\rangle_t = \int_0^t N^2\big( \mathscr{L}_N F^2(s,\eta (s)) - 2F(s,\eta (s))\mathscr{L}_N F(s,\eta (s))\big)\diff s.
    \end{equation*}
\end{proposition}

Applying this result to the fluctuation field, and after some computations, we obtain that for any $G\in\mathcal{S}_\theta$, the process defined for $t\in [0,T]$ by
\begin{equation}\label{def:dynkin_martingale}
    M_t^N(G) = \mathcal{Y}_t^N(G) - \mathcal{Y}_0^N(G) - \int_0^t \mathcal{Y}_s^N(A_N^\theta G)\diff s
\end{equation}
where $A_N^\theta$ is the operator that acts on functions $f:\Lambda_N\longrightarrow\R$ via 
\begin{equation*}
    A_N^\theta f(x) = \begin{cases}
        DN^2 \big( f(2)-f(1)\big) - DN^{2-\theta} f(1)& \mbox{ if } x=1\\
        DN^2\big( f(x-1)-2f(x) +f(x+1)\big) & \mbox{ if }x\in \{2,\hdots ,N-2\},\\
        DN^2\big( f(N-2)-f(N-1)\big) -DN^{2-\theta}f(N-1) & \mbox{ if } x=N-1.
    \end{cases}
\end{equation*}
Then, in the integral term of formula \eqref{def:dynkin_martingale}, we let $A_N^\theta$ act componentwise, and with a slight abuse of notation, we identify $G\in\mathcal{S}_\theta$ with the vector $\big( G(\frac xN)\big)_{x\in\Lambda_N}$. In matrix form, we can write 
\begin{equation}\label{eq:matrix_form_A_N}
    A_N^\theta = DN^2 \begin{pmatrix}
        -\bigl(1+N^{-\theta}\bigr) & 1 & 0 & \cdots & 0 & 0\\
        1 & -2 & 1 & \ddots & \vdots & \vdots\\
        0 & 1 & -2 & \ddots & 0 & 0\\
        \vdots & \ddots & \ddots & \ddots & 1 & 0\\
        0 & \cdots & 0 & 1 & -2 & 1\\
        0 & \cdots & 0 & 0 & 1 & -\bigl(1+N^{-\theta}\bigr)
    \end{pmatrix}.
\end{equation}
This matrix is clearly symmetric, hence the operator $A_N^\theta$ is self-adjoint with respect to the discrete inner product \eqref{eq:discrete_inner_product}. We introduce the Dirichlet form associated to the operator $A_N^\theta$, namely
\begin{equation*}
    \mathscr{E}_N^\theta (G) \coloneq - \langle G,A_N^\theta G\rangle_N.
\end{equation*}
Thanks to the self-adjointness of $A_N^\theta$, the Dirichlet form reduces to
\begin{equation}\label{eq:dirichlet_form_A_N}
    \mathscr{E}_N^\theta (G) = \frac{D}{N}\sum_{x=1}^{N-2} \big| \nabla_NG\big( \tfrac xN\big)\big|^2 + DN^{1-\theta}\left( \big| G\big(\tfrac 1N\big)\big|^2 + \big| G\big(\tfrac{N-1}{N}\big)\big|^2\right)
\end{equation}
where $\nabla_N$ is the discrete gradient operator defined by $\nabla_NG(\frac xN) = N\big( G(\frac{x+1}{N})-G(\frac xN)\big)$ for any~$x\in\{1,\hdots ,N-2\}$. This quantity is clearly positive, so the operator $A_N^\theta$ is negative definite.

\medskip 

The matrix $A_N^\theta$ is the generator of a continuous-time symmetric random walk on $\Lambda_N$, which jumps at rate $DN^2$ to its nearest neighbours, and is killed at the boundaries at rate $DN^{2-\theta}$. This is a sub-Markovian process (off diagonal coefficients are non-negative, and the sum on each row is non-positive), so if we define the semigroup $P_t^N = e^{tA_N^\theta}$, then the sum of each row of $P_t^N$ is less than or equal to $1$. In particular, it satisfies the $\ell^\infty$-contraction property 
\begin{equation*}
    \| P_t^NG\|_\infty \le \| G\|_\infty \qquad\mbox{ for any }G\in\mathcal{S}_\theta.
\end{equation*}
Moreover, since $A_N^\theta$ is self-adjoint, the semigroup $P_t^N$ is also self-adjoint, and it satisfies the~$\ell^2$-contraction property
\begin{equation}\label{eq:contraction_property_semigroup}
    \| P_t^NG\|_N \le \| G\|_N \qquad\mbox{ for any }G\in\mathcal{S}_\theta.
\end{equation}
Indeed, differentiating the function $t\longmapsto \| P_t^NG\|_N^2$ gives $-2\mathcal{E}_N^\theta (P_t^NG)$. Integrating this identity over $[0,t]$ yields that
\begin{equation*}
    \| P_t^NG\|_N^2 + 2\int_0^t \mathcal{E}_N^\theta (P_s^NG)\diff s = \| G\|_N^2
\end{equation*}
and the bound \eqref{eq:contraction_property_semigroup} readily follows.

\medskip

\cref{prop:dynkin_formula} also gives the expression of the quadratic variation of the martingale~$M_t^N(G)$, which is given by
\begin{equation*}
    \langle M^N(G)\rangle_t = \int_0^t \Gamma^N(\eta (s),G)\diff s \quad\mbox{ with }\quad \Gamma^N(\eta (s),G) = N^2\mathscr{L}_N \mathcal{Y}_s^N(G)^2 - 2N^2\mathcal{Y}_s^N(G)\mathscr{L}_N \mathcal{Y}_s^N(G).
\end{equation*}
We have the following estimate on the quadratic variation.

\begin{proposition}\label{prop:estimate_quadratic_variation}
    There exists a constant $C_1=C_1(\kappa ,\rho_\ell, \mathcal{E}_\ell ,\rho_r,\mathcal{E}_r)>0$ such that for any $G\in\mathcal{S}_\theta$ and any $\eta\in\Omega_N$, we have
    \begin{equation*}
        \Gamma^N(\eta ,G) \le C_1\mathscr{E}_N^\theta (G).
    \end{equation*}
\end{proposition}

\begin{proof}
    After some tedious computations, we obtain that for any $\eta\in\Omega_N$ and any $G\in\mathcal{S}_\theta$, we have
    \begin{multline}\label{eq:expression_GammaN}
        \Gamma^N(\eta ,G) = \frac1N\sum_{x=1}^{N-2} \nabla_NG\big( \tfrac xN\big)^\dagger \Upsilon_x(\eta )\nabla_NG\big( \tfrac xN\big) \\
        + N^{1-\theta} G\big( \tfrac1N\big)^\dagger B_\ell (\eta_1)G\big( \tfrac1N\big) + N^{1-\theta}G\big( \tfrac{N-1}{N}\big)^\dagger B_r (\eta_{N-1})G\big( \tfrac{N-1}{N}\big)
    \end{multline}
    where $\Upsilon_x(\eta )$ is the $2\times 2$ matrix given by
    \begin{multline}\label{eq:matrix_Upsilon}
    \Upsilon_x(\eta )\coloneq c_{x,x+1}^p (\eta )\begin{pmatrix}
      (\xi_x-\xi_{x+1})^2 & (\xi_x-\xi_{x+1})(\eta_x-\eta_{x+1}) \\
      (\xi_x-\xi_{x+1})(\eta_x-\eta_{x+1}) & (\eta_x-\eta_{x+1})^2
   \end{pmatrix} \\ 
   + \big( c_{x\to x+1}^e(\eta )+c_{x+1\to x}^e(\eta )\big) \begin{pmatrix}
      0 & 0 \\
      0 & 1
   \end{pmatrix},
    \end{multline}
    the matrix $B_\ell (\eta_1)$ is given by
    \begin{multline}\label{eq:boundarymatrix_left}
        B_\ell (\eta_1) \coloneq D\begin{pmatrix}
       (1-\rho_\ell )\xi_1 +\rho_\ell(1-\xi_1) &  (1-\rho_\ell )\eta_1 + \mathcal{E}_\ell(1-\xi_1 ) \\
       (1-\rho_\ell)\eta_1 + \mathcal{E}_\ell(1-\xi_1 ) &  (1-\rho_\ell )\eta_1^2 + (1-\xi_1 )m_2(\rho_\ell ,\mathcal{E}_\ell)
   \end{pmatrix} \\ + \big( (\mathcal{E}_\ell -\rho_\ell )(\kappa\xi_1 -\eta_1 )+ (\kappa\rho_\ell -\mathcal{E}_\ell)(\eta_1 -\xi_1)\big) \begin{pmatrix}
      0 & 0 \\
      0 & 1
   \end{pmatrix},
    \end{multline}
    and the matrix $B_r (\eta_{N-1})$ is defined similarly. As all the entries of the matrices $\Upsilon_x$, $B_\ell$ and $B_r$ are bounded, uniformly in $\eta\in\Omega_N$, by constants that depend only on the parameters $\kappa$, $\rho_\ell$,~$\rho_r$,~$\mathcal{E}_\ell$ and $\mathcal{E}_r$, the result follows easily.
\end{proof}

\section{Tightness}
\label{sec:tightness}

We start by proving that the variance estimate \eqref{eq:assumption_variance_estimate} is propagated along the dynamics, which is a crucial step in the proof of tightness.

\begin{proposition}\label{prop:variance_estimate_dynamics}
    There exists a constant $C_2>0$ such that, for any $G\in\mathcal{S}_\theta$, we have
    \begin{equation}
        \sup_{t\ge 0} \E_{\mu^N} \left[ \big| \mathcal{Y}_t^N(G)\big|^2\right] \le C_2 \| G\|_{N}^2.
    \end{equation}
\end{proposition}

\begin{proof}
    Fix $t\ge 0$ and $G\in\mathcal{S}_\theta$. Again, throughout, we use the abuse of notation of identifying a function $G$ with the vector $\big( G(\frac xN), x\in\Lambda_N\big)$ as the fluctuation field actually looks only at the values of $G$ on the discrete set $\frac1N\Lambda_N$. Then, using the time-dependent test function $P_{t-s}^NG$ in Dynkin's formula, we obtain that
    \begin{equation}\label{eq:decomposition_fluctuation_field_dynamics}
        \mathcal{Y}_t^N(G) = \mathcal{Y}_0^N(P_t^NG) + \int_0^t \diff M_s^N(P_{t-s}^NG)
    \end{equation}
    as $P_{t-s}^NG$ solves the backwards equation $\partial_s P_{t-s}^NG + A_N^\theta P_{t-s}^NG =0$. Both terms are orthogonal as the first one is measurable with respect to the initial configuration $\eta (0)$, while the second one is a martingale starting at $0$. Therefore,
    \begin{equation*}
        \E_{\mu^N}\left[ \big| \mathcal{Y}_t^N(G)\big|^2\right] = \E_{\mu^N}\left[ \big| \mathcal{Y}_0^N(P_t^NG)\big|^2\right] + \E_{\mu^N}\left[ \left( \int_0^t \diff M_s^N(P_{t-s}^NG)\right)^2\right].
    \end{equation*}
    By assumption \ref{eq:assumption_variance_estimate}, the first term on the right-hand side is bounded by $C_0\| P_t^NG\|_N^2$, which is in turn bounded by $C_0\| G\|_{N}^2$ thanks to the contraction property \eqref{eq:contraction_property_semigroup}. The remaining expectation can be rewritten as
    \begin{equation*}
        \E_{\mu^N}\left[ \left( \int_0^t \diff M_s^N(P_{t-s}^NG)\right)^2\right] = \E_{\mu^N}\left[ \int_0^t\Gamma^N(\eta (s),P_{t-s}^NG)\diff s\right] \le C_1\int_0^t \mathscr{E}_N^\theta (P_{t-s}^NG)\diff s
    \end{equation*}
    according to \cref{prop:estimate_quadratic_variation}. But notice that
    \begin{equation*}
        \mathscr{E}_N^\theta (P_{t-s}^NG) = - \langle P_{t-s}^NG,A_N^\theta P_{t-s}^NG\rangle_N = \frac12 \frac{\diff}{\diff s} \| P_{t-s}^NG\|_N^2
    \end{equation*}
    so the integral can be computed explicitly as
    \begin{equation*}
        \frac{C_1}{2} \left( \| G\|_N^2 - \| P_t^NG\|_N^2\right) \le \frac{C_1}{2} \| G\|_N^2.
    \end{equation*}
    We proved that for any $t\ge 0$ and any $G\in\mathcal{S}_\theta$, we have
    \begin{equation*}
        \E_{\mu^N}\left[ \big| \mathcal{Y}_t^N(G)\big|^2\right] \le \left( C_0+\frac{C_1}{2}\right)\| G\|_N^2,
    \end{equation*}
    which concludes the proof.
\end{proof}

\medskip

We are now ready to start the proof of tightness of the sequence of fluctuation fields $(\mathcal{Y}^N)$. First, as the space $\mathcal{S}_\theta$ is a nuclear Fréchet space, we can rely on Mitoma's criterion which states that, in the Skorokhod topology, proving tightness for the sequence of $\mathcal{S}_\theta'$-valued processes $(\mathcal{Y}^N)$ is equivalent to proving tightness for the sequence of real-valued processes $(\mathcal{Y}^N(G))$ for any~$G\in\mathcal{S}_\theta$. Then, according to the decomposition \eqref{def:dynkin_martingale}, it is enough to prove tightness of the processes
\begin{equation*}
    \big(\mathcal{Y}_0^N(G)\big)_{t\ge 0} ,\qquad \big( M_t^N(G)\big)_{t\ge 0} \qquad\mbox{ and }\qquad \left( \int_0^t \mathcal{Y}_s^N(A_N^\theta G)\diff s\right)_{t\ge 0}
\end{equation*}
for any $G\in\mathcal{S}_\theta$. This is what we do in the remaining of the section.

\subsection{Initial field}

Thanks to the variance estimate \ref{eq:assumption_variance_estimate} and since $G$ is smooth, the sequence of real-valued random variables $(\mathcal{Y}_0^N(G))_{N\ge 1}$ is bounded in $L^2$ by $C_0\| G\|_\infty^2$. This directly implies that the sequence of constant processes $(\mathcal{Y}_0^N(G))_{t\ge 0}$ is tight.

\subsection{Integral term}

In order to prove tightness of the integral term, we rely on the following result, known as \emph{Kolmogorov-Chentsov criterion}.

\begin{theorem}[Kolmogorov-Chentsov criterion, \cite{chentsov_weak_1956}]
    A sequence of real-valued, continuous stochastic processes $(X_t^N, t\in [0,T])_{N\ge 1}$ is tight with respect to the uniform topology of $C([0,T],\R )$ if the sequence of real-valued random variables $(X_0^N)_{N\ge 1}$ is tight, and there exist positive constants~$K,\gamma_1,\gamma_2$ such that for any $0\le s\le t\le T$, and any $N\ge 1$, we have
    \begin{equation*}
        \E\left[ |X_t^N - X_s^N|^{\gamma_1}\right] \le K (t-s)^{1+\gamma_2}.
    \end{equation*}
\end{theorem}

We start by proving a decomposition of the operator $A_N^\theta$, and then apply the Kolmogorov-Chentsov criterion to each part of the decomposition. Again, we identify a function with the vector of its values on the discrete set $\frac1N\Lambda_N$. In the bulk, the operator $A_N^\theta$ acts as a discrete Laplacian, so we set
\begin{equation*}
    (A_N^\theta G)\big( \tfrac xN\big) = DG''\big(\tfrac xN\big) + R_N\big( \tfrac xN\big) 
\end{equation*}
for any $x\in \{2,\hdots ,N-2\}$, where
\begin{equation*}
    R_N\big( \tfrac xN\big) = D\left( N^2\big( G(\tfrac{x+1}{N})-2G(\tfrac xN)+G(\tfrac{x-1}{N})\big) - G''\big(\tfrac xN\big)\right)
\end{equation*}
At the boundaries, set $R_N(\frac 1N) = R_N(\frac{N-1}{N})=0$. At the left boundary, we can write
\begin{equation*}
    (A_N^\theta G)\big( \tfrac 1N\big) = DN^2 \big( G(\tfrac2N)-G(\tfrac1N)\big) - DN^{2-\theta} G(\tfrac 1N) = DG''(\tfrac1N) + \mathbf{b}_\ell^N(G)
\end{equation*}
where we set
\begin{equation*}
    \mathbf{b}_\ell^N(G) = D\left( N^2\big( G(\tfrac2N) - (1+N^{-\theta})G(\tfrac1N)\big) - G''(\tfrac1N)\right).
\end{equation*}
At the right boundary, we write $(A_N^\theta G)(\tfrac{N-1}{N}) = DG''(\tfrac{N-1}{N})+\mathbf{b}_r^N(G)$ where $\mathbf{b}_r^N(G)$ is defined similarly. Therefore, we have the decomposition
\begin{equation*}
    A_N^\theta G = DG'' + R_N + \mathbf{b}_\ell^N(G)\delta_1 + \mathbf{b}_r^N(G)\delta_{N-1}
\end{equation*}
where $\delta_1(\frac{\cdot}{N})$ and $\delta_{N-1}(\frac{\cdot}{N})$ are respectively the functions $\Lambda_N\longrightarrow\R$ defined by $\ind_{x=1}$ and $\ind_{x=N-1}$. Below, we show that the integral term associated to each part of this decomposition satisfies the hypotheses of the Kolmogorov-Chentsov criterion.

\subsubsection{Bulk part}

Applying Cauchy-Schwarz inequality to the time integral, the variance estimate in \cref{prop:variance_estimate_dynamics}, we have that for any $0\le s\le t\le T$,
\begin{align*}
    \E_{\mu^N}\left[ \left( \int_s^t \mathcal{Y}_r^N(DG'')\diff r\right)^2\right]  \le D^2(t-s)\int_s^t\E_{\mu^N}\left[ \big| \mathcal{Y}_r^N(G'')\big|^2\right]\diff r
    & \le D^2C_2(t-s)^2\| G''\|_N^2\\
    & \le D^2C_2\| G''\|_\infty^2 (t-s)^2.
\end{align*}
Therefore, the integral term associated to the bulk part satisfies the hypotheses of the Kolmogorov-Chentsov criterion, and is thus tight.

\subsubsection{Discretization error term}

By smoothness of the function $G$, a Taylor expansion yields a constant $C>0$, such that for any $x\in\{2,\hdots ,N-2\}$, it holds that
\begin{equation*}
    \big| R_N\big( \tfrac xN\big)\big| \le \frac{C}{N^2}
\end{equation*}
In particular, this implies that
\begin{equation*}
    \|R_N\|_N^2 \le \frac{C^2}{N^4}.
\end{equation*}
Using this, together with Cauchy-Schwarz inequality and \cref{prop:variance_estimate_dynamics}, we have that for any~$0\le s\le t\le T$,
\begin{align*}
    \E_{\mu^N}\left[ \left( \int_s^t \mathcal{Y}_r^N(R_N)\diff r\right)^2\right]  \le (t-s)\int_s^t\E_{\mu^N}\left[ \big| \mathcal{Y}_r^N(R_N)\big|^2\right]\diff r
    & \le C_2(t-s)^2\| R_N\|_N^2\\
    & \le \frac{C_2C^2}{N^4}(t-s)^2.
\end{align*}
Therefore, the integral term associated to the discretization error part satisfies the hypotheses of the Kolmogorov-Chentsov criterion, and is thus tight. Actually, the integral term even vanishes in $L^2$ as $N$ goes to infinity.

\subsubsection{Boundary part}

We now get the more delicate part of the proof, which is to prove tightness of the integral term associated to the boundary part. We only treat the left boundary, as the right boundary is similar. We prove the following.

\begin{lemma}\label{lemma:boundary_part_tightness}
    For any $\gamma\in (0,1)$, there exists a constant $C_\gamma>0$ such that for any $0\le s\le t\le T$, and any deterministic vector $\mathbf{u}\in\R^2$, we have
    \begin{equation}
        \E_{\mu^N}\left[ \left(\int_s^t \mathbf{u}^\dagger \bar{\mathbf{Q}}_1(r)\diff r\right)^2\right] \le C_\gamma |\mathbf{u}|^2 (t-s)^{1+\gamma}N^{2\gamma +\theta -2}.
    \end{equation}
    This estimate remains in force with $\bar{\mathbf{Q}}_1(r)$ replaced by $\bar{\mathbf{Q}}_{N-1}(r)$.
\end{lemma}

\begin{proof}
    Recall the definition \eqref{eq:matrix_form_A_N} of $A_N^\theta$, and let $B_N^\theta = -A_N^\theta$. The matrix $B_N^\theta$ is positive definite, thus invertible and we can define its powers $(B_N^\theta)^\gamma$ for any $\gamma\in\R$ thanks to its spectral decomposition. Recall also that $P_t^N = e^{tA_N^\theta} = e^{-tB_N^\theta}$ is the semigroup associated to $A_N^\theta$. Fix a deterministic function $G :[0,1]\longrightarrow\R^2$, and for any $0\le s\le t\le T$, set
    \begin{equation*}
        I_{s,t}^N(G) = \int_s^t \mathcal{Y}_r^N(G)\diff r,\qquad K_t^NG = \int_0^t P_r^NG\,\diff r.
    \end{equation*}
    As the matrix $B_N^\theta$ is invertible, and by definition of the semigroup, one can easily check that~$K_t^N = (B_N^\theta)^{-1}(I-P_t^N)$ where $I$ is the identity matrix. Fix $t\ge 0$ and consider the time-dependent test function $H_s=K_{t-s}^NG$ for $s\in [0,t]$. Then $H_t=0$ and we have
    \begin{equation*}
        \partial_sH_s = -P_{t-s}^NG,\qquad A_N^\theta H_s = P_{t-s}^NG-G\quad\mbox{ so }\quad \partial_sH_s+A_N^\theta H_s=-G.
    \end{equation*}
    As a consequence, Dynkin's formula, which writes under the form
    \begin{equation*}
        \mathcal{Y}_t^N(H_t)-\mathcal{Y}_s^N(H_s) - \int_s^t \mathcal{Y}_r^N(\partial_rH_r + A_N^\theta H_r)\diff r = \int_s^t \diff M_r^N(H_r),
    \end{equation*}
    reduces to
    \begin{equation*}
        I_{s,t}^N(G) = \mathcal{Y}_s^N(K_{t-s}^NG) + \int_s^t \diff M_r^N(K_{t-r}^NG).
    \end{equation*}
    The first term on the right-hand side is measurable with respect to the filtration at time $s$, and the second term is a future martingale increment, so they are orthogonal in $L^2$. Therefore, we have
    \begin{equation}\label{eq:lemma_decomposition_variance}
        \E_{\mu^N}\left[ \big| I_{s,t}^N(G)\big|^2\right] = \E_{\mu^N}\left[ \big| \mathcal{Y}_s^N(K_{t-s}^NG)\big|^2\right] + \E_{\mu^N}\left[ \left( \int_s^t \diff M_r^N(K_{t-r}^NG)\right)^2\right].
    \end{equation}
    We now estimate each term on the right-hand side separately. By \cref{prop:variance_estimate_dynamics}, the first term on the right-hand side is bounded by $C_2\| K_{t-s}^NG\|_N^2$. In order to estimate this discrete norm, we introduce a basis $(\varphi_k^N)_{1\le k\le N-1}$ of eigenvectors of the matrix $B_N^\theta$, with corresponding positive eigenvalues~$(\lambda_k^N)_{1\le k\le N-1}$ sorted in increasing order. Then, if we decompose the vector $G$ in this basis as 
    \begin{equation*}
        G = \sum_{k=1}^{N-1} g_k \varphi_k^N,
    \end{equation*}
    then the operator $K_t^N$ acts on $G$ as
    \begin{equation*}
        K_t^NG = \sum_{k=1}^{N-1}  \frac{1-e^{-\lambda_k^Nt}}{\lambda_k^N}g_k\varphi_k^N
    \end{equation*}
    which yields 
    \begin{equation*}
        \| K_{t-s}^NG\|_N^2 = \sum_{k=1}^{N-1} \left( \frac{1-e^{-(t-s)\lambda_k^N}}{\lambda_k^N}\right)^2 g_k^2.
    \end{equation*}
    Notice that $(1-e^{-x})^2\le x^{1+\gamma}$ for any $x\ge 0$ and any $\gamma\in (0,1)$, so we have
    \begin{equation*}
        \| K_{t-s}^NG\|_N^2 \le (t-s)^{1+\gamma} \big\langle G,(B_N^\theta)^{\gamma -1}G\big\rangle_N.
    \end{equation*}
    On the other hand, the second term on the right-hand side of \eqref{eq:lemma_decomposition_variance} can be estimated using \cref{prop:estimate_quadratic_variation} as
    \begin{equation*}
        \E_{\mu^N}\left[ \left( \int_s^t \diff M_r^N(K_{t-r}^NG)\right)^2\right] = \E_{\mu^N}\left[ \int_s^t \Gamma^N(\eta (r),K_{t-r}^NG)\diff r\right] \le C_1\int_s^t \mathscr{E}_N^\theta (K_{t-r}^NG)\diff r.
    \end{equation*}
    Set $h=t-s$. By a change of variable, and using again the decomposition in the eigenbasis of~$B_N^\theta$, we have that the integral on the right-hand side is equal to
    \begin{equation*}
        \int_0^h \mathscr{E}_N^\theta (K_r^NG)\diff r = \sum_{k=1}^{N-1} \int_0^h \frac{(1-e^{-r\lambda_k^N})^2}{\lambda_k^N} g_k^2\diff r = \sum_{k=1}^{N-1} \frac{g_k^2}{(\lambda_k^N)^2}\int_0^{h\lambda_k^N} (1-e^{-x})^2\diff x.
    \end{equation*}
    Notice that the quantity $\int_0^y (1-e^{-x})^2\diff x$ is at most of order $y^3$ for $y\le 1$ and at most of order $y$ for $y\ge 1$, so it is of order $y^{1+\gamma}$ for any $\gamma\in (0,1)$. As a consequence, the second term in the right-hand side of \eqref{eq:lemma_decomposition_variance} can also be bounded above by a constant times $(t-s)^{1+\gamma}\langle G,(B_N^\theta)^{\gamma -1}G\rangle_N$. More broadly, we proved that for any $\gamma\in (0,1)$, there exists a constant $C_\gamma>0$ such that for any $0\le s\le t\le T$, and any deterministic function~$G:[0,1]\longrightarrow\R^2$, we have
    \begin{equation*}
        \E_{\mu^N}\left[ \left(\int_s^t\mathcal{Y}_r^N(G)\diff r\right)^2\right] \le C_\gamma (t-s)^{1+\gamma}\big\langle G,(B_N^\theta)^{\gamma -1}G\big\rangle_N.
    \end{equation*}
    In order to obtain \cref{lemma:boundary_part_tightness}, we need to specify the choice of test function. Namely, we take~$G=\mathbf{u}\delta_1$ where $\mathbf{u}\in\R^2$ is a deterministic vector, and we notice that $\mathcal{Y}_r^N(\mathbf{u}\delta_1) = \frac{1}{\sqrt{N}}\mathbf{u}^\dagger \bar{\mathbf{Q}}_1(r)$. Then, we have
    \begin{equation*}
        \E_{\mu^N}\left[ \left(\int_s^t \mathbf{u}^\dagger \bar{\mathbf{Q}}_1(r)\diff r\right)^2\right] \le C_\gamma (t-s)^{1+\gamma}N|\mathbf{u}|^2\big\langle \delta_1,(B_N^\theta)^{\gamma -1}\delta_1\big\rangle_N 
    \end{equation*}
    and we are left to estimate $\langle\delta_1, (B_N^\theta)^{\gamma -1}\delta_1\rangle_N$. Recall that the eigenvalues are sorted in increasing order, so using the spectral decomposition, for any function $f:\Lambda_N\longrightarrow\R$, we have
    \begin{equation*}
        \big\langle f,(B_N^\theta)^{\gamma -1}f\big\rangle_N \le (\lambda_{N-1}^N)^\gamma \big\langle f,(B_N^\theta )^{-1}f\big\rangle_N.
    \end{equation*}
    Moreover, we have the variational formula
    \begin{equation*}
        \big\langle f,(B_N^\theta)^{-1}f\big\rangle_N = \sup_g \big\{ 2\langle f,g\rangle_N - \mathscr{E}_N^\theta (g)\big\}.
    \end{equation*}
    Plugging $f=\delta_1$ inside this formula, and using the fact that $\mathscr{E}_N^\theta (g) \ge DN^{1-\theta}g(1)^2$, we obtain that
    \begin{equation*}
        \big\langle\delta_1 , (B_N^\theta)^{-1}\delta_1\big\rangle_N \le \sup_{a\in\R} \left\{\frac{2a}{N} - DN^{1-\theta}a^2\right\} = \frac{N^{\theta -3}}{D}.
    \end{equation*}
    Moreover, notice that for any $f:\Lambda_N\longrightarrow\R$, we have
    \begin{equation*}
        \mathscr{E}_N^\theta (f) = DN\sum_{x=1}^{N-2}\big( f(x+1)-f(x)\big)^2 + DN^{1-\theta} \big( f(1)^2+f(N-1)^2\big) \le (4DN^2+2DN^{2-\theta})\| f\|_N^2
    \end{equation*}
    using the inequality $(a-b)^2\le 2(a^2+b^2)$. As $\theta\ge 0$, plugging $f=\varphi_{N-1}^N$ inside this inequality, we obtain that $\lambda_{N-1}^N \le 6DN^2$. Therefore, we have proved that
    \begin{equation*}
        \big\langle\delta_1 , (B_N^\theta)^{\gamma -1}\delta_1\big\rangle_N \le (6DN^2)^\gamma \frac{N^{\theta -3}}{D} = \frac{(6D)^\gamma}{D} N^{2\gamma +\theta -3}
    \end{equation*}
    and the result follows.
\end{proof}

We need to apply this lemma when the vector $\mathbf{u}$ is equal to $\mathbf{b}_\ell^N(G)$ or $\mathbf{b}_r^N(G)$, which depends on $N$. Therefore, we need to have an estimate on the growth of the boundary terms $\mathbf{b}_\ell^N(G)$ and~$\mathbf{b}_r^N(G)$. This is the content of the following lemma.

\begin{lemma}\label{lemma:boundary_terms_estimate}
    For any $G\in\mathcal{S}_\theta$, there exists a constant $C_G>0$ such that
    \begin{equation*}
        |\mathbf{b}_\ell^N(G)|+|\mathbf{b}_r^N(G)| \le \begin{cases}
            C_GN & \mbox{ if } 0\le \theta <1,\\
            C_G & \mbox{ if }\theta =1,\\
            C_G(1+N^{2-\theta}) & \mbox{ if }\theta >1.
        \end{cases}
    \end{equation*}
\end{lemma}

\begin{proof}
    For simplicity, denote $h=\frac1N$. Recall that
    \begin{equation*}
        \mathbf{b}_\ell^N(G) = Dh^{-2}\big( G(2h)-G(h)-h^\theta G(h)\big) - DG''(h)
    \end{equation*}
    where $G\in\mathcal{S}_\theta$. Assume first that $0\le\theta <1$. Since $G(0)=0$, a Taylor expansion yields that
    \begin{equation*}
        h^{-2}\big| G(2h)-G(h)\big| \le h^{-1}\| G'\|_\infty \qquad\mbox{ and }\qquad h^{\theta -2} G(h) \le h^{\theta -1}\| G'\|_\infty \le h^{-1}\| G'\|_\infty ,
    \end{equation*}
    so we indeed have that $|\mathbf{b}_\ell^N(G)|\le C_GN$ for some constant $C_G>0$. 

    \medskip

    Assume now that $\theta =1$. A Taylor expansion yields that, as $h$ goes to $0$,
    \begin{align*}
      D^{-1}\mathbf{b}_\ell^N & = h^{-2} \big( G(2h)-G(h)\big) - h^{-1}G(h)  - G''(h) \\
      & = h^{-1}G'(0) + \frac32 G''(0) + O(h) - h^{-1}G(0) - G'(0)+ \frac12 hG''(0) + O(h^2) - G''(0) \\
      & = \frac12 G''(0) - G'(0) + O(h)
   \end{align*}
   as the divergent term vanished thanks to the boundary condition $G'(0)=G(0)$. Thus, $|\mathbf{b}_\ell^N(G)|$ is bounded by a constant $C_G>0$ that depends only on $G$.

   \medskip

   Finally, assume that $\theta >1$. Since $G'(0)=0$, a Taylor expansion yields that, as $h$ goes to~$0$,
   \begin{equation*}
        h^{-2}\big| G(2h)-G(h)\big| \le \frac32 \| G''\|_\infty \qquad\mbox{ and }\qquad h^{\theta -2} \big| G(h)\big| = O(h^{\theta -2}) = O(N^{2-\theta}),
   \end{equation*}
   so we indeed have that $|\mathbf{b}_\ell^N(G)|\le C_G(1+N^{2-\theta})$ for some constant $C_G>0$. The proof for the right boundary is similar.
\end{proof}

\medskip

We are now in position to prove tightness of the integral term associated to the boundary part. We split the proof for different values of the parameter $\theta$.
\begin{itemize}
    \item Assume that $0\le \theta <1$. Applying \cref{lemma:boundary_part_tightness} and \cref{lemma:boundary_terms_estimate}, for any~$0\le s\le t\le T$, we have
    \begin{align*}
        \E_{\mu^N}\left[ \left( \int_s^t \mathcal{Y}_r^N(\mathbf{b}_\ell^N(G)\delta_1)\diff r\right)^2\right] & = \frac1N\E_{\mu^N}\left[ \left( \int_s^t \mathbf{b}_\ell^N(G)^\dagger \bar{\mathbf{Q}}_1(r)\diff r\right)^2\right] \\
        & \le C (t-s)^{1+\gamma} N^{-1}N^2N^{2\gamma +\theta -2} \\
        & = C (t-s)^{1+\gamma} N^{2\gamma +\theta -1}.
    \end{align*}
    Choosing $0<\gamma <\frac{1-\theta}{2}<\frac12$ yields the desired estimate. Therefore, the integral term is tight, and even vanishes in $L^2$ as $N$ goes to infinity.
    \item Assume that $\theta =1$. By a similar computation, we get the bound
    \begin{equation*}
        \E_{\mu^N}\left[ \left(\int_s^t \mathcal{Y}_r^N(\mathbf{b}_\ell^N(G)\delta_1)\diff r\right)^2\right] \le C(t-s)^{1+\gamma}N^{2\gamma -2},
    \end{equation*}
    so choosing any $\gamma\in (0,1)$ gives that the integral term is tight, and even vanishes in $L^2$ as $N$ goes to infinity.
    \item Assume that $1<\theta \le 2$. Then the bound from \cref{lemma:boundary_terms_estimate} becomes $|\mathbf{b}_\ell^N(G) | \le C_GN^{2-\theta}$. Therefore, by \cref{lemma:boundary_part_tightness}, we have
    \begin{align*}
        \E_{\mu^N}\left[ \left(\int_s^t \mathcal{Y}_r^N(\mathbf{b}_\ell^N(G)\delta_1)\diff r\right)^2\right] & \le  C (t-s)^{1+\gamma}N^{-1}N^{4-2\theta}N^{2\gamma +\theta -2}\\
        & \le C (t-s)^{1+\gamma}N^{2\gamma -\theta +1}.
    \end{align*}
    Choosing $0<\gamma < \frac{\theta -1}{2} \le\frac12$ gives the desired bound, and the integral term is tight, and even vanishes in $L^2$ as $N$ goes to infinity.
    \item Assume that $\theta >2$. Then, the bound from \cref{lemma:boundary_terms_estimate} becomes $|\mathbf{b}_\ell^N(G) | \le C_G$. Therefore, using the fact that $\bar{\mathbf{Q}}_1$ is bounded, we directly get that
    \begin{equation*}
        \E_{\mu^N}\left[ \left(\int_s^t \mathcal{Y}_r^N(\mathbf{b}_\ell^N(G)\delta_1)\diff r\right)^2\right] \le CN^{-1}(t-s)^2 
    \end{equation*}
    which implies that the integral term is tight, and even vanishes in $L^2$ as $N$ goes to infinity.
\end{itemize}

\noindent This concludes the proof of tightness of the integral term.

\subsubsection{Martingale term}
\label{sec:martingale_term_tightness}

In order to prove tightness of the martingale term, we simply show that it converges to a Gaussian process relying on the following general result, which can be found in \cite[Theorem VIII.3.12]{jacod_limit_2003}.

\begin{theorem}
    Let $(M_t^N,\; t\in [0,T])_{N\ge 1}$ be a sequence of càdlàg real-valued martingales and denote by $\langle M^N\rangle_t$ its quadratic variation for all $t\ge 0$. Assume that:
   \begin{enumerate}[label=(\roman*)]
      \item The quadratic variation has continuous trajectories almost surely ;
      \item The following limit holds 
      \begin{equation*}
         \lim_{N\to\infty} \E\left[ \sup_{0\le t\le T} |M_t^N-M_{t^-}^N|\right] = 0 \; ;
      \end{equation*}
      \item For any $t\in [0,T]$, the sequence of random variables $\big(\langle M^N\rangle_t\big)_{N\ge 1}$ converges in probability to~$c(t)$ where $c:[0,T]\longrightarrow\R_+$ is a continuous deterministic function.
   \end{enumerate}
   Then, $(M^N)$ converges in distribution in the Skorokhod topology of $\mathcal{D}([0,T],\R)$ to a mean-zero Gaussian process $(M_t)_{t\ge 0}$ which is a martingale with continuous trajectories and quadratic variation $\langle M\rangle_t = c(t)$.
\end{theorem}

Recall expression \eqref{eq:expression_GammaN} of the integrand of the quadratic variation of the martingale $M^N(G)$. As the integrand is bounded, the quadratic variation has continuous trajectories almost surely and the first condition of the theorem is satisfied. The second condition is also satisfied, as the size of the jumps of the martingale is of order $N^{-1/2}$. Indeed, as the integral part of the martingale is continuous, we have
\begin{equation*}
    \E_{\mu^N}\left[ \sup_{0\le t\le T} |M_t^N(G)-M_{t^-}^N(G)|\right] = \E_{\mu^N}\left[ \sup_{0\le t\le T} |\mathcal{Y}_t^N(G)-\mathcal{Y}_{t^-}^N(G)|\right] \le \frac{C}{\sqrt{N}}\| G\|_\infty
\end{equation*}
as a particle jump or energy transfer changes the configuration on at most two sites by at most one unit. Therefore, we only need to check the third condition of the theorem. 

\medskip

For this, we need to introduce further notations. For $x\in\Lambda_N$ and $\ell \ge 1$, we define the box 
\begin{equation}\label{eq:lambda_x_ell}
    \Lambda_x^\ell = \begin{cases}
        \{1,\hdots ,2\ell +1\} & \mbox{ if } x\in \{1,\hdots, \ell\},\\
        \{x-\ell ,\hdots ,x+\ell\} & \mbox{ if } x\in \{\ell +1,\hdots ,N-\ell -1\},\\
        \{N-2\ell -1,\hdots ,N-1\} & \mbox{ if } x\in \{N-\ell ,\hdots ,N-1\}.
    \end{cases}
\end{equation}
This is a box of size $2\ell +1$, included in $\Lambda_N$, that contains $x$ but is not necessarily centered at~$x$ if $x$ is close to the boundary. Then, we define the local averages of the conserved quantities over the box $\Lambda_x^\ell$ as
\begin{equation}\label{eq:local_averages}
    \xi_x^\ell = \frac{1}{2\ell +1}\sum_{y\in\Lambda_x^\ell} \xi_y,\qquad  \eta_x^\ell = \frac{1}{2\ell +1}\sum_{y\in\Lambda_x^\ell} \eta_y\qquad\mbox{ and }\qquad \mathbf{Q}_x^\ell = \begin{pmatrix}
        \xi_x^\ell \\
        \eta_x^\ell
    \end{pmatrix}.
\end{equation}
If $f:\Omega_N\longrightarrow\R$ is a local function\footnote{\textit{i.e.}~a function that depends on a finite, independent of $N$, number of coordinates.}, then we define the function $\Psi_f : \mathcal{T}_\kappa\longrightarrow\R$ to be the expectation under the stationary measure~$\Psi_f(\rho ,\mathcal{E}) = \nu_{\rho ,\mathcal{E}}^N(f)$ for any $(\rho ,\mathcal{E})\in\mathcal{T}_\kappa$. For $x\in\Lambda_N$, let $\tau_x : \Omega_N\longrightarrow\Omega_N$ be the translation operator defined by $(\tau_xf)(\eta ) = f(\tau_x\eta )$ where $\tau_x\eta$ is the configuration defined by $(\tau_x\eta)_y = \eta_{x+y}$ for any $y\in\Lambda_N$.

\medskip

Then, we have the following \emph{local equilibrium property}.

\begin{proposition}[Local equilibrium property]\label{prop:local_equilibrium}
    Let $f:\Omega_N\longrightarrow\R$ be a local function, and let~$\Lambda_f \subset\Lambda_N$ be the set of sites $x$ for which $\tau_xf$ is supported in $\Lambda_N$. For any~$t\in [0,T]$, we have
    \begin{equation*}
        \limsup_{\varepsilon\to 0}\limsup_{N\to\infty} \sup_{x\in\Lambda_f} \E_{\mu^N}\left[ \bigg| \int_0^t  \Big( \tau_xf(\eta (s)) - \Psi_f \big( \xi_x^{\varepsilon N}(s),\eta_x^{\varepsilon N}(s)\big)\Big)\diff s\bigg|\right] = 0.
    \end{equation*}
\end{proposition}

We defer the proof of this proposition to \cref{sec:local_equilibrium}, and we now show how it is used to prove convergence of the quadratic variation of the martingale. Notice that for any $(\rho ,\mathcal{E})\in\mathcal{T}_\kappa$, the expectation of the matrix $\Upsilon_x(\eta )$ defined in \eqref{eq:matrix_Upsilon} under the stationary measure $\nu_{\rho ,\mathcal{E}}^N$ is equal to 
\begin{equation*}
    \Psi_\Upsilon (\rho ,\mathcal{E}) = 2D\chi (\rho ,\mathcal{E})
\end{equation*}
where $\chi$ is the compressibility matrix defined in \eqref{defin:compressibility_matrix}. Therefore, by the local equilibrium property \cref{prop:local_equilibrium}, and replacing the discrete gradients by their continuous counterparts, the bulk term in the expression of the quadratic variation can be rewritten as 
\begin{equation*}
    2D\int_0^t \frac1N\sum_{x=1}^{N-2} \nabla G\big( \tfrac xN\big)^\dagger \chi \big( \xi_x^{\varepsilon N}(s) , \eta_x^{\varepsilon N}(s)\big) \nabla G\big( \tfrac xN\big)\diff s + o_{N,\varepsilon}(1)
\end{equation*}
where $o_{N,\varepsilon}(1)$ is a term that vanishes in probability as $N$ goes to infinity, and then $\varepsilon$ goes to zero. On the other hand, the expectation of the boundary matrices $B_\ell$ and $B_r$ defined in \eqref{eq:boundarymatrix_left} under the stationary measure $\nu_{\rho ,\mathcal{E}}^N$ are equal to
\begin{equation*}
    \Psi_{B_\ell}(\rho ,\mathcal{E}) = \Sigma_\ell (\rho ,\mathcal{E})\qquad\mbox{ and }\qquad \Psi_{B_r}(\rho ,\mathcal{E}) = \Sigma_r (\rho ,\mathcal{E})
\end{equation*}
Observe that when $0\le \theta <1$, the boundary term in the expression of the quadratic variation vanishes as $N$ goes to infinity because the test function $G$ vanishes at the boundary. When~$\theta >1$, the boundary term is of order $O(N^{1-\theta})$ so it also vanishes as $N$ goes to infinity. It survives only when~$\theta =1$, in which case the local equilibrium property \cref{prop:local_equilibrium} allows to replace the boundary term by
\begin{equation*}
    \int_0^t G(0)^\dagger \Sigma_\ell \big( \xi_1^{\varepsilon N}(s),\eta_1^{\varepsilon N}(s)\big) G(0)\diff s + \int_0^t G(1)^\dagger \Sigma_r\big( \xi_{N-1}^{\varepsilon N}(s),\eta_{N-1}^{\varepsilon N}(s)\big) G(1)\diff s + o_{N,\varepsilon}(1)
\end{equation*}
where $o_{N,\varepsilon}(1)$ is a term that vanishes in probability as $N$ goes to infinity, and then $\varepsilon$ goes to zero. Notice that for any $\varepsilon >0$ and any $x\in\{ \varepsilon N+1 \hdots ,(1-\varepsilon )N-1\}$, up to a multiplicative constant that converges to one, we have
\begin{equation*}
    \xi_x^{\varepsilon N}(s) = \pi_s^N * (\iota_\varepsilon \mathbf{e}_1) \big( \tfrac xN\big)\qquad\mbox{ and }\qquad \eta_x^{\varepsilon N}(s) = \pi_s^N * (\iota_\varepsilon \mathbf{e}_2) \big( \tfrac xN\big)
\end{equation*}
where $\iota_\varepsilon \coloneq \frac{1}{2\varepsilon}\ind_{[-\varepsilon ,\varepsilon ]}$ is an approximation of the identity, $*$ is the usual convolution operator and~$\mathbf{e}_1$ and $\mathbf{e}_2$ are the canonical basis vectors of $\R^2$. Similarly, at the left boundary, we have
\begin{equation*}
    \xi_1^{\varepsilon N}(s) = \pi_s^N * (\iota_{2\varepsilon} \mathbf{e}_1) (\tfrac 1N)\qquad\mbox{ and }\qquad \eta_1^{\varepsilon N}(s) = \pi_s^N * (\iota_{2\varepsilon} \mathbf{e}_2) (\tfrac 1N)
\end{equation*}
and the same at the right boundary. Therefore, thanks to the hydrodynamic limit \cref{thm:hydrodynamic_limit}, after letting $N$ go to infinity, and then $\varepsilon$ go to zero, we obtain that the quadratic variation of the martingale converges in probability to
\begin{multline*}
    2D\int_0^t \int_0^1 \nabla G(u)^\dagger \chi \big( \mathbf{q}_s (u)\big) \nabla G(u)\diff u\diff s \\
    +\ind_{\{\theta =1\}}\int_0^t \Big( G(0)^\dagger \Sigma_\ell \big(\mathbf{q}_s(0)\big) G(0) + G(1)^\dagger \Sigma_r \big(\mathbf{q}_s(1)\big) G(1)\Big)\diff s
\end{multline*}
where $\rho_t(u)$ and $\mathcal{E}_t(u)$ are the solutions of the hydrodynamic equations \eqref{eq:hydrodynamic_equation}. Because of the boundary conditions satisfied by the test function $G$ in the case $\theta =1$, this is nothing but
\begin{equation*}
    \int_0^t \| \nabla_\theta G\|_{\theta ,\mathbf{q}_s}^2\diff s
\end{equation*}
defined in \eqref{eq:inner_product_theta}. This shows that the third condition of the theorem is satisfied, and therefore the martingale term converges in distribution to a Gaussian process $(M_t(G))_{t\in [0,T]}$. This limiting process has continuous trajectories and is a mean-zero martingale with quadratic variation~$\int_0^t \| \nabla_\theta G\|_{\theta ,\mathbf{q}_s}^2\diff s$.

\medskip

This concludes the proof of tightness of the martingale term, and therefore of the whole fluctuation field $\mathcal{Y}^N$.

\section{Identification of limit points}
\label{sec:identification}

As the sequence of processes $(\mathcal{Y}^N)_{N\ge 1}$ is tight, it admits limit points in distribution. Consider such a limit point $(\mathcal{Y}_t)_{t\in [0,T]}$. We start by proving that it is a solution to the martingale problem associated to the generalized Ornstein-Uhlenbeck process defined in \cref{defin:OU}. As there is a unique (in law) solution to this martingale problem, there is a unique limit point, and therefore the whole sequence $(\mathcal{Y}^N)_{N\ge 1}$ converges in distribution to this limit point.

\medskip

In \cref{sec:tightness}, we proved a decomposition of the operator $A_N^\theta$ into a sum of a Laplacian, discretization error and boundary terms. Moreover, we proved that the discretization error and boundary terms vanish in the limit, so that
\begin{equation*}
    \int_0^t \mathcal{Y}_s^N(A_N^\theta G)\diff s \xrightarrow[N\to +\infty]{} D\int_0^t \mathcal{Y}_s(\Delta_\theta G)\diff s \quad\mbox{ in distribution.}
\end{equation*}

Therefore, letting $N$ go to infinity in the expression \eqref{def:dynkin_martingale} of the Dynkin martingale, we obtain that
\begin{equation*}
    M_t(G) = \mathcal{Y}_t(G) - \mathcal{Y}_0(G) - D\int_0^t \mathcal{Y}_s(\Delta_\theta G)\diff s
\end{equation*}
defines a mean-zero Gaussian martingale with respect to the natural filtration of $(\mathcal{Y}_t)_{t\in [0,T]}$ and its quadratic variation is given by
\begin{equation*}
    \langle M(G)\rangle_t = \int_0^t \| \nabla_\theta G\|_{\theta ,\mathbf{q}_s}^2\diff s.
\end{equation*}
Notice that since the martingale is Gaussian, starts from $0$ and has deterministic quadratic variation, the independence of the whole martingale from the initial condition $\mathcal{Y}_0$ is guaranteed. We want to apply \cref{prop:uniqueness_OU} to guarantee uniqueness in law of the solution to the martingale problem. We already know that the operator $D\Delta_\theta$ generates a semigroup $(T_t^\theta)_{t\ge 0}$ on $\mathcal{S}_\theta$. We only have to check condition \eqref{eq:condition_uniqueness_OU}, that is
\begin{equation*}
    \int_s^t \| \nabla_\theta T_{t-r}^\theta G\|_{\theta ,\mathbf{q}_r}^2\diff r <\infty
\end{equation*}
for any $0\le s\le t\le T$ and any $G\in\mathcal{S}_\theta$. Recall that we have $\mathbf{q}_t\in\mathring{\mathcal{T}}_\kappa$ for any $t\in (0,T]$, and~$\mathbf{q}_\ell ,\mathbf{q}_r\in\mathring{\mathcal{T}}_\kappa$. Therefore, as the matrices $\chi$, $\Sigma_\ell$ and $\Sigma_r$ depend continuously on their parameters, there exists a constant $C=C(\kappa ,\mathbf{q}_\ell ,\mathbf{q}_r,\mathbf{q}^\mathrm{ini})>0$ such that
\begin{equation*}
    \| \nabla_\theta T_{t-r}^\theta G\|_{\theta ,\mathbf{q}_r}^2 \le C \left(\int_0^1  \big| \nabla_\theta T_{t-r}^\theta G(u)\big|^2\diff u+ \ind_{\{\theta =1\}}\big( \big|T_{t-r}^\theta G(0)\big|^2 + \big|T_{t-r}^\theta G(1)\big|^2\big)\right) .
\end{equation*}
By an integration by parts and thanks to the boundary conditions satisfied by $G$, the quantity on the right-hand side of this inequality is equal to
\begin{equation*}
    -\langle T_{t-r}^\theta G,\Delta_\theta T_{t-r}^\theta G\rangle_{L^2} = \frac{1}{2D} \frac{\diff}{\diff r} \| T_{t-r}^\theta G\|_{L^2}^2
\end{equation*}
hence we have
\begin{equation*}
    \int_s^t \| \nabla_\theta T_{t-r}^\theta G\|_{\theta ,\mathbf{q}_r}^2\diff r \le \frac{C}{2D} \| G\|_{L^2}^2 <\infty .
\end{equation*}
This shows that uniqueness in law of the solution to the martingale problem is satisfied, and the sequence of fluctuation fields $(\mathcal{Y}^N)_{N\ge 1}$ converges in distribution to this unique solution. This proves the second part of \cref{thm:dynamical_fluctuations}. 

\medskip

We now prove the decomposition \eqref{eq:decomposition_limit_point} of the limit point~$\mathcal{Y}_t$. Recall Dynkin's formula in \cref{prop:dynkin_formula}, taking a time-dependent test function $G_\cdot : [0,T]\longrightarrow\mathcal{S}_\theta$ differentiable with respect to time, we have that the process defined by
\begin{equation*}
    \mathbb{M}_t^N(G_\cdot )= \mathcal{Y}_t^N(G_t) - \mathcal{Y}_0^N(G_0) - \int_0^t \mathcal{Y}_s^N(\partial_sG_s +A_N^\theta G_s)\diff s
\end{equation*}
also defines a mean-zero martingale with respect to the natural filtration of $(\mathcal{Y}_t^N)_{t\in [0,T]}$. By repeating the same arguments as in \cref{sec:martingale_term_tightness}, we can show that this martingale converges in distribution, in the Skorokhod topology, to the mean-zero Gaussian martingale
\begin{equation*}
    \left( \int_0^t \diff M_s(G_s)\right)_{t\in [0,T]} \mbox{ whose quadratic variation is given by } \int_0^t \| \nabla_\theta G_s\|_{\theta ,\mathbf{q}_s}^2\diff s.
\end{equation*}
Fix $t\in [0,T]$ and consider this process restricted to the time interval $[0,t]$. If we take the time-dependent test function $F_s = T_{t-s}^\theta G$ where $G\in\mathcal{S}_\theta$ is a fixed test function, then the martingale writes as
\begin{equation*}
    \mathbb{M}_t^N(F_\cdot )= \mathcal{Y}_t^N(G) - \mathcal{Y}_0^N(T_t^\theta G) - \int_0^t \mathcal{Y}_s^N\big( \partial_sF_s+A_N^\theta F_s\big)\diff s.
\end{equation*}
In order to prove the decomposition \eqref{eq:decomposition_limit_point}, we only need to show that the integral term vanishes in the limit. Decompose the integrand as
\begin{equation*}
    \mathcal{Y}_s^N( \partial_sF_s +A_N^\theta F_s)  = \mathcal{Y}_s^N(\partial_sT_{t-s}^\theta G + D\Delta_\theta T_{t-s}^\theta G)+ \mathcal{Y}_s^N (A_N^\theta T_{t-s}^\theta G - D\Delta_\theta T_{t-s}^\theta G).
\end{equation*}
By definition of the semigroup $(T_t^\theta)_{t\ge 0}$, the first term on the right-hand side of this equality is identically zero. Then, by repeating the same arguments as in \cref{sec:tightness}, we can show that 
\begin{equation*}
    \E_{\mu^N}\left[ \left( \int_0^t \mathcal{Y}_s^N (A_N^\theta T_{t-s}^\theta G - D\Delta_\theta T_{t-s}^\theta G)\diff s\right)^2\right] \xrightarrow[N\to +\infty]{} 0
\end{equation*}
which shows that the integral term vanishes in the limit. This proves the decomposition \eqref{eq:decomposition_limit_point} of the limit point $\mathcal{Y}_t$, and concludes the proof of \cref{thm:dynamical_fluctuations}. \hfill\qedsymbol

\section{The stationary picture}
\label{sec:stationary_fluctuations}

\subsection{Proof of the hydrostatic limit}

Let us first prove the hydrostatic limit \cref{thm:hydrostatic_limit}. Let~$\delta >0$ and $G: [0,1]\longrightarrow\R^2$ be a continuous test function. Applying successively Markov's inequality and the inequality $(a+b)^2\le 2(a^2+b^2)$, we have
\begin{multline*}
    \mu_\mathrm{ss}^N\left( \bigg|\frac1N\sum_{x\in\Lambda_N}G\Big( \frac xN\Big)^\dagger\mathbf{Q}_x - \int_0^1 G(u)^\dagger\mathbf{q}_\mathrm{ss}(u)\diff u\bigg| > \delta\right) \\
    \le \frac{2}{\delta^2}\E_{\mu_\mathrm{ss}^N}\left[ \bigg( \frac1N\sum_{x\in\Lambda_N} G\Big( \frac xN\Big)^\dagger \big( \mathbf{Q}_x-\mathbf{m}_\mathrm{ss}^N(x)\big)\bigg)^2\right] \\+ \frac{2}{\delta^2}\bigg(\frac1N\sum_{x\in\Lambda_N} G\Big( \frac xN\Big)^\dagger \!\mathbf{m}_\mathrm{ss}^N(x) - \int_0^1 G(u)^\dagger\mathbf{q}_\mathrm{ss}(u)\diff u\bigg)^2.
\end{multline*}
Using the expression \eqref{eq:stationary_empirical profiles} of the empirical stationary profile $\mathrm{m}_\mathrm{ss}^N$ and the expression \eqref{eq:stationary_solution} of the stationary solution $\mathbf{q}_\mathrm{ss}$, one can check that
\begin{equation*}
    \lim_{N\to\infty} \sup_{x\in\Lambda_N} \big| \mathbf{m}_\mathrm{ss}^N(x) - \mathbf{q}_\mathrm{ss}\big( \tfrac xN\big)\big| = 0.
\end{equation*}
Approximating the integral by a Riemann sum we get that the second term on the right-hand side of the previous inequality vanishes as $N$ goes to infinity.

\medskip

Let us now prove that the expectation on the right-hand side of the previous inequality also vanishes as $N$ goes to infinity. For this, it is enough to prove that the stationary fluctuation field satisfies the variance estimate
\begin{equation}\label{eq:stationary_variance_estimate}
    \forall G\in\mathcal{S}_\theta ,\qquad \E_{\mu_\mathrm{ss}^N}\left[ \big|\mathcal{Y}_\mathrm{ss}^N(G)\big|^2\right] \le C \| G\|_N^2
\end{equation}
for some constant $C>0$. Indeed, the expectation we want to estimate is in fact equal to
\begin{equation*}
    \E_{\mu_\mathrm{ss}^N} \left[ \left( \frac{1}{\sqrt{N}}\mathcal{Y}_\mathrm{ss}^N(G)\right)^2 \right] \le \frac{C}{N}\| G\|_N^2 \le \frac{C}{N}\| G\|_\infty^2.
\end{equation*}

Then, let us prove the stationary variance estimate \eqref{eq:stationary_variance_estimate}, which uses the same techniques as the proof of \cref{prop:variance_estimate_dynamics}. Recall the decomposition \eqref{eq:decomposition_fluctuation_field_dynamics} of the fluctuation field~$\mathcal{Y}_t^N$ into an $\eta (0)$-measurable term, and a future martingale increment term (hence both terms are orthogonal). In particular, since under the stationary measure $\mu_\mathrm{ss}^N$, the law of $\mathcal{Y}_\mathrm{ss}^N$ is invariant under the dynamics, we get the identity
\begin{equation}
    \E_{\mu_\mathrm{ss}^N}\left[ \big|\mathcal{Y}_\mathrm{ss}^N(G)\big|^2\right] = \E_{\mu_\mathrm{ss}^N}\left[ \big|\mathcal{Y}_\mathrm{ss}^N(P_t^NG)\big|^2\right] + \E_{\mu_\mathrm{ss}^N}\left[ \left( \int_0^t \diff M_s^N(P_{t-s}^NG)\right)^2\right]
\end{equation}
valid for any $t\ge 0$, where we recall that $(P_t^N)_{t\ge 0}$ is the semigroup generated by the operator~$A_N^\theta$ defined in \eqref{eq:matrix_form_A_N}. We claim that, for fixed $N$, and for any $G\in\mathcal{S}_\theta$, we have that
\begin{equation*}
    P_t^NG \xrightarrow[t\to\infty]{} 0 \quad\mbox{ in }\ell^2(\Lambda_N).
\end{equation*}
Indeed, as the operator $A_N^\theta$ is negative definite, we have $\|P_t^NG\|_N^2 \le e^{-\lambda_1^Nt}\| G\|_N^2$ where~$\lambda_1^N>0$ is the smallest eigenvalue of $-A_N^\theta$. Therefore, the first term on the right-hand side of the previous identity vanishes as $t$ goes to infinity. On the other hand, we have again the estimate
\begin{equation*}
        \E_{\mu_\mathrm{ss}^N}\left[ \left( \int_0^t \diff M_s^N(P_{t-s}^NG)\right)^2\right] = \E_{\mu_\mathrm{ss}^N}\left[ \int_0^t\Gamma^N(\eta (s),P_{t-s}^NG)\diff s\right] \le C_1\int_0^t \mathscr{E}_N^\theta (P_{t-s}^NG)\diff s
\end{equation*}
thanks to \cref{prop:estimate_quadratic_variation}. Then, letting $t$ go to infinity, we get that
\begin{equation*}
    \E_{\mu_\mathrm{ss}^N}\left[ \big|\mathcal{Y}_\mathrm{ss}^N(G)\big|^2\right] \le C_1\int_0^\infty \mathscr{E}_N^\theta (P_{s}^NG)\diff s = \frac{C_1}{2}\left( \| G\|_N^2 - \lim_{t\to\infty}\| P_t^NG\|_N^2\right) = \frac{C_1}{2}\| G\|_N^2.
\end{equation*}
This proves the stationary variance estimate \eqref{eq:stationary_variance_estimate}, and therefore concludes the proof of the hydrostatic limit \cref{thm:hydrostatic_limit}. \hfill\qedsymbol

\subsection{Proof of the stationary fluctuations}

We now turn to the proof of the stationary fluctuations \cref{thm:stationary_fluctuations,thm:stationary_fluctuations_neumann}. Notice that the hydrostatic limit \cref{thm:hydrostatic_limit} implies that the stationary measure satisfies assumption \ref{eq:association_initial_distribution}. Moreover, in the proof of the hydrostatic limit, we have shown that the stationary variance estimate \eqref{eq:stationary_variance_estimate} holds, thus the stationary measure also satisfies assumption \ref{eq:assumption_variance_estimate}. Then, we can apply \cref{thm:dynamical_fluctuations} to get that the stationary fluctuation field converges in distribution, along a subsequence, to a process $\mathcal{Y}_\mathrm{ss}$ that satisfies the identity
\begin{equation}\label{eq:decomposition_limit_stationary_field}
    \mathcal{Y}_\mathrm{ss}(G) = \mathcal{Y}_0(T_t^\theta G) + \int_0^t \diff M_s(T_{t-s}^\theta G)\quad \mbox{ in distribution}
\end{equation}
for any $t\ge 0$ and any $G\in\mathcal{S}_\theta$. Above, $\mathcal{Y}_0$ is distributed like $\mathcal{Y}_{ss}$ and the second term on the right-hand side of this equality is a mean-zero Gaussian variable of variance
\begin{equation*}
    \int_0^t \| \nabla_\theta T_{t-s}^\theta G\|_{\theta ,\mathbf{q}_\mathrm{ss}}^2\diff s
\end{equation*}
which is independent of $\mathcal{Y}_0$. Thanks to \eqref{eq:decomposition_limit_stationary_field}, we have a decomposition for the variance of the stationary fluctuation field under the form
\begin{equation}
    \E\left[ \big|\mathcal{Y}_\mathrm{ss}(G)\big|^2\right] = \E\left[ \big|\mathcal{Y}_\mathrm{ss}(T_t^\theta G)\big|^2\right] + \int_0^t \| \nabla_\theta T_{t-s}^\theta G\|_{\theta ,\mathbf{q}_\mathrm{ss}}^2\diff s.
\end{equation}
We are going to examine the asymptotic behaviour of this identity as $t$ goes to infinity, and show that we recover the formulas given in \cref{thm:stationary_fluctuations}. 

\subsubsection{Dirichlet case}

Assume $0\le\theta <1$. We start by proving that the first term on the right-hand side of the previous identity vanishes as $t$ goes to infinity, so that all the variance of the stationary fluctuation field is given by the martingale part.

\medskip

Fix $R>0$. Thanks to the convergence in distribution and to the variance estimate \eqref{eq:stationary_variance_estimate}, we have that
\begin{equation*}
    \E \left[ \big|\mathcal{Y}_\mathrm{ss}(G)\big|^2\wedge R\right] = \lim_{N\to\infty} \E_{\mu_\mathrm{ss}^N}\left[ \big|\mathcal{Y}_\mathrm{ss}^N(G)\big|^2\wedge R\right]\le C \lim_{N\to\infty} \| G\|_N^2 = C\| G\|_{L^2}^2.
\end{equation*}
Letting $R$ go to infinity and by the monotone convergence theorem, we get that the stationary variance estimate passes to the limit \textit{i.e.}
\begin{equation*}
    \E \left[ \big|\mathcal{Y}_\mathrm{ss}(G)\big|^2\right] \le  C\| G\|_{L^2}^2
\end{equation*}
for any $G\in\mathcal{S}_\theta$. In particular, we have that
\begin{equation*}
    \E \left[ \big|\mathcal{Y}_\mathrm{ss}(T_t^\theta G)\big|^2\right] \le  C\| T_t^\theta G\|_{L^2}^2 
\end{equation*}
and the right-hand side converges exponentially fast to zero as $t$ goes to infinity thanks to \cref{prop:heat_semigroup_properties}. At this stage, we get that $\mathcal{Y}_\mathrm{ss}(G)$ is a mean-zero Gaussian variable whose variance is given by the asymptotic variance of the martingale part. We want to compute the asymptotic behaviour as $t$ goes to infinity of the quantity
\begin{equation}\label{eq:variance_tostudy_dirichlet}
    2D\int_0^t \int_0^1 \nabla_\theta T_s^\theta G(u)^\dagger \chi(\mathbf{q}_\mathrm{ss}(u))\nabla_\theta T_s^\theta G(u)\diff u\diff s.
\end{equation}
For simplicity, write $F_t=T_t^\theta G$, and let 
\begin{equation}\label{eq:def_phi}
    \phi (t) \coloneq \int_0^1 F_t(u)^\dagger \chi(\mathbf{q}_\mathrm{ss}(u)) F_t(u)\diff u.
\end{equation}
Differentiating $\phi$ with respect to time, and performing an integration by parts recalling that $F_t$ vanishes at the boundaries, we get that
\begin{align*}
    \phi '(t) & = 2\int_0^1 \partial_tF_t(u)^\dagger \chi (\mathbf{q}_\mathrm{ss}(u)) F_t(u)\diff u \\
    & = 2D\int_0^1 \Delta_\theta F_t(u)^\dagger \chi (\mathbf{q}_\mathrm{ss}(u)) F_t(u)\diff u \\
    & = -2D\int_0^1 \nabla_\theta F_t(u)^\dagger \partial_u\big( \chi (\mathbf{q}_\mathrm{ss}(u))\big) F_t(u) \diff u - 2D\int_0^1 \nabla_\theta F_t(u)^\dagger \chi (\mathbf{q}_\mathrm{ss}(u))\nabla_\theta F_t(u)\diff u
\end{align*}
Notice that
\begin{equation*}
    2\nabla_\theta F_t^\dagger \partial_u\big(\chi (\mathbf{q}_t)\big)F_t = \partial_u \Big( F_t^\dagger \partial_u\big(\chi (\mathbf{q}_t)\big)F_t\Big) - F_t^\dagger \partial_u^2\big(\chi (\mathbf{q}_t)\big)F_t
\end{equation*}
so, using once again that $F_t$ vanishes at the boundaries, we get that
\begin{equation*}
    \phi '(t) = -2D\int_0^1 \nabla_\theta F_t(u)^\dagger \chi (\mathbf{q}_\mathrm{ss}(u))\nabla_\theta F_t(u)\diff u + D\int_0^1 F_t(u)^\dagger \partial_u^2\big(\chi (\mathbf{q}_\mathrm{ss}(u))\big)F_t(u)\diff u.
\end{equation*}
Hence, we have that the limit of the quantity \eqref{eq:variance_tostudy_dirichlet} writes as
\begin{multline*}
    2D\int_0^\infty\int_0^1\nabla_\theta T_s^\theta G(u)^\dagger \chi (\mathbf{q}_\mathrm{ss}(u))\nabla_\theta T_s^\theta G(u)\diff u\diff s \\ = D \int_0^\infty \int_0^1 T_s^\theta G(u)^\dagger \partial_u^2\big(\chi (\mathbf{q}_\mathrm{ss}(u))\big)T_s^\theta G(u)\diff u\diff s
    + \phi (0)-\lim_{t\to\infty} \phi (t).
\end{multline*}
Thanks to \cref{prop:heat_semigroup_properties}, $\phi (t)$ converges exponentially fast to zero as $t$ goes to infinity. Plugging in the value of $\phi (0)$, we just proved the variance is equal to
\begin{equation*}
    \E \left[ \big|\mathcal{Y}_\mathrm{ss}(G)\big|^2\right] = \int_0^1 G(u)^\dagger \chi (\mathbf{q}_\mathrm{ss}(u))G(u)\diff u + D\int_0^\infty \int_0^1 T_s^\theta G(u)^\dagger \partial_u^2\big(\chi (\mathbf{q}_\mathrm{ss}(u))\big)T_s^\theta G(u)\diff u\diff s
\end{equation*}
as desired. Then, we obtain \eqref{eq:covariance_dirichlet_stationary} by polarization. This concludes the proof of the Dirichlet case.

\subsubsection{Robin case} Assume $\theta =1$. It also holds that $T_t^\theta G$ converges exponentially fast to zero as $t$ goes to infinity, so using the same arguments as in the Dirichlet case, we get that $\mathcal{Y}_\mathrm{ss}(G)$ is a mean-zero Gaussian variable whose variance is given by the asymptotic variance of the martingale part. Therefore, we want to compute the asymptotic behaviour as $t$ goes to infinity of the quantity
\begin{multline}\label{eq:variance_tostudy_robin}
    2D\int_0^t \int_0^1 \nabla_\theta T_s^\theta G(u)^\dagger \chi (\mathbf{q}_\mathrm{ss}(u))\nabla_\theta T_s^\theta G(u)\diff u\diff s \\ + \int_0^t \Big( T_s^\theta G(0)^\dagger \Sigma_\ell (\mathbf{q}_\mathrm{ss}(0))T_s^\theta G(0) + T_s^\theta G(1)^\dagger \Sigma_r (\mathbf{q}_\mathrm{ss}(1))T_s^\theta G(1)\Big)\diff s.
\end{multline}
As before, set $F_t=T_t^\theta G$, and $\phi (t)$ as in \eqref{eq:def_phi}. Now, the function $F_t$ satisfies Robin boundary conditions, so the integration by parts will produce boundary terms in the expression of $\phi '(t)$. Namely, we get that
\begin{multline*}
    \phi' (t) = -2D\int_0^1\nabla_\theta F_t(u)^\dagger \chi (\mathbf{q}_\mathrm{ss}(u))\nabla_\theta F_t(u)\diff u + D\int_0^1 F_t(u)^\dagger \partial_u^2\big(\chi (\mathbf{q}_\mathrm{ss}(u))\big)F_t(u)\diff u\\
    - 2D\Big( F_t(1)^\dagger \chi (\mathbf{q}_\mathrm{ss}(1))F_t(1) + F_t(0)^\dagger \chi (\mathbf{q}_\mathrm{ss}(0))F_t(0)\Big) \\
    - D\Big( F_t(1)^\dagger \partial_u (\chi (\mathbf{q}_\mathrm{ss}))(1)F_t(1) - F_t(0)^\dagger \partial_u(\chi (\mathbf{q}_\mathrm{ss}))(0)F_t(0)\Big).
\end{multline*}
Therefore, letting 
\begin{align}
    & \Xi_\ell \coloneq \Sigma_\ell (\mathbf{q}_\mathrm{ss}(0)) - 2D\chi (\mathbf{q}_\mathrm{ss}(0)) + D\partial_u(\chi (\mathbf{q}_\mathrm{ss}))(0) , \label{eq:preXil}  \\
    &\Xi_r \coloneq \Sigma_r (\mathbf{q}_\mathrm{ss}(1))  - 2D\chi (\mathbf{q}_\mathrm{ss}(1)) - D\partial_u(\chi (\mathbf{q}_\mathrm{ss}))(1) , \label{eq:preXir}
\end{align}
we get that the limit at $t$ goes to infinity of the quantity \eqref{eq:variance_tostudy_robin} is equal to
\begin{multline*}
    D \int_0^\infty \int_0^1 T_s^\theta G(u)^\dagger \partial_u^2\big(\chi (\mathbf{q}_\mathrm{ss}(u))\big)T_s^\theta G(u)\diff u\diff s
    + \phi (0)-\lim_{t\to\infty} \phi (t) \\
    + \int_0^\infty \Big( T_s^\theta G(1)^\dagger \Xi_r T_s^\theta G(1) + T_s^\theta G(0)^\dagger \Xi_\ell T_s^\theta G(0)\Big)\diff s
\end{multline*}
We know that $\phi (t)$ vanishes at infinity, and plugging in the value of $\phi (0)$, we get the desired variance formula. The result follows by polarization. 

\medskip

In \cref{appendix:robin_boundary_matrices}, we check that we recover exactly the expressions of $\Xi_\ell$ and $\Xi_r$ given in \eqref{eq:Xil} and \eqref{eq:Xir} from the definitions \eqref{eq:preXil} and \eqref{eq:preXir} after some tedious computations using the explicit expression \eqref{eq:stationary_solution} of the stationary solution $\mathbf{q}_\mathrm{ss}$. This concludes the proof in the Robin case and thus the proof of \cref{thm:stationary_fluctuations}. \hfill\qedsymbol

\subsubsection{Neumann case} Assume $\theta >1$. Recall the decomposition \eqref{eq:decomposition_limit_stationary_field} of the stationary fluctuation as a sum of two independent terms, valid for any $t\ge 0$. By \cref{prop:heat_semigroup_properties}, we know that for any~$G\in\mathcal{S}_\theta$, the semigroup $T_t^\theta G$ converges exponentially fast to $\bar{G}$ as $t$ goes to infinity. Therefore, using the variance estimate for the stationary fluctuation field and thanks to \cref{prop:heat_semigroup_properties}, we have that
\begin{equation*}
    \E\left[ \big|\mathcal{Y}_\mathrm{ss}(T_t^\theta G)-\mathcal{Y}_\mathrm{ss}(\bar{G})\big|^2\right] \le C\| T_t^\theta G-\bar{G}\|_{L^2}^2 \xrightarrow[t\to\infty]{} 0
\end{equation*}
so that the first term on the right-hand side of the decomposition converges to $\mathcal{Y}_\mathrm{ss}(\bar{G})$ in $L^2$. This term is nothing but $\bar{G}^\dagger Z$ where $Z$ is the limit in distribution of the sequence of total mass fluctuations
\begin{equation*}
    \frac{1}{\sqrt{N}}\sum_{x\in\Lambda_N} \big( \mathbf{Q}_x-\mathbf{m}_\mathrm{ss}^N(x)\big).
\end{equation*}
We have no information on the law of this term, except that it is independent of the second term of the decomposition.

\medskip

On the other hand, the second term on the right-hand side of the decomposition is a centered Gaussian variable whose variance is given by the asymptotic behaviour of 
\begin{equation}\label{eq:variance_tostudy_neumann}
    2D\int_0^t \int_0^1 \nabla_\theta T_s^\theta G(u)^\dagger \chi (\bar{\mathbf{q}})\nabla_\theta T_s^\theta G(u)\diff u\diff s = 2D\int_0^t \int_0^1 \nabla_\theta T_s^\theta G^\circ(u)^\dagger \chi (\bar{\mathbf{q}})\nabla_\theta T_s^\theta G^\circ(u)\diff u\diff s .
\end{equation}
This equality comes from the fact that $\mathbf{q}_\mathrm{ss}$ is constant equal to $\bar{\mathbf{q}}$ in this case, and 
\begin{equation*}
T_t^\theta G = T_t^\theta (\bar{G}+G^\circ) = \bar{G}+T_t^\theta G^\circ\qquad\mbox{ hence }\qquad\nabla_\theta T_t^\theta G = \nabla_\theta T_t^\theta G^\circ.
\end{equation*} 
Define $F_t=T_t^\theta G^\circ$ and $\phi (t)$ as in \eqref{eq:def_phi}. Then, differentiating $\phi$ with respect to time, and performing an integration by parts relying on the boundary conditions satisfied by $F_t$, we get that
\begin{equation*}
    \phi '(t) = -2D\int_0^1 \nabla_\theta F_t(u)^\dagger \chi (\bar{\mathbf{q}})\nabla_\theta F_t(u)\diff u.
\end{equation*}
Therefore, we have that the limit of the quantity \eqref{eq:variance_tostudy_neumann} writes as
\begin{equation*}
   2D\int_0^\infty \int_0^1 \nabla_\theta T_s^\theta G^\circ(u)^\dagger \chi (\bar{\mathbf{q}})\nabla_\theta T_s^\theta G^\circ(u)\diff u\diff s  = \phi (0)-\lim_{t\to\infty} \phi (t) = \int_0^1 G^\circ(u)^\dagger \chi (\bar{\mathbf{q}})G^\circ(u)\diff u.
\end{equation*}
This concludes the proof of the Neumann case, and therefore the proof of \cref{thm:stationary_fluctuations_neumann}. \hfill\qedsymbol

\section{Local equilibrium property}
\label{sec:local_equilibrium}

This section is devoted to the proof of the local equilibrium property \cref{prop:local_equilibrium}. This result is stronger than usual replacement lemmas because there is no space averaging in the statement, and allows rather a uniform replacement over space. The proof is based on the classical \emph{one-block} and \emph{two-blocks} estimates, and for this we adapt the strategy developed in \cite{da_cunha_hydrodynamic_2026} to prove such stronger replacement for our present model. The present case is actually slightly simpler than the one treated in \cite{da_cunha_hydrodynamic_2026} because the grand-canonical measures are product measures. Crucial ingredients are the spectral gap estimate and the equivalence of ensembles for the grand-canonical measures that were established in \cite{da_cunha_stationary_2026}.

\medskip

To be more precise, we decompose the proof of the replacement of $\tau_xf$ by $\Psi_f(\mathbf{Q}_x^{\varepsilon N} )$ into two steps:
\begin{itemize}
    \item First, we introduce a new scaling parameter $\ell$ that plays the role of an intermediary between microscopic and macroscopic scales. We prove that the replacement of $\tau_xf$ by~$\Psi_f(\mathbf{Q}_x^\ell)$ holds over large microscopic boxes, that is of size $\ell$ independent of $N$. This is the aim of the \emph{one-block estimate} \cref{lemma:one_block}.
    \item Then, we prove that the averaged densities $\mathbf{Q}_x^\ell$ over large microscopic boxes is close to the averaged densities $\mathbf{Q}_x^{\varepsilon N}$ over small macroscopic boxes. This is the content of the \emph{two-blocks estimate} \cref{lemma:two_blocks}. As the function $\Psi_f$ is continuous on the compact set~$\mathcal{T}_\kappa$, it is uniformly continuous and therefore the replacement of $\Psi_f(\mathbf{Q}_x^\ell )$ by $\Psi_f(\mathbf{Q}_x^{\varepsilon N})$ holds.
\end{itemize}
This is the strategy we implement in the remaining of this section, but before starting the proof, we need to introduce some useful tools.

\subsection{Reference measure and Dirichlet forms}

First of all, we introduce a \emph{reference measure} that is meant to approximate the stationary state $\mu_\mathrm{ss}^N$, and we construct it in such a way that it is reversible with respect to the boundary generators $\mathscr{L}_\ell$ and $\mathscr{L}_r$. Introduce the affine interpolation of the density profiles
\begin{equation*}
    \bar{\mathbf{q}}_x = \begin{pmatrix}
    \bar{\rho}_x \\
    \bar{\mathcal{E}}_x
    \end{pmatrix} \coloneq (\mathbf{q}_r-\mathbf{q}_\ell)\frac{x-1}{N-2} + \mathbf{q}_\ell
\end{equation*}
for any $x\in\Lambda_N$. Notice that $\bar{\mathbf{q}}_x\in\mathring{\mathcal{T}}_\kappa$ for any $x\in\Lambda_N$ as we assumed the boundary parameters~$\mathbf{q}_\ell$ and $\mathbf{q}_r$ to be in the interior of the triangle $\mathcal{T}_\kappa$. By construction, it satisfies the boundary conditions~$\bar{\mathbf{q}}_1=\mathbf{q}_\ell$ and $\bar{\mathbf{q}}_{N-1}=\mathbf{q}_r$.

\medskip

Then, we define the reference measure $\bar{\nu}^N$ as the product measure on $\Omega_N$ given by
\begin{equation*}
    \bar{\nu}^N = \bigotimes_{x\in\Lambda_N} \nu_{\bar{\mathbf{q}}_x}.
\end{equation*}
where we recall that the marginal $\nu_{\bar{\mathbf{q}}_x}=\nu_{\bar{\rho}_x,\bar{\mathcal{E}}_x}$ has been defined in \cref{defin:invariant_measures}. We start proving some results about this reference measure.

\begin{lemma}\label{lemma:reference_measure_lower_bound}
    We have that
    \begin{equation*}
        \inf_{\eta\in\Omega_N} \bar{\nu}^N(\eta )\ge c^{N-1}
    \end{equation*}
    where $c>0$ is a constant that depends only on the boundary parameters $\mathbf{q}_\ell$, $\mathbf{q}_r$ and $\kappa$. 
\end{lemma}

\begin{proof}
    Notice that for any $x\in\Lambda_N$, we have
    \begin{equation*}
        0<\rho_\ell\wedge\rho_r \le \bar{\rho}_x \le \rho_\ell\vee\rho_r <1\qquad\mbox{ and }\qquad 0<\mathcal{E}_\ell\wedge\mathcal{E}_r \le \bar{\mathcal{E}}_x \le \mathcal{E}_\ell\vee\mathcal{E}_r < \kappa
    \end{equation*}
    hence
    \begin{equation*}
        0<\mathfrak{p}_- \coloneq \mathfrak{p}_\ell\wedge\mathfrak{p}_r\le \mathfrak{p}(\bar{\rho}_x,\bar{\mathcal{E}}_x) \le \mathfrak{p}_+\coloneq \mathfrak{p}_\ell\vee\mathfrak{p}_r<1.
    \end{equation*}
    Notice that $\nu_{\bar{\mathbf{q}}_x}(0) = 1-\bar{\rho}_x \ge 1-\rho_\ell\vee\rho_r$, and for any $j\in \{1,\hdots ,\kappa\}$, we have
    \begin{equation*}
        \nu_{\bar{\mathbf{q}}_x}(j) = \rho {{\kappa -1}\choose{j-1}}\mathfrak{p}(\bar{\mathbf{q}}_x)^{j-1}(1-\mathfrak{p}(\bar{\mathbf{q}}_x))^{\kappa-j} \ge (\rho_\ell\wedge\rho_r) \mathfrak{p}_-^{j-1}(1-\mathfrak{p}_+)^{\kappa-j}.
    \end{equation*}
    This yields a constant $c>0$ such that $\nu_{\bar{\mathbf{q}}_x}(j)\ge c$ for any $j\in\{0,\hdots ,\kappa\}$ and any $x\in\Lambda_N$. Then
    \begin{equation*}
        \bar{\nu}^N(\eta )= \prod_{x\in\Lambda_N} \nu_{\bar{\mathbf{q}}_x}(\eta_x) \ge c^{N-1}
    \end{equation*}
    and the result is proved.
\end{proof}

\begin{lemma}[Quasi-reversibility]\label{lemma:quasi_reversibility}
    The measure $\bar{\nu}^N$ is reversible with respect to the left boundary generators $\mathscr{L}_\ell$ and $\mathscr{L}_r$. In the bulk, it is ``\emph{quasi-reversible}'' in the sense that there exists a constant $C>0$ such that for any $x\in \{1,\hdots ,N-2\}$, we have the estimates
    \begin{equation*}
        \sup_{\eta\in\Omega_N} \left| \frac{\bar{\nu}^N(\eta^{x,x+1})}{\bar{\nu}^N(\eta )}-1\right| \le \frac{C}{N},
    \end{equation*}
    and also, for $|x-y|=1$,
    \begin{equation*}
        \sup_{\substack{\eta\in\Omega_N \\ c_{x\to y}^e(\eta )>0}} \left| \frac{c_{y\to x}^e(\eta^{x\to y})\bar{\nu}^N(\eta^{x\to y})}{c_{x\to y}^e(\eta)\bar{\nu}^N(\eta)} - 1\right| \le \frac{C}{N}.
    \end{equation*}
\end{lemma}

\begin{proof}
    The respective marginals of $\bar{\nu}^N$ on $x=1$ and $x=N-1$ are equal to $\nu_{\mathbf{q}_\ell}$ and $\nu_{\mathbf{q}_r}$, so the reversibility is quite immediate by checking the detailed balance condition for all the possible transitions at the boundaries.

    \medskip

    The other two estimates are consequences of the fact that the reference measure $\bar{\nu}^N$ is a product measure, and of the equalities
    \begin{equation}\label{eq:lipschitz_interpolation}
        |\bar{\rho}_{x+1}-\bar{\rho}_x| = \frac{|\rho_r-\rho_\ell|}{N-2}\qquad\mbox{ and }\qquad |\bar{\mathcal{E}}_{x+1}-\bar{\mathcal{E}}_x| = \frac{|\mathcal{E}_r-\mathcal{E}_\ell|}{N-2}.
    \end{equation}
    We do not give the details of the computations here, but they are similar to ones performed for instance in \cite{bernardin2019slow}.
\end{proof}

\begin{definition}[Relative entropy]
    Given two probability measures $\mu$ and $\nu$ on $\Omega_N$, we define their \emph{relative entropy} by
    \begin{equation*}
        H(\mu |\nu ) = \int_{\Omega_N} \log\left( \frac{\mu (\eta )}{\nu (\eta )}\right)\diff \mu (\eta ) = \E_\nu\left[ \frac{\diff\mu}{\diff\nu}\log\left(\frac{\diff\mu}{\diff\nu}\right)\right].
    \end{equation*}
    We refer to \cite[Appendix 1.1.8]{kipnis_scaling_1999} for some properties of the relative entropy.
\end{definition} 

\noindent Then, the reference measure $\bar{\nu}^N$ satisfies the following entropy estimate.

\begin{lemma}[Entropy estimate]\label{lemma:entropy_estimate}
    There exists a constant $C=C(\mathbf{q}_\ell ,\mathbf{q}_r,\kappa )>0$ such that for any probability measure $\mu$ on $\Omega_N$, we have
    \begin{equation}
        H(\mu |\bar{\nu}^N) \le C N.
    \end{equation}
\end{lemma}

\begin{proof}
    Let $\mu$ be any probability measure on $\Omega_N$. As $\bar{\nu}^N(\eta )\ge c^{N-1}$ for any $\eta\in\Omega_N$ by \cref{lemma:reference_measure_lower_bound}, and $\mu (\eta )\le 1$, we have that $\log (\frac{\diff\mu}{\diff\bar{\nu}^N}) \le \log (\frac1c)(N-1)$. This yields the desired estimate.
\end{proof}

\noindent Finally, we introduce the Dirichlet forms associated to the dynamics of the model.

\begin{definition}[Dirichlet forms]
    Let $\mu$ be any probability measure on $\Omega_N$. We define the \emph{Dirichlet form} of a non-negative local function $f:\Omega_N\longrightarrow\R_+$ with respect to the measure $\mu$ by
    \begin{equation*}
        \mathfrak{D}_N(f;\mu) = \sum_{x=1}^{N-2} \mathfrak{D}_{x,x+1}^p(f ;\mu ) + \sum_{x=1}^{N-2}\sum_{|y-x|=1}\mathfrak{D}_{x\to y}^e(f;\mu ) + N^{-\theta}\mathfrak{D}_\ell(f;\mu) + N^{-\theta}\mathfrak{D}_r(f;\mu)
    \end{equation*}
    where
    \begin{subequations}
        \begin{align}
            & \mathfrak{D}_{x,y}^p (f;\mu ) = \int_{\Omega_N}c_{x,y}^p (\eta ) \big[\sqrt{f(\eta^{x,y})}-\sqrt{f(\eta )}\big]^2\diff\mu (\eta )\\
            & \mathfrak{D}_{x\to y}^e(f ;\mu ) =  \int_{\Omega_N}c_{x\to y}^e (\eta ) \big[\sqrt{f(\eta^{x\to y})}-\sqrt{f(\eta )}\big]^2\diff\mu (\eta )
        \end{align}
    \end{subequations}
    are the bulk Dirichlet forms, and the left boundary Dirichlet form is defined by
    \begin{multline}
        \mathfrak{D}_\ell (f;\mu ) = (\kappa -1)(1-\rho_\ell) \int_{\Omega_N} \xi_1\big[ \sqrt{f(\eta^{1\to 0})}-\sqrt{f(\eta )}\big]^2\diff\mu (\eta ) \\
        + \sum_{j=1}^\kappa {{\kappa -1}\choose{j-1}} (\kappa -1)\rho_\ell \mathfrak{p}_\ell^{j-1}(1-\mathfrak{p}_\ell)^{\kappa -j} \int_{\Omega_N} (1-\xi_1)\big[ \sqrt{f(\eta^{1;j})}-\sqrt{f(\eta )}\big]^2\diff\mu (\eta )\\ 
        + (\mathcal{E}_\ell - \rho_\ell)\int_{\Omega_N} (\kappa \xi_1-\eta_1)\big[ \sqrt{f(\eta^{1,+})}-\sqrt{f(\eta )}\big]^2\diff\mu (\eta ) \\+ (\kappa\rho_\ell -\mathcal{E}_\ell)\int_{\Omega_N}(\eta_1-\xi_1)\big[ \sqrt{f(\eta^{1,-})}-\sqrt{f(\eta )}\big]^2\diff\mu (\eta ).
    \end{multline}
    The Dirichlet form $\mathfrak{D}_r(f;\mu )$ at the right boundary is defined similarly. The map $f\longmapsto \mathfrak{D}_N(f;\mu)$ is convex, non-negative and lower semi-continuous.
\end{definition}

The next result gives an estimation of the spectral radius of the generator $\mathscr{L}_N$ by the Dirichlet form with respect to the reference measure $\bar{\nu}^N$. 

\begin{lemma}\label{lemma:dirichlet_form_estimate}
    There exists a constant $C=C(\mathbf{q}_\ell ,\mathbf{q}_r,\kappa )>0$ such that for any probability density function $f:\Omega_N\longrightarrow\R_+$ with respect to the reference measure $\bar{\nu}^N$, we have
    \begin{equation}
        \big\langle \sqrt{f},\mathscr{L}_N\sqrt{f}\big\rangle_{\bar{\nu}^N} \le -\frac14 \mathfrak{D}_N(f;\bar{\nu}^N) + \frac{C}{N}\, .
    \end{equation}
\end{lemma}

\begin{proof}
    Recall that the measure $\bar{\nu}^N$ is reversible with respect to the boundary generators $\mathscr{L}_\ell$ and $\mathscr{L}_r$. Standard manipulations then yield that
    \begin{equation*}
        \big\langle \sqrt{f},\mathscr{L}_\ell\sqrt{f}\big\rangle_{\bar{\nu}^N} = -\frac12 \mathfrak{D}_\ell(f;\bar{\nu}^N)\qquad\mbox{ and }\qquad \big\langle \sqrt{f},\mathscr{L}_r\sqrt{f}\big\rangle_{\bar{\nu}^N} = -\frac12 \mathfrak{D}_r(f;\bar{\nu}^N).
    \end{equation*}
    Moreover, thanks to the quasi-reversibility property \cref{lemma:quasi_reversibility}, and by \cite[Lemma 5.1]{bernardin2019slow}, we have that 
    \begin{equation*}
        \big\langle \sqrt{f},\mathscr{L}_0\sqrt{f}\big\rangle_{\bar{\nu}^N} \le -\frac 14 \sum_{x=1}^{N-2} \mathfrak{D}_{x,x+1}^p(f ;\bar{\nu}^N) - \frac14 \sum_{x=1}^{N-2}\sum_{|y-x|=1}\mathfrak{D}_{x\to y}^e(f;\bar{\nu}^N) + \frac{C}{N}.
    \end{equation*} 
    Gathering all the previous estimates, we get the desired result.
\end{proof}

We are now in position to prove the one-block and two-blocks estimates that are the main ingredients to prove the local equilibrium property \cref{prop:local_equilibrium}.

\subsection{One-block estimate}

\begin{lemma}[One-block estimate]\label{lemma:one_block}
    For any $t\in [0,T]$, we have that
    \begin{equation}
        \lim_{\ell\to\infty}\lim_{N\to\infty} \sup_{x\in\Lambda_f}\E_{\mu^N}\left[ \bigg|\int_0^t \tau_xf(\eta (s)) - \Psi_f(\mathbf{Q}_x^\ell (s))\diff s\bigg|\right] =0.
    \end{equation}
\end{lemma}

\begin{proof}
    Fix $x\in\Lambda_f$, an admissible site for the local function $\tau_x f$ to be supported in $\Lambda_N$. Consider the stationary measure $\nu_{\rho,\mathcal{E}}^N$ defined in \cref{defin:invariant_measures} for some fixed parameters $(\rho ,\mathcal{E})\in\mathring{\mathcal{T}}_\kappa$. Let~$f_x^\ell$ be the conditional expectation of $\tau_xf$ with respect to the empirical average $\mathbf{Q}_x^\ell$ under the restriction of $\nu_{\rho,\mathcal{E}}^N$ to the box $\Lambda_x^\ell$. As the measure $\nu_{\rho,\mathcal{E}}^N$ conditioned to a fixed value of~$\mathbf{Q}_x^\ell$ is uniform, it does not depend on the parameters $(\rho ,\mathcal{E})$, and therefore $f_x^\ell$ is also independent of~$(\rho ,\mathcal{E})$. The equivalence of ensembles \cite[Proposition 17]{da_cunha_stationary_2026} for the measure $\nu_{\rho,\mathcal{E}}^N$ then yields that
    \begin{equation*}
        \sup_{\eta\in\Omega_N} \big| f_x^\ell(\eta ) - \Psi_f(\mathbf{Q}_x^\ell (\eta ))\big| \le \frac{C}{\ell}
    \end{equation*}
    for some constant $C>0$ that depends only on the local function $f$. Therefore, it is enough to prove that
    \begin{equation}\label{eq:one_block_estimate_to_prove}
        \lim_{N\to\infty} \sup_{x\in\Lambda_f}\E_{\mu^N}\left[ \bigg|\int_0^t V_x^\ell (\eta (s))\diff s\bigg|\right] =0
    \end{equation}
    where $V_x^\ell = \tau_xf - f_x^\ell$. The entropy inequality \cite[Appendix 1.8]{kipnis_scaling_1999} implies that the expectation in \eqref{eq:one_block_estimate_to_prove} is bounded above by
    \begin{equation*}
        \frac{H(\mu^N|\bar{\nu}^N)}{\gamma N} + \frac{1}{\gamma N}\log \E_{\bar{\nu}^N}\left[ \exp\left( \gamma N \bigg|\int_0^t V_x^\ell (\eta (s))\diff s\bigg|\right)\right]
    \end{equation*}
    for any $\gamma >0$. Thanks to the inequality $e^{|x|}\le e^x+e^{-x}$ together with the inequality
    \begin{equation*}
        \limsup_{N\to\infty} \frac{1}{N}\log (a_N+b_N) \le \max\left\{ \limsup_{N\to\infty}\frac{1}{N}\log a_N , \limsup_{N\to\infty}\frac{1}{N}\log b_N\right\},
    \end{equation*}
    we can remove the absolute value in the expectation above. Then, the Feynman-Kac formula \cite[Appendix 1.7]{kipnis_scaling_1999} together with \cref{lemma:entropy_estimate,lemma:dirichlet_form_estimate} imply that the expectation in \eqref{eq:one_block_estimate_to_prove} is bounded above by
    \begin{equation}\label{eq:supremum_estimate_one_block}
        \frac{C}{\gamma} + t\sup_g \left\{ \bar{\nu}^N(|V_x^\ell|g) - \frac{N}{4\gamma}\mathfrak{D}_N(g;\bar{\nu}^N) \right\}.
    \end{equation}
    where the supremum is taken over density functions with respect to $\bar{\nu}^N$.
    Let us estimate the supremum in the above expression. We start by localizing everything to the block $\Lambda_x^\ell$. For this, we introduce the following notations:
    \begin{itemize}
        \item The measure $\bar{\nu}_x^\ell$ is the restriction of the reference measure $\bar{\nu}^N$ to the box $\Lambda_x^\ell$ ;
        \item The function $g_x^\ell$ is the conditional expectation of $g$ with respect to the box $\Lambda_x^\ell$ under the measure $\bar{\nu}^N$ ;
        \item If $h:\{0,1,\hdots ,\kappa\}^{\Lambda_x^\ell}\longrightarrow\R_+$, we define its Dirichlet form $\mathfrak{D}_x^\ell$ on $\Lambda_x^\ell$ by
        \begin{equation*}
            \mathfrak{D}_x^\ell (h;\bar{\nu}_x^\ell) = \sum_{\{y,y+1\}\subset\Lambda_x^\ell} \mathfrak{D}_{y,y+1}^p (h ;\bar{\nu}_x^\ell ) + \sum_{\{y,y+1\}\subset\Lambda_x^\ell}\sum_{|z-y|=1}\mathfrak{D}_{y\to z}^e(h;\bar{\nu}_x^\ell ),
        \end{equation*}
        where with a slight abuse of notation, the integrals are now taken over $\{0,1,\hdots, \kappa\}^{\Lambda_x^\ell}$ instead of $\Omega_N$.
    \end{itemize}
    Then, by convexity of the Dirichlet form, and as $g_x^\ell$ can be seen either as a function on $\Omega_N$ or on~$\{0,1,\hdots ,\kappa\}^{\Lambda_x^\ell}$, we have that
    \begin{align*}
        \mathfrak{D}_x^\ell (g_x^\ell ; \bar{\nu}_x^\ell ) \le \mathfrak{D}_N(g;\bar{\nu}^N).
    \end{align*}
    Note that the function $V_x^\ell$ depends only on the coordinates inside the box $\Lambda_x^\ell$, therefore
    \begin{equation*}
        \bar{\nu}^N(V_x^\ell g) = \bar{\nu}^N(V_x^\ell g_x^\ell) = \bar{\nu}_x^\ell(V_x^\ell g_x^\ell).
    \end{equation*}
    It follows that the supremum in \eqref{eq:supremum_estimate_one_block} is bounded above by
    \begin{equation}\label{eq:localized_supremum_estimate_one_block}
        \sup_g \left\{ \bar{\nu}_x^\ell(V_x^\ell g) - \frac{N}{4\gamma}\mathfrak{D}_x^\ell(g;\bar{\nu}_x^\ell) \right\}
    \end{equation}
    where this time the supremum is taken over density functions $g:\{0,1,\hdots ,\kappa\}^{\Lambda_x^\ell}$ with respect to~$\bar{\nu}_x^\ell$. 

    \medskip

    Now that everything is localized to the box $\Lambda_x^\ell$, we change the inhomogeneous measure $\bar{\nu}_x^\ell$ to an homogeneous measure $\tilde{\nu}_x^\ell$. More precisely, define the homogeneous measure with parameter~$\bar{\mathbf{q}}_x$, \emph{i.e.}
    \begin{equation*}
        \tilde{\nu}_x^\ell = \big(\nu_{\bar{\mathbf{q}}_x} \big)^{\otimes\Lambda_x^\ell}
    \end{equation*}
    and the Radon-Nikodym derivative
    \begin{equation*}
        h_x^\ell = \frac{\diff\bar{\nu}_x^\ell}{\diff\tilde{\nu}_x^\ell} = \prod_{y\in\Lambda_x^\ell} \frac{\nu_{\bar{\mathbf{q}}_y}}{\nu_{\bar{\mathbf{q}}_x}}.
    \end{equation*}
    Thanks to the structure \eqref{eq:lipschitz_interpolation} of the interpolation, one can check that for any $\sigma\in\{0,1,\hdots ,\kappa\}^{\Lambda_x^\ell}$
    \begin{equation*}
       \big|\log h_x^\ell (\sigma )\big| \le \sum_{y\in\Lambda_x^\ell} \big|\log \nu_{\bar{\mathbf{q}}_y}(\sigma_y) - \log \nu_{\bar{\mathbf{q}}_x}(\sigma_y)\big| \le C\frac{\ell^2}{N}
    \end{equation*}
    where $C>0$ is a constant that is independent of $x$, but depends only on the boundary parameters~$\mathbf{q}_\ell$, $\mathbf{q}_r$ and $\kappa$. This implies that uniformly over $x$, we have the bounds
    \begin{equation}\label{eq:bounds_Radon_Nikodym_derivative}
        e^{-C\ell^2/N} \le h_x^\ell (\sigma ) \le e^{C\ell^2/N}.
    \end{equation}

    If $g$ is a density function with respect to $\bar{\nu}_x^\ell$, define $\tilde{g} = g/\tilde{\nu}_x^\ell (g)$ so that $\tilde{g}$ is a density function with respect to $\tilde{\nu}_x^\ell$. Notice that $1=\bar{\nu}_x^\ell (g) = \tilde{\nu}_x^\ell (h_x^\ell g)$ so the bound \eqref{eq:bounds_Radon_Nikodym_derivative} implies that
    \begin{equation}\label{eq:bound_mass_of_g}
        e^{-C\ell^2/N} \le \tilde{\nu}_x^\ell (g)\le e^{C\ell^2/N}.
    \end{equation}
    Therefore,
    \begin{align*}
        \big|\bar{\nu}_x^\ell (V_x^\ell g) - \tilde{\nu}_x^\ell (V_x^\ell \tilde{g})\big| & = \tilde{\nu}_x^\ell \left( V_x^\ell g\big( h_x^\ell -\tilde{\nu}_x^\ell (g)^{-1}\big)\right) \\
        & \le 2\| f\|_\infty |e^{C\ell^2/N}-e^{-C\ell^2/N}| \\
        & \le C\frac{\ell^2}{N}
    \end{align*}
    because $|V_x^\ell|\le 2\| f\|_\infty$. Therefore, we can replace the inhomogeneous measure $\bar{\nu}_x^\ell$ by the homogeneous measure $\tilde{\nu}_x^\ell$ in the first term inside the supremum \eqref{eq:localized_supremum_estimate_one_block} up to an error of order $\ell^2/N$. For the Dirichlet form part, thanks to \eqref{eq:bounds_Radon_Nikodym_derivative}, we have that $\mathfrak{D}_x^\ell (g ; \bar{\nu}_x^\ell ) \ge e^{-C\ell^2/N}\mathfrak{D}_x^\ell (g ; \tilde{\nu}_x^\ell )$. Moreover, by the definition of $\tilde{g}$ and the bounds \eqref{eq:bound_mass_of_g}, we have that
    \begin{equation*}
        \mathfrak{D}_x^\ell (g ; \bar{\nu}_x^\ell )\ge e^{-2C\ell^2/N}\mathfrak{D}_x^\ell (\tilde{g} ; \tilde{\nu}_x^\ell ).
    \end{equation*}
    Gathering all the previous estimates, we have proved that the supremum \eqref{eq:localized_supremum_estimate_one_block} is bounded above by
    \begin{equation}\label{eq:homogeneous_supremum_estimate_one_block}
        \sup_{\tilde{g}} \left\{ \tilde{\nu}_x^\ell(V_x^\ell \tilde{g}) - \frac{N}{4\gamma}e^{-2C\ell^2/N}\mathfrak{D}_x^\ell(\tilde{g};\tilde{\nu}_x^\ell) \right\} + C\frac{\ell^2}{N}
    \end{equation}
    where the supremum is taken over density functions $\tilde{g}$ with respect to $\tilde{\nu}_x^\ell$. We are reduced to an expression that contains only a homogeneous product measure. Set $w=\sqrt{\tilde{g}}$ and define its conditional expectation $\bar{w}=\tilde{\nu}_x^\ell (w|\mathbf{Q}_x^\ell)$ with respect to the empirical particle and energy density averages over $\Lambda_x^\ell$. The conditional expectation of $V_x^\ell$ with respect to $\mathbf{Q}_x^\ell$ is equal to zero by construction, therefore we have that
    \begin{equation*}
        \tilde{\nu}_x^\ell (V_x^\ell \tilde{g}) = \tilde{\nu}_x^\ell (V_x^\ell (w^2-\bar{w}^2)) = \tilde{\nu}_x^\ell\big( V_x^\ell (w-\bar{w})(w+\bar{w})\big).
    \end{equation*}
    Cauchy-Schwarz inequality then yields that
    \begin{equation*}
        \tilde{\nu}_x^\ell (V_x^\ell \tilde{g}) \le \sqrt{\tilde{\nu}_x^\ell\big( (V_x^\ell)^2 (w+\bar{w})^2\big)}\sqrt{\tilde{\nu}_x^\ell\big( (w-\bar{w})^2\big)} \le 4\| f\|_\infty \sqrt{\tilde{\nu}_x^\ell\big( (w-\bar{w})^2\big)}.
    \end{equation*}
    Notice that the function $w-\bar{w}$ has mean zero under $\nu_{\rho,\mathcal{E}}^{\otimes 2\ell +1}$ for any $(\rho ,\mathcal{E})\in\mathcal{T}_\kappa$. Therefore, we can apply the spectral gap inequality proved for these homogeneous product measures in \cite[Proposition 22]{da_cunha_stationary_2026}. It gives a universal constant $C'=C'(\kappa )>0$ for which
    \begin{equation*}
        \tilde{\nu}_x^\ell\big( (w-\bar{w})^2\big) \le C'\ell^2 \mathfrak{D}_x^\ell (g;\tilde{\nu}_x^\ell).
    \end{equation*}
    As a consequence, the supremum in \eqref{eq:homogeneous_supremum_estimate_one_block} is bounded above by
    \begin{multline*}
        \sup_{\tilde{g}} \left\{ 4\sqrt{C'}\|f\|_\infty \ell \sqrt{\mathfrak{D}_x^\ell (g;\tilde{\nu}_x^\ell)} -\frac{N}{4\gamma}e^{-2C\ell^2/N}\mathfrak{D}_x^\ell (g;\tilde{\nu}_x^\ell)\right\}\\
        \le \sup_{a\ge 0}\left\{ 4\sqrt{C'}\| f\|_\infty \sqrt{a} - \frac{N}{4\gamma}e^{-2C\ell^2/N}a\right\} = \frac{16C'\| f\|_\infty^2\ell^2\gamma}{N}e^{2C\ell^2/N}.
    \end{multline*}
    Gathering all the previous estimates, we have proved that for any $\gamma >0$, we have 
    \begin{equation*}
       \E_{\mu^N}\left[ \int_0^t\big| V_x^\ell (\eta (s))|\diff s\right] \le \frac{C}{\gamma} + \frac{CT\ell^2}{N} +\frac{16C'\| f\|_\infty^2\ell^2\gamma}{N}e^{2C\ell^2/N}.
    \end{equation*}
    This estimate is uniform over $x\in\Lambda_f$, therefore we can take the supremum over $x$ on the left-hand side at this point. Letting $N$ go to infinity, and then $\gamma$ go to infinity, we proved \eqref{eq:one_block_estimate_to_prove}, and this concludes the proof of \cref{lemma:one_block}.
\end{proof}

\subsection{Two-blocks estimate}

\begin{lemma}[Two-blocks estimate]\label{lemma:two_blocks}
    For any $t\in [0,T]$, we have that
    \begin{equation}
        \limsup_{\ell\to\infty}\limsup_{\varepsilon\to 0}\limsup_{N\to\infty} \sup_{x\in\Lambda_N}\E_{\mu^N}\left[\int_0^t |\mathbf{Q}_x^\ell (s)-\mathbf{Q}_x^{\varepsilon N}(s)|\diff s\right] =0.
    \end{equation}
\end{lemma}

\begin{proof}
    The ideas of the proof are similar to the ones used in the proof of the one-block estimate, so some steps are only sketched while some others are more detailed.

    \medskip

    Fix $x\in\Lambda_N$. As before, using subsequently the entropy inequality, the Feynman-Kac formula, the entropy bound \cref{lemma:entropy_estimate} and the Dirichlet form estimate \cref{lemma:dirichlet_form_estimate}, we get that the expectation in the statement is bounded above by
    \begin{equation*}
        \frac{C}{\gamma} + t\sup_g \left\{ \bar{\nu}^N\big(|\mathbf{Q}_x^\ell - \mathbf{Q}_x^{\varepsilon N}|g\big) - \frac{N}{4\gamma}\mathfrak{D}_N(g;\bar{\nu}^N) \right\}
    \end{equation*}
    for any $\gamma >0$, where the supremum is taken over density functions with respect to $\bar{\nu}^N$. Divide the box $\Lambda_x^{\varepsilon N}$ into $p=\big\lfloor \frac{2\varepsilon N+1}{2\ell +1}\big\rfloor$ disjoint boxes of size $2\ell +1$, plus possibly two leftover blocks of size at most $2\ell +1$. Then, we can write that 
    \begin{equation*}
        \mathbf{Q}_x^\ell -\mathbf{Q}_x^{\varepsilon N} =\frac1p \sum_{j=-p/2}^{p/2} (\mathbf{Q}_x^\ell -\mathbf{Q}_{x+j(2\ell +1)}^\ell ) + r (\ell ,\varepsilon ,N)
    \end{equation*}
    where $r(\ell ,\varepsilon ,N)$ is a remainder term or order $O\big(\frac{\ell}{\varepsilon N}\big)$. In order to estimate the supremum above, we have to estimate the supremum
    \begin{equation}\label{eq:supremum_estimate_two_blocks}
        \sup_g \left\{ \frac1p \sum_{j=-p/2}^{p/2}\bar{\nu}^N\big( |\mathbf{Q}_x^\ell -\mathbf{Q}_{x+j(2\ell +1)}^\ell |g\big) - \frac{N}{4\gamma}\mathfrak{D}_N(g;\bar{\nu}^N) \right\}.
    \end{equation}
    In order to simplify the notations, we define $y_j= x+j(2\ell +1)$. Notice that $\mathbf{Q}_x^\ell - \mathbf{Q}_{y_j}^\ell$ depends only on the coordinates inside the box $\Lambda_{x,j}^\ell\coloneq \Lambda_x^\ell\cup \Lambda_{y_j}^\ell$. We introduce the following notations:
    \begin{itemize}
        \item The measure $\bar{\nu}_{x,j}^\ell$ is the restriction of the reference measure $\bar{\nu}^N$ to $\Lambda_{x,j}^\ell$ ;
        \item The function $g_{x,j}^\ell$ is the conditional expectation of $g$ with respect to the box $\Lambda_{x,j}^\ell$ under the measure $\bar{\nu}^N$ ;
        \item If $h:\{0,1,\hdots ,\kappa\}^{\Lambda_{x,j}^\ell}\longrightarrow\R_+$, we define its Dirichlet form $\mathfrak{D}_{x,j}^\ell$ on $\Lambda_{x,j}^\ell$ by
        \begin{equation*}
            \mathfrak{D}_{x,j}^\ell (h;\bar{\nu}_{x,j}^\ell) = J_{x,y_j}(h;\bar{\nu}_{x,j}^\ell)+ \sum_{\{y,y+1\}\subset\Lambda_{x,j}^\ell} \mathfrak{D}_{y,y+1}^p (h ;\bar{\nu}_{x,j}^\ell ) + \sum_{\{y,y+1\}\subset\Lambda_{x,j}^\ell}\sum_{|z-y|=1}\mathfrak{D}_{y\to z}^e(h;\bar{\nu}_{x,j}^\ell )
        \end{equation*}
        where $J_{x,y_j}(h;\bar{\nu}_{x,j}^\ell)$ is a Dirichlet form that allows the two boxes $\Lambda_x^\ell$ and $\Lambda_{y_j}^\ell$ to communicate. It is defined by
        \begin{equation}\label{eq:def_J_x_yj}
            J_{x,y_j}(h;\bar{\nu}_{x,j}^\ell) = \mathfrak{D}_{x,y_j}^p (h ;\bar{\nu}_{x,j}^\ell ) + \mathfrak{D}_{x\to y_j}^e(h;\bar{\nu}_{x,j}^\ell ) + \mathfrak{D}_{y_j\to x}^e(h;\bar{\nu}_{x,j}^\ell ).
        \end{equation}
        Adding this Dirichlet form is like adding a fictitious bond $\{x,y_j\}$ on which particles can jump and exchange energy.
    \end{itemize}
    Our goal is first to estimate
    \begin{equation*}
        \frac1p \sum_{j=-p/2}^{p/2}\mathfrak{D}_{x,j}^\ell (g_{x,j}^\ell ;\bar{\nu}_{x,j}^\ell)
    \end{equation*}
    in terms of the total Dirichlet form $\mathfrak{D}_N(g;\bar{\nu}^N)$. Performing the same proof as in the one-block estimate, we easily check that
    \begin{equation*}
        \frac{1}{p}\sum_{j=-p/2}^{p/2}\left( \sum_{\{y,y+1\}\subset\Lambda_{x,j}^\ell} \mathfrak{D}_{y,y+1}^p (g_{x,j}^\ell ;\bar{\nu}_{x,j}^\ell ) + \sum_{\{y,y+1\}\subset\Lambda_{x,j}^\ell}\sum_{|z-y|=1}\mathfrak{D}_{y\to z}^e(g_{x,j}^\ell;\bar{\nu}_{x,j}^\ell )\right) \le \mathfrak{D}_N(g;\bar{\nu}^N),
    \end{equation*}
    so we only have to estimate the contribution
    \begin{equation*}
        \frac{1}{p}\sum_{j=-p/2}^{p/2} J_{x,y_j}(g_{x,j}^\ell ;\bar{\nu}_{x,j}^\ell)
    \end{equation*}
    coming from the fictitious bonds $\{x,y_j\}$. We can extend $J_{x,y_j}$ to a Dirichlet form $\bar{J}_{x,y_j}$ on the whole space $\Omega_N$. Then, by convexity, we have that
    \begin{equation*}
        J_{x,y_j}(g_{x,j}^\ell ;\bar{\nu}_{x,j}^\ell) \le \bar{J}_{x,y_j}(g;\bar{\nu}^N).
    \end{equation*}
    We use the forthcoming moving particle-energy \cref{lemma:moving_particle} to estimate the above quantity as follows
    \begin{equation*}
        \bar{J}_{x,y_j}(g;\bar{\nu}^N) \le C|y_j-x|\mathfrak{D}_N(g;\bar{\nu}^N) \le C\varepsilon N\mathfrak{D}_N(g;\bar{\nu}^N)
    \end{equation*}
    Gathering all the previous estimates, we have proved that
    \begin{equation*}
        \frac1p \sum_{j=-p/2}^{p/2}\mathfrak{D}_{x,j}^\ell (g_{x,j}^\ell ;\bar{\nu}_{x,j}^\ell) \le C\varepsilon N\mathfrak{D}_N(g;\bar{\nu}^N).
    \end{equation*}
    Thanks to this estimate, we can proceed as in the proof of the one-block to localize everything to the box $\Lambda_{x,j}^\ell$. Indeed, this implies that the supremum \eqref{eq:supremum_estimate_two_blocks} is bounded above by
    \begin{multline*}
        \sup_g \left\{\frac1p\sum_{j=-p/2}^{p/2} \left( \bar{\nu}_{x,j}^\ell\big( |\mathbf{Q}_x^\ell -\mathbf{Q}_{y_j}^\ell |g_{x,j}^\ell\big) - \frac{1}{4C\gamma\varepsilon}\mathfrak{D}_{x,j}^\ell (g_{x,j}^\ell ;\bar{\nu}_{x,j}^\ell)\right)\right\} \\
        \le \sup_{|j|\le p/2} \sup_g \left\{ \bar{\nu}_{x,j}^\ell (|\mathbf{Q}_x^\ell -\mathbf{Q}_{y_j}^\ell |g) - \frac{1}{4C\gamma\varepsilon}\mathfrak{D}_{x,j}^\ell (g ; \bar{\nu}_{x,j}^\ell)\right\}
    \end{multline*}
    where this time, the supremum is taken over density functions $g:\{0,1,\hdots ,\kappa\}^{\Lambda_{x,j}^\ell}\longrightarrow\R_+$ with respect to $\bar{\nu}_{x,j}^\ell$. We now change the inhomogeneous product measure $\bar{\nu}_{x,j}^\ell$ to a homogeneous product measure 
    \begin{equation*}
        \tilde{\nu}_{x,j}^\ell = \big( \nu_{\bar{\mathbf{q}}_x}\big)^{\otimes \Lambda_{x,j}^\ell}
    \end{equation*}
    using the same techniques as in the proof of the one-block estimate. More precisely, by estimating the Radon-Nikodym derivative between both measures, one can check that the supremum above is bounded by
    \begin{equation}\label{eq:supremum_estimate_two_blocks_homogeneous}
        C\delta (N,\varepsilon ,\ell) + \sup_{g}\left\{ \tilde{\nu}_{x,j}^\ell (|\mathbf{Q}_x^\ell -\mathbf{Q}_{y_j}^\ell |g) - \frac{e^{-2\delta (N,\varepsilon ,\ell)}}{4C\gamma\varepsilon}\mathfrak{D}_{x,j}^\ell(g;\tilde{\nu}_{x,j}^\ell)\right\}
    \end{equation}
    where $\delta (N,\varepsilon ,\ell )$ is some error term that satisfies
    \begin{equation*}
        \delta (N,\varepsilon ,\ell ) \le C\ell \left( \varepsilon + \frac{\ell}{N}\right)
    \end{equation*}
    and the supremum is taken over densities with respect to $\tilde{\nu}_{x,j}^\ell$. Hence, we are left to estimate the supremum in \eqref{eq:supremum_estimate_two_blocks_homogeneous} with respect to a homogeneous product measure. If $g$ is such a density, we set~$w=\sqrt{g}$ and $\bar{w} = \tilde{\nu}_{x,j}^\ell (w|\mathbf{Q}_x^\ell+\mathbf{Q}_{y_j}^\ell)$. Set also $V_{x,j}^\ell = |\mathbf{Q}_x^\ell -\mathbf{Q}_{y_j}^\ell |$ and $\bar{V}_{x,j}^\ell = \tilde{\nu}_{x,j}^\ell (V_{x,j}^\ell | \mathbf{Q}_x^\ell +\mathbf{Q}_{y_j}^\ell)$. Then, by Cauchy-Schwarz inequality we have that
    \begin{align*}
        \big| \tilde{\nu}_{x,j}^\ell (V_{x,j}^\ell g) - \tilde{\nu}_{x,j}^\ell (\bar{V}_{x,j}^\ell g)\big| & = \big| \tilde{\nu}_{x,j}^\ell \big( (V_{x,j}^\ell - \bar{V}_{x,j}^\ell)(w^2-\bar{w}^2)\big)\big| \\
        & \le \sqrt{\tilde{\nu}_{x,j}^\ell \big( (V_{x,j}^\ell - \bar{V}_{x,j}^\ell)^2 (w+\bar{w})^2\big)}\sqrt{\tilde{\nu}_{x,j}^\ell\big( (w-\bar{w})^2\big)} \\
        & \le 4\| V_{x,j}^\ell\|_\infty \sqrt{\tilde{\nu}_{x,j}^\ell\big( (w-\bar{w})^2\big)}.
    \end{align*}
    Notice that $\| V_{x,j}^\ell\|_\infty \le C$ for some constant $C=C(\kappa )>0$. Moreover, the function $w-\bar{w}$ has mean zero under a homogeneous product measure for any value of the parameters $(\rho ,\mathcal{E})\in\mathcal{T}_\kappa$. Therefore, we can apply the spectral gap inequality \cite[Proposition 22]{da_cunha_stationary_2026} to get that the variance on the right-hand side is bounded above by $C\ell^2 \mathfrak{D}_{x,j}^\ell (g;\tilde{\nu}_{x,j}^\ell)$ for some universal constant~$C(\kappa )>0$. Therefore, we have proved that the supremum in \eqref{eq:supremum_estimate_two_blocks_homogeneous} is bounded above by
    \begin{equation*}
        \tilde{\nu}_{x,j}^\ell (\bar{V}_{x,j}^\ell g) + \sup_{g} \left\{ C'\ell \sqrt{\mathfrak{D}_{x,j}^\ell (g;\tilde{\nu}_{x,j}^\ell)} - \frac{e^{-2\delta (N,\varepsilon ,\ell)}}{4C\gamma\varepsilon}\mathfrak{D}_{x,j}^\ell(g;\tilde{\nu}_{x,j}^\ell)\right\}.
    \end{equation*}
    This latter supremum is bounded above by
    \begin{equation*}
        \sup_{a\ge 0} \left\{ C'\ell\sqrt{a} - \frac{e^{-2\delta (N,\varepsilon ,\ell)}}{4C\gamma\varepsilon}a\right\} = C''\ell^2\gamma\varepsilon e^{2\delta (N,\varepsilon ,\ell)}
    \end{equation*}
    which vanishes as $N\to\infty$, $\varepsilon\to 0$ and $\ell\to\infty$ uniformly over $x\in\Lambda_N$ and $|j|\le p/2$. 

    \medskip

    To conclude, we only need to estimate the first term $\tilde{\nu}_{x,j}^\ell (\bar{V}_{x,j}^\ell g)$, and as $g$ integrates to one, we have
    \begin{equation*}
        \tilde{\nu}_{x,j}^\ell (\bar{V}_{x,j}^\ell g) \le \| \bar{V}_{x,j}^\ell\|_\infty
    \end{equation*}
    so it is enough to estimate the sup-norm on the right-hand side. Under the measure $\tilde{\nu}_{x,j}^\ell$, all the coordinates~$\mathbf{Q}_y$ for $y\in \Lambda_{x,j}^\ell$ are independent and identically distributed (i.i.d.) with law $\nu_{\bar{\mathbf{q}}_x}$, we can use the forthcoming \cref{lemma:conditional_expectation} proved at the end of this section to get that
    \begin{equation*}
        \| \bar{V}_{x,j}^\ell \|_\infty \le \frac{C}{\sqrt{\ell}}
    \end{equation*}
    for some constant $C>0$ that depends only on $\kappa$. 

    \medskip

    If we gather all the previous estimates, we have proved that for any $\gamma >0$, the expectation in the statement satisfies the bound
    \begin{equation*}
        \E_{\mu^N}\left[\int_0^t |\mathbf{Q}_x^\ell (s)-\mathbf{Q}_x^{\varepsilon N}(s)\diff s\right] \le C\left(\frac{1}{\gamma} + \frac{\ell}{\varepsilon N} + \ell \left( \varepsilon +\frac{\ell}{N}\right) + \ell^2 \gamma\varepsilon e^{2\ell (\varepsilon +\ell /N)} + \frac{1}{\sqrt{\ell}}\right)
    \end{equation*}
    uniformly over $x\in\Lambda_N$. Taking the supremum over $x\in\Lambda_N$, then letting $N\to\infty$, $\varepsilon\to 0$ and $\ell\to\infty$, and finally letting $\gamma\to\infty$, the proof of \cref{lemma:two_blocks} is complete.
\end{proof}

During the proof of the two-blocks estimate, we have used the following lemma that allows to perform a long-range particle jump or energy transfer at a cost that is proportional to the distance between the two sites. 

\begin{lemma}[Moving particle-energy lemma]\label{lemma:moving_particle}
    There exists a constant $C=C(\mathbf{q}_\ell ,\mathbf{q}_r,\kappa )>0$ such that for any $x,y\in\Lambda_N$ and any density function $g$ with respect to $\bar{\nu}^N$, we have
    \begin{equation*}
        \bar{J}_{x,y}(g;\bar{\nu}^N) \le C|y-x|\mathfrak{D}_N(g;\bar{\nu}^N)
    \end{equation*}
    where $\bar{J}_{x,y}$ is the extension to the whole space of the Dirichlet form defined in \eqref{eq:def_J_x_yj}. 
\end{lemma}

\begin{proof}
    The idea of the proof is similar to the one of the moving-particle and moving-energy \cite[Lemmas 27 and 28]{da_cunha_stationary_2026}, so we only sketch the proof. The goal is to find a deterministic path of jumps and energy transfers allowed by the dynamics of the EPE, in order to perform a long-range particle jump or energy transfer from site $x$ to site $y$. 

    \medskip

    If the values of two neighbouring sites in the configuration are $0$ and $m\ge 1$, then their values can be exchanged by a one-step particle jump. If the values of two neighbouring sites are $m\ge 1$ and $n\ge 1$, then their values can be exchanged by performing $|m-n|\le\kappa -1$ one-step energy transfers from the larger to the smaller. Therefore, every transposition of the values of two neighbouring sites can be performed in at most $\kappa$ steps allowed by the EPE dynamics.

    \medskip

    If we want to perform a long-range particle jump from $x$ to $y$, \textit{i.e.}~to change a configuration $\eta$ into $\eta^{x,y}$ then we can perform a sequence of neighbouring transpositions to bring $\eta_x$ to site $y$, and then bring $\eta_y$ back to site $x$. This can be done in at most $2|y-x|$ neighbouring transpositions, and each of them can be performed in at most $\kappa$ steps allowed by the EPE dynamics.

    \medskip

    For a long-range energy transfer from $x$ to $y$, \textit{i.e.}~if we want to change $\eta$ into $\eta^{x\to y}$, then we can bring $\eta_x$ to site $y-1$ by a sequence of neighbouring transpositions, then perform a one-step energy transfer from $y-1$ to $y$, and then undo the sequence of neighbouring transpositions to bring $\eta_x-1$ back to site $x$. This can be done in at most $2|y-x|-2$ neighbouring transpositions, plus one energy transfer, and each of them can be performed in at most $\kappa$ steps allowed by the EPE dynamics.

    \medskip

    More precisely, this means that for any of these transformations $\eta\mapsto T\eta$, we can find a sequence $(\eta^{(k)})_{0\le k \le n}$ such that $\eta^{(0)}=\eta$, $\eta^{(n)}=T\eta$, and for any $0\le k \le n(x,y)-1$, there exists a site $z_k$ such that either $\eta^{(k+1)} = (\eta^{(k)})^{z_k,z_k+1}$ with $c_{z_k,z_k+1}^p(\eta^{(k)})>0$, or $\eta^{(k+1)} = (\eta^{(k)})^{z_k\to z_k+1}$ with $c_{z_k\to z_k+1}^e(\eta^{(k)})>0$. Moreover $n=n(x,y)\le 2\kappa|y-x|$. Using the Cauchy-Schwarz inequality to write that
    \begin{equation*}
        [\sqrt{g(T\eta )}-\sqrt{g(\eta )}]^2 \le n \sum_{i=0}^{n-1} \Big[ \sqrt{g(\eta^{(i+1)})}-\sqrt{g(\eta^{(i)})}\Big]^2,
    \end{equation*}
    together with the quasi-reversibility relations in \cref{lemma:quasi_reversibility} satisfied by $\bar{\nu}^N$, we can conclude the proof of \cref{lemma:moving_particle}.
\end{proof}

Finally, we also have used the following lemma that allows to estimate the conditional expectation of the difference of two empirical averages knowing their sum.

\begin{lemma}\label{lemma:conditional_expectation}
    Let $(X_i)_{i\ge 1}$ and $(Y_i)_{i\ge 1}$ be two independent sequences of i.i.d. random variables taking values in a compact subset $K$ of $\R^2$. Define the empirical averages
    \begin{equation*}
        \bar{X}_n = \frac1n\sum_{i=1}^n X_i,\qquad \bar{Y}_n = \frac1n\sum_{i=1}^n Y_i.
    \end{equation*}
    Then, for any $n\ge 1$, we have
    \begin{equation*}
        \big\| \E \big[|\bar{X}_n-\bar{Y}_n| \big| \bar{X}_n+\bar{Y}_n\big]\big\|_\infty \le \frac{2\mathrm{diam}(K)}{\sqrt{2n-1}}.
    \end{equation*}
\end{lemma}

\begin{proof}
    Fix $n\ge 1$. We define $(Z_i)_{1\le i\le 2n}$ by $Z_i=X_i$ and $Z_{n+i}=Y_i$ for $1\le i\le n$. If we set~$D_n = \bar{X}_n - \bar{Y}_n$, then we have 
    \begin{equation*}
        D_n = \frac1n\sum_{i=1}^n \varepsilon_iZ_i
    \end{equation*}
    where $\varepsilon_i=1$ and $\varepsilon_{n+i}=-1$ for $1\le i\le n$. Define also the empirical average
    \begin{equation*}
        \bar{Z}_n = \frac{1}{2n}\sum_{i=1}^{2n} Z_i = \frac{\bar{X}_n+\bar{Y}_n}{2}.
    \end{equation*}
    Therefore, in order to estimate the conditional expectation in the statement, we only have to estimate $\E \big[ |D_n| \big| \bar{Z}_n\big]$. Thanks to the conditional Cauchy-Schwarz inequality, it satisfies
    \begin{equation}\label{eq:conditional_Cauchy_Schwarz}
        \E \big[ |D_n| \big| \bar{Z}_n\big] \le \sqrt{\E \big[ |D_n|^2 \big| \bar{Z}_n\big]}.
    \end{equation}
    Let us compute the conditional second moment. For this, set $W_i = Z_i - \bar{Z}_n$ for $1\le i\le 2n$, and also set 
    \begin{equation*}
        S_n = \sum_{i=1}^{2n} |W_i|^2.
    \end{equation*}
    By exchangeability of the random variables $(Z_i)_{1\le i\le 2n}$, we have that
    \begin{equation*}
        \E [W_i|\bar{Z}_n]=0 \qquad\mbox{ and }\qquad \E \big[ |W_i|^2 \big| \bar{Z}_n\big] = \frac{1}{2n}\E [S_n|\bar{Z}_n].
    \end{equation*}
    Again by exchangeability, the conditional expectations $\E [W_i^\dagger W_j|\bar{Z}_n]$ are all equal for $i\neq j$, and we can compute them using the identity
    \begin{equation*}
        0 = \bigg| \sum_{i=1}^{2n} W_i\bigg|^2 = S_n + \sum_{i\neq j}W_i^\dagger W_j,
    \end{equation*}
    which by conditioning on $\bar{Z}_n$ gives that
    \begin{equation*}
        \E [W_i^\dagger W_j|\bar{Z}_n] = -\frac{\E [S_n|\bar{Z}_n]}{2n(2n-1)}.
    \end{equation*}
    We are now ready to compute the conditional second moment of $D_n$. It writes
    \begin{align*}
        \E \big[ |D_n|^2\big| \bar{Z}_n\big] & = \frac{1}{n^2}\sum_{i=1}^{2n} \E \big[ |Z_i|^2\big| \bar{Z}_n\big] + \frac{1}{n^2}\sum_{i\neq j} \varepsilon_i\varepsilon_j \E [Z_i^\dagger Z_j|\bar{Z}_n] \\
        & = \frac{1}{n^2}\sum_{i=1}^{2n} \Big( \E \big[ |W_i|^2\big|\bar{Z}_n\big] + \bar{Z}_n^2\Big) + \frac{1}{n^2}\bigg(\sum_{i\neq j}\varepsilon_i\varepsilon_j\bigg)\left( \E [W_i^\dagger W_j|\bar{Z}_n] + \bar{Z}_n^2\right) \\
        & = \frac2n \E \big[ |W_1|^2\big|\bar{Z}_n\big] - \frac2n \E [W_1^\dagger W_2|\bar{Z}_n] 
    \end{align*}
    where we have used that $\sum_{i\neq j}\varepsilon_i\varepsilon_j = -2n$, and the fact that $W_i$ and $\bar{Z}_n$ are orthogonal. Therefore, using the identities we have computed above, we get that
    \begin{equation*}
        \E \big[ |D_n|^2\big| \bar{Z}_n\big] = \frac{2 \E [S_n|\bar{Z}_n]}{n(2n-1)}.
    \end{equation*}
    But notice that $S_n \le 2n\,\mathrm{diam}(K)^2$ almost surely, so we also have that $\E [S_n|\bar{Z}_n] \le 2n\,\mathrm{diam}(K)^2$ almost surely. If we plug this estimate into \eqref{eq:conditional_Cauchy_Schwarz}, we get the desired result.
\end{proof}

\appendix
\crefalias{section}{appendix}
\section{Generalized Ornstein-Uhlenbeck process}
\label{sec:appendix_OU}

In this appendix, we give a precise definition of the generalized Ornstein-Uhlenbeck process that appears in the statement of \cref{thm:dynamical_fluctuations}, following the theory of martingale problems initiated in \cite{holley_generalized_1978}, and further developed in \cite{kipnis_scaling_1999,bernardin_equilibrium_2022}. Here, we adapt the definition to Ornstein-Uhlenbeck processes with time-dependent coefficients.

\begin{definition}[Generalized Ornstein-Uhlenbeck process]\label{defin:OU}
    Let $\mathcal{C}$ be a topological vector space, let $\mathcal{A}:\mathcal{C}\longrightarrow\mathcal{C}$ be a linear operator letting $\mathcal{C}$ invariant, and let $c : [0,T]\times \mathcal{C}\longrightarrow \R_+$ be a non-negative time-dependent continuous functional satisfying $c_t(\lambda G) = |\lambda | c_t(G)$ for all $\lambda\in\R$, all~$G\in\mathcal{C}$ and all $t\in [0,T]$. Assume also that
    \begin{equation*}
        \forall G\in\mathcal{C},\qquad \int_0^T c_t(G)^2\diff t <\infty.
    \end{equation*}
    Let $\mathcal{C}'$ be the topological dual of $\mathcal{C}$ endowed with the weak-$*$ topology. We say that the $\mathcal{C}'$-valued process~$(Y_t)_{t\in [0,T]}$ with continuous trajectories is a solution to the Ornstein-Uhlenbeck martingale problem $\mathrm{OU}(\mathcal{C},\mathcal{A},c)$ on the time interval $[0,T]$ with (random) initial condition $y_0$ if:
    \begin{enumerate}[label=(\roman*)]
        \item $Y_0=y_0$ in law ;
        \item for any $G\in\mathcal{C}$, the process
        \begin{equation*}
            M_t(G) = Y_t(G) - Y_0(G) - \int_0^t Y_s(\mathcal{A}G)\diff s
        \end{equation*}
        is a mean-zero martingale with respect to the natural filtration of $(Y_t)_{t\in [0,T]}$ and with quadratic variation given by
        \begin{equation*}
            \langle M(G)\rangle_t = \int_0^t c_s(G)^2\diff s.
        \end{equation*}
    \end{enumerate}
\end{definition}

Then, we have the following result which states uniqueness in law of the solution to the Ornstein-Uhlenbeck martingale problem.

\begin{proposition}[Uniqueness of solutions]\label{prop:uniqueness_OU}
    Assume that there exists a semigroup $(P_t)_{t\ge 0}$ on~$\mathcal{C}$ associated to the operator~$\mathcal{A}:\mathcal{C}\longrightarrow\mathcal{C}$ in the sense that
    \begin{equation*}
        P_{t+\varepsilon}G-P_tG = \varepsilon\mathcal{A}P_tG + o_t(\varepsilon )
    \end{equation*}
    where $\varepsilon^{-1}o_t(\varepsilon )$ goes to zero, as $\varepsilon$ goes to zero, in $\mathcal{C}$ uniformly on compact time intervals. Assume also that for all $0\le s\le t\le T$ and all $G\in\mathcal{C}$, we have
    \begin{equation}\label{eq:condition_uniqueness_OU}
        \int_s^t c_r(P_{t-r}G)^2\diff r <\infty.
    \end{equation}
    Then, there exists at most one distribution for the solution to the Ornstein-Uhlenbeck martingale problem $\mathrm{OU}(\mathcal{C},\mathcal{A},c)$ on the time interval $[0,T]$ with (random) initial condition $y_0$.
\end{proposition}

\begin{proof}
    We only sketch the proof of this result. By adapting the proof of \cite[Lemma 7.1]{bernardin_equilibrium_2022}, one gets that for any $S\le T$, the process defined for any $t\in [0,S]$ and any $G\in\mathcal{C}$ by
    \begin{equation}
        Z_t^S(G) = \exp \left( \mathrm{i}Y_t(P_{S-t}G) + \frac12\int_0^t c_r(P_{S-r}G)^2\diff r\right)
    \end{equation}
    is a continuous martingale with respect to the natural filtration $(\mathcal{F}_t)_{t\in [0,T]}$ of $(Y_t)_{t\in [0,T]}$. In particular, for $0\le s\le t\le S\le T$, the equality $\E \big[ Z_t^S(G)\big|\mathcal{F}_s\big] = Z_s^S(G)$ directly implies that 
    \begin{equation*}
        \E \Big[ \exp\big( \mathrm{i}Y_t(P_{S-t}G))\Big|\mathcal{F}_s\Big] = \exp\left( \mathrm{i}Y_s(P_{S-s}G) - \frac12\int_s^t c_r(P_{S-r}G)^2\diff r\right).
    \end{equation*}
    Setting $S=t$ and replacing $G$ by $\lambda G$ for any $\lambda\in\R$, we obtain that
    \begin{equation*}
        \E \Big[ \exp\big( \mathrm{i}\lambda Y_t(G))\Big|\mathcal{F}_s\Big] = \exp\left( \mathrm{i}\lambda Y_s(P_{t-s}G) - \frac{\lambda^2}{2}\int_s^t c_r(P_{t-r}G)^2\diff r\right) .
    \end{equation*}
    This equality shows that conditioning to $\mathcal{F}_s$, the variable $Y_t(G)$ is a Gaussian with mean~$Y_s(P_{t-s}G)$ and variance $\int_s^t c_r(P_{t-r}G)^2\diff r$. Since the distribution at initial time is given, this is enough to characterize all the finite-dimensional distributions of the process $(Y_t)_{t\in [0,T]}$, and therefore its law. 
\end{proof}

We can apply this result to get that the Ornstein-Uhlenbeck problem $\mathrm{OU}(\mathcal{S}_\theta ,D\Delta_\theta ,\|\nabla\cdot\|_{\theta ,\mathbf{q}_t})$ appearing in the statement of \cref{thm:dynamical_fluctuations} has a unique solution in law, which is therefore the limit of the sequence of fluctuation fields $(\mathcal{Y}^N)_{N\ge 1}$.

\section{Boundary matrices of the Robin stationary fluctuations}
\label{appendix:robin_boundary_matrices}

In this appendix, we compute explicitly the boundary matrices $\Xi_\ell$ and $\Xi_r$ appearing in the statement of \cref{thm:stationary_fluctuations} in the Robin case $\theta =1$. During the proof, we have seen that these matrices are respectively given by
\begin{align*}
    & \Xi_\ell \coloneq \Sigma_\ell (\mathbf{q}_\mathrm{ss}(0)) - 2D\chi (\mathbf{q}_\mathrm{ss}(0)) + D\partial_u(\chi (\mathbf{q}_\mathrm{ss}))(0) ,   \\
    &\Xi_r \coloneq \Sigma_r (\mathbf{q}_\mathrm{ss}(1))  - 2D\chi (\mathbf{q}_\mathrm{ss}(1)) - D\partial_u(\chi (\mathbf{q}_\mathrm{ss}))(1) , 
\end{align*}
where $\chi$ is the compressibility matrix given in \cref{defin:compressibility_matrix}, $\Sigma_\ell ,\Sigma_r$ are the boundary matrices defined in \cref{defin:boundary_matrices} and $\mathbf{q}_\mathrm{ss}$ is the stationary profile \eqref{eq:stationary_solution}. We will compute explicitly the matrix $\Xi_\ell$, but the computation of $\Xi_r$ is similar. 

\medskip 
For simplicity, we denote $\mathbf{q}_\mathrm{ss}(0) = (\rho,\mathcal{E})$. We compute each coefficient of matrix. The first coefficient, the one corresponding to the particle-particle correlation, is given by
\begin{align*}
   (\Xi_\ell)_{1,1}(\rho , \mathcal{E}) & = (\Sigma_\ell)_{1,1}(\rho ,\mathcal{E}) - 2D\chi_{1,1}(\rho ,\mathcal{E}) + D\mathbf{q}_\mathrm{ss}'(0)^\dagger \nabla\chi_{1,1}(\rho ,\mathcal{E}) \\
   & = D\big[(1-\rho )\rho_\ell + (1-\rho_\ell )\rho \big] -2D\rho (1-\rho ) + D(\rho -\rho_\ell )(1-2\rho )\\
   & = 0
\end{align*}
where we have used the fact that $\mathbf{q}_\mathrm{ss}'(0) = \mathbf{q}_\mathrm{ss}(0)-\mathbf{q}_\ell$. The second coefficient, corresponding to the particle-energy correlation, is given by
\begin{align*}
   (\Xi_\ell)_{1,2}(\rho , \mathcal{E}) & = (\Sigma_\ell)_{1,2}(\rho ,\mathcal{E}) - 2D\chi_{1,2}(\rho ,\mathcal{E}) + D\mathbf{q}_\mathrm{ss}'(0)^\dagger \nabla\chi_{1,2}(\rho ,\mathcal{E}) \\
   & = D\big[(1-\rho )\mathcal{E}_\ell + (1-\rho_\ell )\mathcal{E} \big] -2D\mathcal{E}(1-\rho ) - D(\rho -\rho_\ell )\mathcal{E} + D(\mathcal{E}-\mathcal{E}_\ell)(1-\rho )\\
   & = 0.
\end{align*}
By symmetry, we also have $(\Xi_\ell)_{2,1}(\rho , \mathcal{E}) = 0$. Let us turn to the last coefficient, corresponding to the energy-energy correlation. For this, we introduce the ratios $\lambda = \frac{\mathcal{E}}{\rho}$ and $\lambda_\ell = \frac{\mathcal{E}_\ell}{\rho_\ell}$, and the function 
\begin{equation*}
    \varphi (\lambda )=\lambda^2+\lambda -1 - \frac{(\lambda -1)^2}{D}
\end{equation*}
so that $m_2(\rho ,\mathcal{E}) = \rho\varphi (\lambda )$ where we recall that $m_2$ is the second moment of $\eta_x$ under $\nu_{\rho ,\mathcal{E}}$. Notice that we have the identity 
\begin{equation*}
    \chi_{2,2}(\rho ,\mathcal{E}) = m_2(\rho ,\mathcal{E})-\mathcal{E}^2 = \rho\varphi (\lambda ) - \rho^2\lambda^2.
\end{equation*}
Then, the last coefficient of the matrix $\Xi_\ell$ is given by
\begin{align*}
    (\Xi_\ell)_{2,2}(\rho , \mathcal{E}) & = (\Sigma_\ell)_{2,2}(\rho ,\mathcal{E}) - 2D\chi_{2,2}(\rho ,\mathcal{E}) + D\mathbf{q}_\mathrm{ss}'(0)^\dagger \nabla\chi_{2,2}(\rho ,\mathcal{E}) \\
    & = D\big[(1-\rho_\ell)\rho\varphi (\lambda ) +(1-\rho )\rho_\ell \varphi (\lambda_\ell )\big] + \rho\rho_\ell \big[ (\lambda_\ell -1)(\kappa -\lambda ) + (\lambda -1)(\kappa -\lambda_\ell )\big] \\
    & \qquad -2D\rho\varphi (\lambda ) + 2D\rho^2\lambda^2 + D(\rho -\rho_\ell )(\varphi (\lambda) - 2\rho\lambda^2) + D\rho_\ell (\lambda -\lambda_\ell )(\varphi '(\lambda )- 2\rho\lambda )
\end{align*}
Plugging the definition of $\varphi$ and simplifying, it reduces to
\begin{equation*}
    (\Xi_\ell)_{2,2}(\rho , \mathcal{E}) = (D-1)(1-\rho )\rho_\ell (\lambda -\lambda_\ell )^2
\end{equation*}
If we insert the value of $(\rho ,\mathcal{E})$ given by the stationary profile $\mathbf{q}_\mathrm{ss}(0)$, namely
\begin{equation*}
    \rho = \frac{2\rho_\ell +\rho_r}{3} \qquad\mbox{ and }\qquad \mathcal{E} = \frac{2\mathcal{E}_\ell +\mathcal{E}_r}{3},
\end{equation*}
we get exactly the expression of $(\Xi_\ell)_{2,2}$ given in \eqref{eq:Xil}.

\medskip

We conclude having a word on the case of the SSEP with slow boundaries ($\theta =1$) for which the stationary fluctuations were investigated in \cite{franco_non-equilibrium_2019}. The stationary fluctuations of this model with a single conservation law must be exactly the ones obtained for the particle density field in the EPE model. Let us check this. If we transpose our notations to the SSEP studied in \cite{franco_non-equilibrium_2019}, the parameters are given by
\begin{equation*}
    \chi (\rho )=\rho (1-\rho ),\qquad D=1,\qquad \Sigma_\ell = (1-\alpha )\rho + \alpha (1-\rho)
\end{equation*}
if $\alpha$ is the density of the left reservoir. In this case, we have that the left boundary term appearing in the stationary fluctuations is given by
\begin{equation*}
    \Xi_\ell (\rho )= \alpha + (1-2\alpha )\rho - 2\rho (1-\rho )+ (\rho -\alpha )(1-2\rho )=0.
\end{equation*}

This contradicts the result of stationary fluctuations obtained in \cite[Theorem 2.8]{franco_non-equilibrium_2019}, where non-zero boundary contributions appear. This is due to an apparent sign error during the proof. Indeed, the boundary coefficient $\Sigma_\ell (\rho )= \alpha  +(1-2\alpha )\rho$ is correctly written in \cite[Equation 2.16]{franco_non-equilibrium_2019}, and it is exactly the expected boundary flip rate. Meanwhile, the proof starts with  \cite[Equation 6.1]{franco_non-equilibrium_2019} where the corresponding boundary coefficient is written as $\Sigma_\ell (\rho ) = \alpha - (1-2\alpha )\rho$. We conclude that the additional boundary terms in \cite[Theorem 2.8]{franco_non-equilibrium_2019} arise from this sign discrepancy, and should be absent, as we have shown in the present work.

\section*{Declarations}
\subsection*{Competing interests}
The author has no competing interests to declare that are relevant to the content of this article.
\subsection*{Data availability}
No datasets were generated or analysed during the current study.
\section*{Acknowledgments}
The author warmly thanks Makiko Sasada for bringing this research problem to his attention. The author also wants to thank Marielle Simon and Clément Erignoux for valuable discussions in the making of this work. In particular Clément Erignoux for his reading of the manuscript and his suggestions to improve it. This project has received financial support from the CNRS through the MITI interdisciplinary programs .

\bibliography{biblio}
\bibliographystyle{plain}

\end{document}